\documentclass[11pt]{article}

\usepackage[latin1]{inputenc} 
\usepackage{amsthm,amsmath,amssymb} 
\usepackage{graphicx}
\usepackage{color}
\usepackage{verbatim}

\usepackage{enumitem} 
\setlist[enumerate]{topsep=0pt,itemsep=-1ex,partopsep=1ex,parsep=1ex}

\theoremstyle{plain}
\newtheorem{theo}{Theorem}[section]

\newtheorem{lemma}[theo]{Lemma}

\newtheorem{conj}[theo]{Conjecture}

\theoremstyle{definition}
\newtheorem{dfn}[theo]{Definition}
\newtheorem{rem}[theo]{Remark}

\newcommand{\mc}[1]{\mathcal{#1}}
\newcommand{\mb}[1]{\mathbb{#1}}
\newcommand{\nib}[1]{\noindent {\bf #1}}

\newcommand{\brac}[1]{\left( #1 \right)}

\newcommand{\bsize}[1]{\left| #1 \right|}

\newcommand{\bfl}[1]{\left\lfloor #1 \right\rfloor}
\newcommand{\bcl}[1]{\left\lceil #1 \right\rceil}

\newcommand{\vc}[1]{\mathbf{#1}}

\newcommand{\sub}{\subset}

\newcommand{\Lra}{\Leftrightarrow}
\newcommand{\Ra}{\Rightarrow}
\newcommand{\sm}{\setminus}
\newcommand{\ov}{\overline}

\newcommand{\wt}{\widetilde}
\newcommand{\eps}{\varepsilon}
\newcommand{\es}{\emptyset}

\newcommand{\aA}{\alpha}
\newcommand{\bB}{\beta}
\newcommand{\gG}{\gamma}
\newcommand{\dD}{\delta}

\newcommand{\kK}{\kappa}
\newcommand{\zZ}{\zeta}
\newcommand{\lL}{\lambda}
\newcommand{\tT}{\theta}

\newcommand{\GG}{\Gamma}
\newcommand{\DD}{\Delta}
\newcommand{\OO}{\Omega}

\newcommand{\LL}{\Lambda}

\newcommand{\blA}{\boldsymbol\alpha}
\newcommand{\blB}{\boldsymbol\beta}
\newcommand{\blL}{\boldsymbol\lambda}
\newcommand{\blS}{\boldsymbol\sigma}
\newcommand{\blP}{\boldsymbol\pi}

\def\qed{\hfill $\Box$}

\title{Forbidden vector-valued intersections}

\author{Peter Keevash\thanks{Mathematical Institute, University of Oxford, Oxford, UK. 
E-mail: keevash@maths.ox.ac.uk.
\newline \hspace*{1.8em}Research supported in part by ERC Consolidator Grant 647678.}
\and Eoin Long\thanks{Mathematical Institute, University of Oxford, Oxford, UK. E-mail: long@maths.ox.ac.uk.
\newline \hspace*{1.8em}Supported in part by ERC Starter Grant 633509.}
}

\begin{document}

\maketitle

\begin{abstract}
We solve a generalised form of a conjecture of Kalai 
motivated by attempts to improve the bounds for Borsuk's problem. 
The conjecture can be roughly understood as asking for 
an analogue of the Frankl-R\"odl forbidden intersection theorem 
in which set intersections are vector-valued. 
We discover that the vector world is richer in surprising ways: 
in particular, Kalai's conjecture is false, but
we prove a corrected statement that is essentially best possible,
and applies to a considerably more general setting.
Our methods include the use of maximum entropy measures, 
VC-dimension, Dependent Random Choice 
and a new correlation inequality for product measures.
\end{abstract}

\section{Introduction}

Intersection theorems have been a central topic of Extremal Combinatorics since the seminal paper of Erd\H os, Ko and Rado \cite{Erdos-Ko-Rado}, and the area has grown into a vast body of research (see \cite{Babai-Frankl}, \cite{com} or \cite{jukna} for an overview). The Frankl-R\"odl forbidden intersection theorem is a fundamental result of this type, which has had a wide range of applications to different areas
of mathematics, including discrete geometry \cite{FRGeomRams}, 
communication complexity \cite{Sgall} and quantum computing \cite{BCW}.

To state their result we introduce the following notation. Let $[n] = \{1,\ldots , n\}$ and let $\tbinom{[n]}{k} = \{ A \subset [n]: |A|=k\}$.
For $\mc{A} \subset \tbinom{[n]}{k}$ and $t \in [n]$
let $\mc{A} \times_t \mc{A}$ be the set of
all $(A,B) \in \mc{A} \times \mc{A}$ with $|A \cap B|=t$.
Note that $\tbinom{[n]}{k} \times_t \tbinom{[n]}{k}$ is 
non-empty if and only if $\max (2k-n, 0) \leq t \leq k$.
Frankl and R\"odl proved the following `supersaturation theorem',
showing that if $t$ is bounded away from these extremes 
and $\mc{A}$ is `exponentially dense' in $\tbinom{[n]}{k}$
then $\mc{A} \times_t \mc{A}$ is  `exponentially dense' 
in $\tbinom{[n]}{k} \times_t \tbinom{[n]}{k}$.

\begin{theo}[Frankl--R\"odl \cite {FrRo}]
	\label{FranklRodl}
Let\footnote{Our notation `$n^{-1}\ll \dD \ll \eps $' here means that for any $0<\eps<1$
there is $\dD_0>0$ such that for any $0<\dD<\dD_0$
there is $n_0$ such that for $n \ge n_0$
the following statement holds.}
$0 < n^{-1} \ll \dD \ll \eps < 1$ and 
$\max (2k-n, 0) + \eps n \leq t \leq k - \eps n$.
Suppose $\mc{A} \subset \tbinom{[n]}{k}$ with $|\mc{A} | \geq (1-\delta )^n \tbinom{n}{k}$.
Then $|\mc{A} \times_t \mc{A} | \ge
 (1-\eps)^n \bsize{ \tbinom{[n]}{k} \times_t \tbinom{[n]}{k} } $.
\end{theo}

In a recent survey on the Borsuk problem,
Kalai \cite{Kalai} remarked that the Frankl--R\"odl theorem
can be used to give a counterexample to the Borsuk conjecture
(the Frankl--Wilson intersection theorem \cite{FrWil}
was used in Kahn and Kalai's celebrated counterexample \cite{KahnKalai}),
and suggested that improved bounds might follow
from a suitably generalised Frankl--R\"odl theorem.
He proposed the following supersaturation conjecture 
as a possible step in this direction,
in which one measures a set by 
its size $|A| = \sum_{i \in A} 1$
and its sum $\sum A = \sum_{i \in A} i$.
Let $[n]_{k,s}$ be the set of $A \subset [n]$
with $|A|=k$ and $\sum A = s$.
For $\mc{A} \subset [n]_{k,s}$ write 
\[\mc{A} \times_{(t,w)} \mc{A} = \big \{ (A,B) \in \mc{A} \times \mc{A}: 
A \neq B \mbox{ with } |A \cap B|=t \mbox{ and } \sum( A \cap B)  = w \big \}.\]

\begin{conj}[Kalai] \label{kalai}
Let $0 < n^{-1} \ll \dD \ll \eps, \aA_1, \aA_2, \bB_1, \bB_2 < 1$,
$k = \bfl{ \aA_1 n } $, $s = \bfl{ \aA_2 \tbinom{n}{2} } $, 
$t = \bfl{ \bB_1 n } $ and $w = \bfl{ \bB_2 \tbinom{n}{2} }$.
Suppose $\mc{A} \subset [n]_{k,s}$ with $|\mc{A}| \geq (1-\delta )^n |[n]_{k,s}|$.
Then $|\mc{A} \times_{(t,w)}\mc{A}| \geq \bfl{ (1-\eps )^n 
 \bsize{ [n]_{k,s} \times_{(t,w)} [n]_{k,s} } }$.
\end{conj}

Somewhat surprisingly, this conjecture is false! In fact, although the conjecture holds in a number of natural special cases, it fails quite dramatically in general; for most pairs $(\alpha _1,\alpha _2)$ there is exactly one choice of $(\beta _1, \beta _2)$ for which Conjecture \ref{kalai} holds. Before stating this result, we first remark that
Conjecture \ref{kalai} is only non-trivial when
$[n]_{k,s}$ is exponentially large in $n$
(when $|[n]_{k,s} \times _{(t,w)} [n]_{k,s}| \geq (1-\eps )^{-n}$, requiring $|[n]_{k,s}| \ge (1-\eps)^{-n/2}$).
Defining $\aA_1,\aA_2$ as in Conjecture \ref{kalai},
we can therefore assume that $(\aA_1,\aA_2)$ belongs to
\[\Lambda := \{(x,y): 0<x<1, x^2 < y < 2x - x^2\}.\]
We say that ${\bf g} = (\alpha _1,\alpha _2,\beta _1,\beta _2)$ 
is {\em $(n,\dD,\eps)$-Kalai}
if Conjecture \ref{kalai} holds for ${\bf g}$, i.e.\ any
$\mc{A} \subset [n]_{k,s}$ with 
$|\mc{A}| \geq (1-\delta )^n |[n]_{k,s}|$ satisfies	
$|\mc{A} \times_{(t,w)}\mc{A}| 
\geq \bfl{ (1-\eps )^n 
\bsize{ [n]_{k,s} \times_{(t,w)} [n]_{k,s} } }$.
We will classify the Kalai parameters ${\bf g}$
in terms of the following set $\GG$;
note that the definition of $\GG_1$ uses two functions  
$\beta _1, \beta _2 : \Lambda \to {\mathbb R}$ 
that will be defined in Section \ref{sec:concrete description}.
Let $\Gamma = \bigcup _{i\in [3]} \Gamma _i$, 
where	\begin{align*}
	& \hspace{.2cm} \mbox{\textbf{\emph{(Popular intersections)}}} & \Gamma _{1} 
		= 
	&  \big \{ (\alpha _1,\alpha _2 , \beta _1,\beta _2): 
		(\alpha _1,\alpha _2) \in \Lambda, \mbox{ }
		\alpha _1 \neq \alpha _2 \mbox { and } 
		\beta _i = \beta _i(\alpha _1,\alpha _2) \big \};\\		
	& \hspace{.8cm} \mbox{\emph{\textbf{(Doubly random)}}} & \Gamma _{2} 		
		= &\big \{ (\alpha ,\alpha , \beta ,\beta ): 
		0 < \alpha < 1 \mbox{ \mbox{and} } 
		\max (2\alpha -1 , 0) < \beta < \alpha \big \};\\
	&\mbox{\textbf{\emph{(Uniformly random sets)}}} &\Gamma _{3} 
		= 
	&\big \{ (1/2,1/2 , \beta _1,\beta _2):
		(2\beta _1, 2\beta _2) \in \Lambda \big \}. 	
	\end{align*}

\begin{theo}\label{thm: Kalai}
Supppose ${\bf g} = (\aA_1, \aA_2, \bB_1, \bB_2) \in [0,1]^4$
with $(\aA_1,\aA_2) \in \LL$ and  
$n^{-1} \ll \dD \ll \eps \ll \eps' \ll \aA_1, \aA_2, \bB_1, \bB_2$.
Let ${\bf g}' \in \GG$ minimise $\|{\bf g} - {\bf g}'\|_1$.
\begin{enumerate}
\item If $\|{\bf g} - {\bf g}'\|_1 \le \dD$ 
then ${\bf g}$ is $(n,\dD,\eps)$-Kalai.
\item If $\|{\bf g} - {\bf g}'\|_1 \ge \eps'$
then ${\bf g}$ is not $(n,\dD,\eps)$-Kalai.
\end{enumerate}
\end{theo}

\noindent The labels assigned to the parts of $\GG$ 
correspond to the following interpretations:\vspace{1mm}
\begin{itemize}
	\item \textbf{Popular intersections:} For $(\alpha _1,\alpha _2) \in \Lambda$ with 
	$\alpha _1 \neq \alpha _2$ there is exactly one 
	$(\alpha _1,\alpha _2,\beta _1,\beta _2) \in {\Gamma }_1$. 
	For $n$ large, the value 
	$(\beta _1n, \beta _2\tbinom {n}{2})$ is 
	essentially the most popular intersection between 
	sets in $[n]_{k,s}$, where $(k,s) = (\alpha _1n,\alpha _2\tbinom {n}{2})$.

	\item \textbf{Doubly random:} If $A \subset [n]$ is a uniformly random set of size 
	$k = \alpha n$, the expected size of $\sum A$ 
	is $s = \alpha \tbinom {n}{2} + o(n^2)$. Similarly, if two 
	sets $A$ and $B$ in $[n]_{k,s}$ with $|A\cap B| =\beta n$ 
	are randomly selected then the expected value of 
	$\sum (A \cap B)$ is $\beta \tbinom {n}{2} + o(n^2)$.  
	Theorem \ref{thm: Kalai} for $\Gamma _2$ shows Conjecture \ref{kalai} 
        holds for `random-like $\beta n$
	intersections' between `random-like $\alpha n$ sets', provided 
	$\alpha $ and $\beta $ satisfy the Frankl-R\"odl conditions.

	\item \textbf{Uniformly random sets:} 
        Most sets $A \sub [n]$ have $k = \tfrac {1}{2}n + o(n)$, 
        $s = \tfrac{1}{2}\tbinom {n}{2} + o(n^2)$ and
	$|A \cap [2L]| = L \pm o(n)$ for all $L \leq n/2$. Intersections 
	of type $(t,w) = (\beta _1n, \beta _2 \tbinom {n}{2})$ 
	can only occur between such sets if
         $(2\beta _1 + o(1), 2\beta _2 + o(1)) \in \Lambda$. 
        Theorem \ref{thm: Kalai} for $\Gamma _3$ 
	shows that Conjecture \ref{kalai} is true for 
	$(\alpha _1, \alpha _2) = (1/2,1/2)$ provided 
	this necessary condition is fulfilled.
\end{itemize}\vspace{1.5mm}

Although the bounds from Conjecture \ref{kalai} in general do not hold, 
it is still natural to ask whether we can find {\em any} $(t,w)$-intersection
in such `exponentially dense' subsets ${\cal A} \subset [n]_{k,s}$.
If so, what is the optimal lower bound on $|{\cal A} \times _{(t,w)}{\cal A}|$?
This paper investigates these questions; 
in particular, we give a natural correction to Conjecture \ref{kalai}. 

Our results will apply to the following more general setting 
of vector-valued set `sizes': 
given vectors $\mc{V} = (\vc{v}_i: i \in [n])$ in $\mb{R}^D$,
we define the $\mc{V}$-size of $A \sub [n]$ by 
\[ |A|_{\mc{V}} = \sum_{i \in A} \vc{v}_i. \]
We note that the Frankl-R\"odl theorem concerns
$\mc{V}$-sizes where $D=1$ and all $\vc{v}_i=1$,
and the Kalai conjecture concerns
$\mc{V}$-sizes where $D=2$ and $\vc{v}_i=(1,i)$.

\subsection{Vector-valued intersections} \label{subsec: intro-trichotomy}

In order to prove our forbidden ${\cal V}$-intersection theorem, 
we need to work over a general alphabet, where we associate a vector 
with each possible value of each coordinate, as follows.

\begin{dfn}
Suppose $\vc{v}^i_j \in \mb{Z}^D$ for all $i \in [n]$ and $j \in J$.
We call $\mc{V} = (\vc{v}^i_j)$ an $(n,J)$-array in $\mb{Z}^D$. 
For $\vc{a} \in J^n$ we define
\[\mc{V}(\vc{a}) = \sum_{i \in [n]} \vc{v}^i_{a_i}. \]
For $\mc{A} \sub J^n$ and $\vc{w} \in \mb{Z}^D$ we define 
$ \mc{A}^{\mc{V}}_{\vc{w}} = \{\vc{a} \in \mc{A}: \mc{V}(\vc{a}) = \vc{w}\}$.
\end{dfn}

By identifying subsets of $[n]$ with their characteristic vectors in $\{0,1\}^n$, 
and pairs of subsets of $[n]$ with vectors in $(\{0,1\}\times \{0,1\})^n$, 
this definition extends the definition of $\mc{V}$-size and $\mc{V}$-intersection 
via the following specialisation 
(note that $\mc{V}(A) = |A|_{\mc{V}}$ 
and $\mc{V}_\cap(A,B)=|A \cap B|_{\mc{V}}$).

\begin{dfn}
Suppose $\mc{V} = (\vc{v}_i: i \in [n])$,
where $\vc{v}_i \in \mb{Z}^D$ for all $i \in [n]$.
We also let $\mc{V} = (\vc{v}^i_j)$ denote the $(n,\{0,1\})$-array in $\mb{Z}^D$,
where $\vc{v}^i_1 = \vc{v}_i$ and $\vc{v}^i_0 = 0$.
We let $\mc{V}_\cap = ((\vc{v}_\cap)^i_{j,j'})$ denote the $(n,\{0,1\} \times \{0,1\})$-array 
in $\mb{Z}^D$, where $(\vc{v}_\cap)^i_{1,1} = \vc{v}_i$ and $(\vc{v}_\cap)^i_{j,j'} = 0$ otherwise.
\end{dfn}

We also introduce a class of norms on $\mb{R}^D$
to account for the possibility that different coordinates of vectors 
in $\mc{V}$ may operate at different scales.
In the following definition we think of ${\bf R}$ as a scaling;
e.g.\ for the Kalai vectors $(1,i)$, we take ${\bf R} = (1,n)$.

\begin{dfn}
Suppose $\vc{R}=(R_1,\dots,R_D) \in \mb{R}^D$. We define the $\vc{R}$-norm 
on $\mb{R}^D$ by $\|\vc{v}\|_{\vc{R}} = \max_{d \in [D]} |v_d|/R_d$.
We say that $\mc{V} = (\vc{v}^i_j)$ is $\vc{R}$-bounded 
if all $\|\vc{v}^i_j\|_{\vc{R}} \le 1$.
\end{dfn}

Our ${\cal V}$-intersection theorem requires two properties of the set of vectors ${\cal V}$. 
The first property, roughly speaking, says that any vector in $\mb{Z}^D$
can be efficiently generated by changing the values of coordinates,
and that furthermore this holds even if 
a small set of coordinates are frozen,
so that no coordinate is overly significant.
To see why such a condition is necessary,
suppose that $D=1$ and almost all
coordinates have only even values:
then there are large families where all intersections 
have a fixed parity.

\begin{dfn} \label{robgen}
Let $\mc{V} = (\vc{v}^i_j)$ be an $(n,J)$-array in $\mb{Z}^D$. 
We say that $\mc{V}$ is $\gG$-robustly $(\vc{R},k)$-generating in $\mb{Z}^D$
if for any $\vc{v} \in \mb{Z}^D$ with $\|\vc{v}\|_{\vc{R}} \le 1$
and $T \sub [n]$ with $|T| \le \gG n$ there is $S \sub [n] \sm T$ 
with $|S| \le k$ and $j_i, j'_i  \in J$ for all $i \in S$ such that
$\vc{v} = \sum_{i \in S} (\vc{v}^i_{j_i}-\vc{v}^i_{j'_i})$.
\end{dfn}

Note that if $\mc{V} = (\vc{v}_i: i \in [n])$, considered as an $(n,\{0,1\})$-array,
then Definition \ref{robgen} says that for all such $\vc{v}$ and $T$  
there are disjoint $S,S' \sub [n] \sm T$ with $|S|+|S'| \le k$ 
such that $\vc{v} = \sum_{i \in S} \vc{v}_i - \sum_{i \in S'} \vc{v}_i$.

In particular, the Kalai vectors are $0.1$-robustly $((1,n),7)$-generating in $\mb{Z}^2$. Indeed, for any vector $(0,b)$ with $b\in [ n/2 ]$, there are $n/3 $ disjoint pairs $\{i_1,i_2\}$ with
$(0,b) = (1,i_1) - (1,i_2)$. This implies that for any vector $(0,b)$ with $|b| \leq n$ there are $n/9$ disjoint sets $\{i_1,i_2,j_1,j_2\}$ with $(0,b) = (1,i_1) + (1,i_2) - (1,j_1) - (1,j_1) $. Also, there are $n/6$ disjoint triples $\{i_1,i_2, i_3\}$ with $(1,0) = (1,i_1) + (1,i_2) - (1,i_3)$. Combined, given $T \subset [n]$ with $|T| \leq n/10 < n/9 -3$ and $(a,b) \in \mb{Z}^2$  with $\|(a,b)\|_{\bf R} \leq 1$, there are disjoint $S,S' \sub [n] \sm T$ with $|S|+|S'| \le 7$ with $(a,b) = \sum_{i \in S} \vc{v}_i - \sum_{i \in S'} \vc{v}_i$.

We also make the following `general position' assumption for $\mc{V}$.

\begin{dfn}\label{robust gen} 
Suppose $\mc{V} = (\vc{v}_i)$ is an $(n,\{0,1\})$-array in $\mb{Z}^D$. 
For $I \in \tbinom{[n]}{D}$, let ${\cal V}_I = \{{\bf v}_i: i\in I\}$ 
and say that $I$ is $(\gG,\vc{R})$-generic if 
$|\det({\cal V}_I)| \ge \gG \prod_{d \in [D]} R_d$.
We say that $\mc{V}$ is $\gG'$-robustly $(\gG,\vc{R})$-generic
if for any $X \sub [n]$ with $|X| > \gG' n$, some $I\sub X$
is $(\gG,\vc{R})$-generic for $\mc{V}$.
\end{dfn}

Note that the Kalai vectors are $\gamma $-robustly $(\gamma /2, {\bf R})$-generic for any $\gamma >0$, since if $X \sub [n]$ with $|X| \geq \gG n$ then we can choose $i,i' \in X$ 
with $|i-i'| \geq \gG n - 1 \geq \gamma n/2$, and then $(1,i)$ and $(1,i')$ span a parallelogram
of area $|i-i'| \geq (\gG /2) \cdot 1 \cdot n$.

We are now in a position to state our main theorem. 
It shows that, under the above assumptions on ${\cal V}$, 
there are only two obstructions to a set ${\cal X} = (\{0,1\}^n)^{\cal V}_{\bf z}$ 
satisfying a supersaturation result as in Kalai's conjecture (case $i$): 
either (case $ii$) there is a small set ${\cal B}_{full} \subset {\cal X}$ 
responsible for almost all $\bf w$-intersections in ${\cal X}$,
or (case $iii$) there is a large set ${\cal B}_{empty} \subset {\cal X}$ 
containing no ${\bf w}$-intersections.
Furthermore, in case $ii$ we obtain optimal supersaturation 
relative to ${\cal B}_{full}$.

\begin{theo}\label{thm: trichotomy}
	Let $n^{-1} \ll \delta \ll \gG _1, \gG _1' \ll \gG _2,\gG _2' \ll \eps , D^{-1}, C^{-1}, k^{-1}$ 
	and ${\bf R} \in {\mathbb R}^D$ with $\max _{d} R_d \leq n^C$. 
Suppose ${\cal V} = ({\bf v}_i : i\in [n])$ where each ${\bf v}_i \in {\mathbb Z}^D$ is 
	${\bf R}$-bounded and ${\cal V}$ is $\gG _i'$-robustly $(\gG _i, {\bf R})$-generic 
	and $\gG _i$-robustly $({\bf R},k)$-generating for $i = 1,2$. Let ${\bf z}, {\bf w} \in {\mathbb Z}^D$ with ${\bf z} \neq {\bf w}$ and let ${\cal X} = (\{0,1\}^n)^{\cal V}_{\bf z}$. Then one of the following holds:
		\begin{enumerate}
			\item All ${\cal A} \subset {\cal X}$ with 
			$|{\cal A}| \geq (1-\delta )^n|{\cal X}|$ satisfy 
			$|({\cal A} \times {\cal A})^{{\cal V}_{\cap }}_{\bf w}| \geq 
			(1-\eps )^n |({\cal X} \times {\cal X})^{{\cal V}_\cap}_{\bf w}|$.
			\item There exists ${\cal B}_{full} \subset {\cal X}$ with 
			$|{\cal B}_{full}| \leq (1-\delta )^n|{\cal X}|$ satisfying 
			$$|({\cal X} \times {\cal X})^{{\cal V}_\cap}_{\bf w} \setminus 
			({\cal B}_{full} \times {\cal B}_{full})^{{\cal V}_{\cap}}_{\bf w}| 
			\leq (1-\delta )^n |({\cal X} \times {\cal X})^{{\cal V}_\cap}_{\bf w}|.$$ 	
			
			\item There is ${\cal B}_{empty} \subset {\cal X}$ with 
			$|{\cal B}_{empty}| \geq \lfloor (1-\eps )^n|{\cal X}|\rfloor $ satisfying 
			$({\cal B}_{empty} \times {\cal B}_{empty})^{{\cal V}_{\cap }}_{\bf w} = \emptyset $.
		\end{enumerate}
Furthermore, if $ii$ holds and $iii$ does not then any ${\cal B} \subset {\cal B}_{full}$ with 
			$|{\cal B}| \geq (1-\delta )^n|{\cal B}_{full}|$ satisfies 
			$|({\cal B} \times {\cal B})^{{\cal V}_{\cap }}_{\bf w}|
			\geq (1-\eps )^n|({\cal X} \times {\cal X})^{{\cal V}_{\cap }}_{\bf w}|$.
\end{theo}

\begin{rem} $ $
\begin{enumerate}
\item Theorem \ref{thm: trichotomy} applies to $(t,w)$-intersections in $[n]_{k,s}$,
as we have shown above that its hypotheses hold for the Kalai vectors.
\item As indicated above, cases $ii$ and $iii$ of Theorem \ref{thm: trichotomy} may simultaneously hold
(see counterexample 1 of Section \ref{sec:counterex}). 
\item The assumption that ${\cal V}$ is 
$\gG_1$-robustly $({\bf R},k)$-generating is redundant,
as it is implied by $\gG_2$-robustly $({\bf R},k)$-generating,
but the assumptions of $\gG _i'$-robustly $(\gG _i, {\bf R})$-generic 
for $i = 1,2$ are incomparable, and our proof seems to require
this `multiscale general position'.
\end{enumerate}
\end{rem}

We have highlighted Theorem \ref{thm: trichotomy} as our main result for the sake of giving a clean combinatorial statement. However, we will in fact obtain considerably more general results in two directions, whose precise statements are postponed until later in the paper.

\begin{itemize}
	\item 
Our most general result, Theorem \ref{general}, implies cross-intersection theorems for two or more families and applies to families of vectors over any finite alphabet.
	\item 
Theorem \ref{thm: trichotomy} leaves open the question
of how many ${\bf w}$-intersections are guaranteed 
in large subsets of ${\cal X}$ when case (ii) holds;
this is answered by Theorem \ref{optsupersat}.
\end{itemize}

It is natural to ask under which conditions 
the alternate cases of Theorem \ref{thm: trichotomy} hold.
These conditions are best understood in relation to our proof framework,
so we postpone this discussion to section \ref{subsec: supersaturation},
after we have introduced the two principal components of the proof.

\subsection{A probabilistic forbidden intersection theorem}

A key paradigm of our approach is that $\mc{V}$-intersection theorems
often have equivalent formulations in terms of certain product measures 
(the maximum entropy measures described in the next subsection),
and that the necessary condition for these theorems appears naturally
as a condition on the product measures.
(A similar idea arose in the new proof of the density Hales-Jewett theorem 
developed by the first Polymath project \cite{DHJ}, although in this case
the natural `equal slices' distribution was not a product measure.)

To illustrate this point, we recast the Frankl-R\"odl theorem in such terms.
Again we identify subsets of $[n]$ with their characteristic vectors in $\{0,1\}^n$,
on which we introduce the product measure $\mu_p(\vc{x}) = \prod_{i \in [n]} p_{x_i}$,
where $p_1 = k/n$ and $p_0=1-p_1$. Pairs of subsets are identified
with $\{0,1\}^n \times \{0,1\}^n$, which we can identify with 
$(\{0,1\} \times \{0,1\})^n$, on which we introduce the product measure 
$\mu_q(\vc{x},\vc{x}') = \prod_{i \in [n]} q_{x_i,x'_i}$,
where $q_{1,1} = t/n$, $q_{0,1}=q_{1,0}=(k-t)/n$ and $q_{0,0}=(n-2k+t)/n$.
It follows from our general large deviation principle in the next subsection
(or is easy to see directly in this case) that the hypothesis of Theorem
\ref{FranklRodl} is essentially equivalent to $\mu_p(\mc{A})>(1-\dD)^n$
and the conclusion to $\mu_q(\mc{A} \times_t \mc{A})>(1-\eps)^n$.
Furthermore, the assumption on $t$ can be rephrased
as $q_{j,j'} \ge \eps$ for all $j,j' \in \{0,1\}$,
and this indicates the condition that we need in general.

Let us formalise the above discussion of product measures in a general context. Although we only considered the cases when the `alphabet' $J$ is $\{0,1\}$ 
or $\{0,1\} \times \{0,1\}$, we remark that it is essential 
for our arguments to work with general alphabets,
as the proofs of our results even in the binary case 
rely on reductions that increase the alphabet size.

\begin{dfn}
Suppose $\vc{p}=(p^i_j: i \in [n], j \in J)$ with all $p^i_j \in [0,1]$
and $\sum_{j \in J} p^i_j = 1$ for all $i \in [n]$. The product measure $\mu_{\vc{p}}$ 
on $J^n$ is given, for $\vc{a} \in J^n$, by $\mu_{\vc{p}}(\vc{a}) = \prod_{i \in [n]} p^i_{a_i}$. 

Given an $(n,J)$-array $\mc{V}$ and a measure $\mu $ on $J^n$, we write 
$\mc{V}(\mu ) = \mb{E}_{\vc{a} \sim \mu } \mc{V}(\vc{a})$.

Suppose $\mu_{\vc{q}}$ is a product measure on $(\prod_{s \in S} J_s)^n$,
with $\vc{q}=(q^i_{j_1,\dots,j_s}: i \in [n], j_1 \in J_1, \dots, j_S \in J_S)$.
For $s \in [S]$ the $s$-marginal of $\mu_{\vc{q}}$
is the product measure $\mu_{\vc{p}_s}$ on $J_s^n$ with
$(p_s)^i_j = \sum q^i_{j_1,\dots,j_S}$ for all $i \in [n]$, $j \in J_s$,
where the sum is over all $(j_1,\dots,j_S)$ with $j_s=j$.

We say that $\mu_{\vc{q}}$ has marginals $(\mu_{\vc{p}_s}: s \in S)$. We say that $\mu_{\vc{q}}$ is $\kK$-bounded if all $q^i_{j,j'} \in [\kK,1-\kK]$. Note that if $\mu_{\vc{q}}$ is $\kK$-bounded then so are its marginals.
\end{dfn}

A rough statement of our probabilistic forbidden intersection theorem 
(Theorem \ref{binary} below) is that if $\mc{A}$ has `large measure'
then the set of $\vc{w}$-intersections in $\mc{A}$ has `large measure'.
We will combine this with an equivalence of measures discussed
in the next subsection to deduce our main theorem.
First will highlight two special cases of Theorem \ref{binary}
that have independent interest. The first is the following result,
which ignore the intersection conditions, and is only concerned
with the relationship between the measures of $\mc{A}$
and $\mc{A} \times \mc{A}$; it is a new of correlation inequality
(see Theorem \ref{gencor} for a more general statement
that applies to several families defined over general alphabets).

\begin{theo} \label{cor} 
Let $0 < n^{-1} , \dD \ll \kK, \eps < 1$ 
and $\mu_{\vc{q}}$ be a $\kK$-bounded product measure
on $(\{0,1\} \times \{0,1\})^n$ with both marginals $\mu_{\vc{p}}$.
Suppose $\mc{A} \sub \{0,1\}^n$ with $\mu_{\vc{p}}(\mc{A}) > (1-\dD)^n$.
Then $\mu_{\vc{q}}(\mc{A} \times \mc{A}) > (1-\eps)^n$.
\end{theo}

Next we consider the problem of
finding ${\cal V}$-intersections that are \emph{close} to ${\bf w}$,
which is also natural, and somewhat easier than
finding ${\cal V}$-intersections that are (exactly) ${\bf w}$.
We require some notation. For $r>0$ let $B_{\bf R}({\bf w},r) 
= \{{\bf w}' \in {\mathbb Z}^D: \|{\bf w} - {\bf w}' \|_{\bf R} \leq r\}$. 
For ${\cal A} \subset {\cal P}[n]$ and $L \subset {\mathbb Z}^D$ 
let $({\cal A} \times {\cal A})^{{\cal V}_{\cap }}_{L} 
= \big \{(A,B) \in {\cal A} \times {\cal A}: 
{\cal V}_{\cap }(A,B) \in L \big \}$. 

\begin{theo} \label{binary approx}
Let $0 < n^{-1}, \dD \ll \zZ \ll \kK, \eps \ll D^{-1}$
and $\vc{R} \in {\mathbb Z}^D$. Suppose that
\begin{enumerate}
\item $\mu_{\vc{q}}$ is a $\kK$-bounded product measure 
on $(\{0,1\} \times \{0,1\})^n$ with both marginals $\mu_{\vc{p}}$,
\item $\mc{V} = (\vc{v}_i: i \in [n])$ is an $\vc{R}$-bounded array in $\mb{Z}^D$,
\item $\mc{V}_\cap(\mu_{\vc{q}}) = \vc{w} \in \mb{R}^D$ 
and $L := B_{\bf R}({\bf w}, \zeta n)$.
\end{enumerate}
Then any $\mc{A} \sub \{0,1\}^n$ with $\mu_{\vc{p}}(\mc{A}) > (1-\dD)^n$ satisfies
$\mu_{\vc{q}}((\mc{A} \times \mc{A})^{\mc{V}_\cap}_{L}) > (1-\eps)^n$.
\end{theo}

Theorem \ref{binary approx} naturally fits into the wide literature on forbidden $L$-intersections in extremal set theory (see \cite{Babai-Frankl}, \cite{com} or \cite{jukna}). Here one aims to understand how large certain families of sets can be if all intersections between elements of ${\cal A}$ are restricted to lie in some set $L$. For example, the Erd\H{o}s-Ko-Rado theorem \cite{Erdos-Ko-Rado} can be viewed as an ${L}_0$-intersection theorem for families ${\cal A} \subset \tbinom {n}{k}$, where $L_0 = \{l\in {\mathbb N}: 1\leq l\leq k\}$. Similarly, Katona's $t$-intersection theorem \cite{Katona} can be viewed as an ${L}_{\geq t}$-intersection theorem for families ${\cal A} \subset {\cal P}[n]$, where $L_{\geq t} = \{l\in {\mathbb N}: l\geq t\}$.

Now we state our probabilistic forbidden intersection theorem:
if ${\cal V}$ is robustly generated then
Theorem \ref{binary approx} can be upgraded to find fixed ${\cal V}$-intersections. 

\begin{theo} \label{binary}
Let $0 < n^{-1}, \dD \ll \zZ \ll \kK, \gG, \eps \ll D^{-1}, C^{-1}, k^{-1}$
and $\vc{R} \in {\mathbb Z}^D$ with $\max_d R_d < n^C$. Suppose that
\begin{enumerate}
\item $\mu_{\vc{q}}$ is a $\kK$-bounded product measure 
on $(\{0,1\} \times \{0,1\})^n$ with both marginals $\mu_{\vc{p}}$,
\item $\mc{V} = (\vc{v}_i: i \in [n])$ is $\vc{R}$-bounded
 and $\gG$-robustly $(\vc{R},k)$-generating in $\mb{Z}^D$,
\item $\vc{w} \in \mb{Z}^D$ with $\|\vc{w}-\mc{V}_\cap(\mu_{\vc{q}})\|_{\vc{R}} < \zZ n$.
\end{enumerate}
Then any $\mc{A} \sub \{0,1\}^n$ with $\mu_{\vc{p}}(\mc{A}) > (1-\dD)^n$ satisfies
$\mu_{\vc{q}}((\mc{A} \times \mc{A})^{\mc{V}_\cap}_{\vc{w}}) > (1-\eps)^n$.
\end{theo}

\subsection{Maximum entropy and large deviations}

Next we will discuss an equivalence of measures 
that will later combine with Theorem \ref{binary} to yield Theorem \ref{thm: trichotomy}.
Here we are guided by the maximum entropy principle
(proposed by Jaynes \cite{jaynes} in the context of Statistical Mechanics)
which suggests considering the distribution with maximum entropy
subject to the constraints of our problem, as defined in the following lemma
(the proof is easy, and will be given in Section \ref{sec:prob}).

\begin{lemma} \label{maxent2}
Suppose $\mc{V} = (\vc{v}^i_j)$ is an $(n,J)$-array in $\mb{Z}^D$ and $\vc{w} \in \mb{Z}^D$.
Let $\mc{M}^{\mc{V}}_{\vc{w}}$ be the set of 
probability measures $\mu$ on $J^n$ such that $\mc{V}(\mu) = \vc{w}$.
Then, provided $\mc{M}^{\mc{V}}_{\vc{w}}$ is non-empty, 
there is a unique distribution $\mu^{\mc{V}}_{\vc{w}} \in \mc{M}^{\mc{V}}_{\vc{w}}$
with $H(\mu^{\mc{V}}_{\vc{w}}) = \max_{\mu \in \mc{M}^{\mc{V}}_{\vc{w}}} H(\mu)$,
and $\mu^{\mc{V}}_{\vc{w}}$ is a product measure $\mu_{\vc{p}^{\mc{V}}_{\vc{w}}}$ on $J^n$,
where $\sum_{i \in [n], j \in J} (p^{\mc{V}}_{\vc{w}})^i_j \vc{v}^i_j = \vc{w}$.
\end{lemma}

We will show that $\mu^{\mc{V}}_{\vc{w}}$ is equivalent to the uniform measure
on $(J^n)^{\mc{V}}_{\vc{w}}$, in the sense of exponential contiguity, defined as follows. 
(It is reminiscent of, but distinct from, the more well-known
theory of contiguity, see \cite[Section 9.6]{JLR}.)

\begin{dfn} \label{expctg} 
Let $\mu=(\mu_n)_{n \in \mb{N}}$ and $\mu'=(\mu'_n)_{n \in \mb{N}}$,
where $\mu_n$ and $\mu'_n$ are probability measures 
on a finite set $\OO_n$ for all $n \in \mb{N}$.
Let $\mc{F} = (\mc{F}_n)_{n \in \mb{N}}$ where each 
$\mc{F}_n$ is a set of subsets of $\OO_n$.

We say that $\mu'$ exponentially dominates $\mu $ relative to $\mc{F}$,
and write $\mu \lesssim_{\mc{F}} \mu'$, if for $n^{-1} \ll \dD \ll \eps \ll 1$
and $A_n \in \mc{F}_n$ with $\mu_n(A_n)  > (1-\dD)^n$
we have $\mu'_n(A_n) > (1-\eps)^n$.
We say that $\mu$ and $\mu'$ are exponentially contiguous relative to $\mc{F}$, 
and write $\mu \approx_{\mc{F}} \mu'$ if $\mu \lesssim_{\mc{F}} \mu'$ and $\mu' \lesssim_{\mc{F}} \mu$.

If $\DD = (\DD_n)_{n \in \mb{N}}$ with each $\DD_n \sub \OO_n$
then we write $\mu \lesssim_\DD \mu'$ if $\mu \lesssim_{\mc{F}} \mu'$,
where $\mc{F}_n$ is the set of all subsets of $\DD_n$;
we define $\mu \approx_\DD \mu'$ similarly.
\end{dfn}

Note that $\lesssim_{\mc{F}}$ is a partial order and $\approx_{\mc{F}}$ is an equivalence relation.

The following result establishes the required equivalence of measures
under the same hypotheses as in the previous subsection.
It can be regarded as a large deviation principle for
conditioning $\vc{x} \in J^n$ on the event $\mc{V}(\vc{x})=\vc{w}$ 
(see \cite{DZ} for an overview of this area).

\begin{theo} \label{ldp} 
Let $0 < n^{-1} \ll \gG, \kK, k^{-1}, D^{-1}, C^{-1} < 1$
and $\vc{R} \in \mb{R}^D$ with $\max_d R_d < n^C$.
Suppose $\mc{V} = (\vc{v}^i_j)$ is an  $\vc{R}$-bounded
$\gG$-robustly $(\vc{R},k)$-generating $(n,J)$-array in $\mb{Z}^D$,
and $\vc{w} \in \mb{Z}^D$ such that
$\mu_{\vc{p}^{\mc{V}}_{\vc{w}}}$ is $\kK$-bounded.
Let $\nu$ be the uniform distribution on $\DD_n := (J^n)^{\mc{V}}_{\vc{w}}$.
Then $\mu_{\vc{p}^{\mc{V}}_{\vc{w}}} \approx_\DD \nu$.
\end{theo} 

To apply Theorem \ref{ldp} under combinatorial conditions,
we will use the following lemma which shows that
$\mu _{{\bf p}^{\cal V}_{\bf w}}$ is $\kK$-bounded
under our general position condition on $\mc{V}$.
(See also Section 4 for a more general result based
on VC-dimension that applies to larger alphabets.)

\begin{lemma} \label{binary bounded}
Let $0 < n^{-1} \ll \kK \ll \gG, \gG' \ll \alpha, D^{-1}$. 
Suppose $\mc{V} = (\vc{v}^i)$ is an $\vc{R}$-bounded $\gG'$-robustly 
$(\gG,\vc{R})$-generic $(n,\{0,1\})$-array in $\mb{Z}^D$
and $|(\{0,1\}^n)^{\mc{V}}_{\vc{w}}| \ge (1+\alpha )^n$. 
Then $\mu^{\mc{V}}_{\vc{w}}$ is $\kK$-bounded.
\end{lemma}

Alexander Barvinok remarked (personal communication) 
that similar results to Theorem \ref{ldp} and Lemma \ref{binary bounded} 
were obtained by Barvinok and Hartigan in \cite{BH}. 
Theorem 3 of \cite{BH} gives stronger 
bounds on $|(\{0,1\}^n)^{\cal V}_{\bf w}|$ where applicable,
but their assumptions are very different to ours
(they assume bounds for quadratic forms of certain inertia tensors),
and they also require that the vectors all operate at the `same scale', 
so their results do not apply to the Kalai vectors.
Although our bounds are weaker, our proofs are considerably shorter,
and furthermore, stronger bounds here would not give any improvements
elsewhere in our paper, as they account for a term subexponential in $n$, 
while our working tolerance is up to a term exponential in $n$.

\subsection{Supersaturation} \label{subsec: supersaturation}

We now give a brief overview of the strategy for combining the results 
of the previous two subsections to prove supersaturation,
and also indicate the conditions that determine
which case of Theorem \ref{thm: trichotomy} holds.
Under the set up of Theorem \ref{thm: trichotomy},
a telegraphic summary of the argument is:
\[	|{\cal A}| \geq (1-\delta )^n|{\cal X} |
	\overset{\text{Theorem } \ref{ldp}}{\implies} 
	\mu _{{\bf p}^{\cal V}_{\bf z}}({\cal A}) 
		\geq 
	(1-\delta ')^n 
	\overset{\text{Theorem } \ref{binary}}{\implies}	
	\mu _{{\bf q}}(({\cal A}\times {\cal A})^{{\cal V}_{\cap }}_{\bf w}) \geq (1-\eps )^n,\]
where $\mu _{{\bf q}}$ is chosen to optimise the lower bound on
$|({\cal A}\times {\cal A})^{{\cal V}_{\cap }}_{\bf w}|$ implied by the final inequality.

The best possible supersaturation bound (case $i$ of Theorem \ref{thm: trichotomy})
arises when Theorem \ref{binary} is applicable with $\mu _{\bf q}$ 
equal to the maximum entropy measure $\mu_{\wt{\bf q}}$ that represents
$({\cal X} \times {\cal X})^{{\cal V}_{\cap }}_{\bf w}$:
this case holds when $\mu_{\wt{\bf q}}$ is $\kappa $-bounded 
and has marginals $\mu_{\wt{\bf p}}$ close to 
$\mu_{\bf p} := \mu _{{\bf p}^{\cal V}_{\bf z}}$.

Case $ii$ of Theorem \ref{thm: trichotomy} holds
if $\mu_{\wt{\bf q}}$ is $\kappa $-bounded but
$\mu_{{\wt{\bf p}}}$ is not close to $\mu_{\bf p}$:
then $\mu_{{\wt{\bf p}}}$ is concentrated on
a small subset ${\cal B}_{full}$ of ${\cal X}$, which is
responsible for almost all $\bf w$-intersections in ${\cal X}$.

Lastly, case $iii$ of Theorem \ref{thm: trichotomy} holds
if $\mu_{\wt{\bf q}}$ is not $\kappa $-bounded.
The key to understanding this case is the well-known \cite{VC} 
Vapnik-Chervonenkis dimension, defined as follows.

\begin{dfn} \label{vc}
We say that $\mc{A} \sub J^n$ shatters $X \sub [n]$ if for any $(j_x: x \in X) \in J^X$ 
there is $\vc{a} \in \mc{A}$ with $a_x=j_x$ for all $x \in X$.
The VC-dimension $\dim_{VC}(\mc{A})$ of $\mc{A}$ is the largest size
of a subset of $[n]$ shattered by $\mc{A}$.
\end{dfn}

To see why it is natural to consider the VC-dimension,
consider the problem of finding an intersection of size $n/3$
among subsets of $[n]$ of size $2n/3$.
The conditions of the Frankl-R\"odl theorem are not satisfied,
and indeed the conclusion is not true:
take $\mc{A} = \{ A \in \tbinom{[n]}{2n/3}: 1 \notin A\}$.
Considering $\tbinom{[n]}{2n/3} \times_{n/3} \tbinom{[n]}{2n/3}$ 
as a subset of $(\{0,1\} \times \{0,1\})^n$, we see that no
coordinate can take the value $(0,0)$, so there is not even
a shattered set of size $1$! Modifying this example in the obvious way we see that it is natural to assume a bound that is linear in $n$. We also note that this example shows that 
the `Frankl-R\"odl analogue' of Conjecture \ref{kalai} is not true, 
and hints towards a counterexample for Kalai's conjecture. More generally, we will prove that $\kappa$-boundedness of ${\mu }_{\bf q}$ is roughly equivalent to the VC-dimension of $({\cal X} \times {\cal X})^{{\cal V}_{\cap }}_{\bf w}$ being large as a subset of $(\{0,1\}\times \{0,1\})^n$ (see Lemma \ref{bounded}). Case $iii$ of Theorem \ref{thm: trichotomy} will apply when $({\cal X} \times {\cal X})^{{\cal V}_{\cap }}_{\bf w}$ has low VC-dimension.

The above outline also gives some indication of 
how the values in Theorem \ref{thm: Kalai} arise.
As described above, the supersaturation conclusion 
desired by Conjecture \ref{kalai} 
(case $i$ of Theorem \ref{thm: trichotomy})
needs $\mu_{\wt{\bf q}}$ to have marginals $\mu_{\wt{\bf p}}$ 
close to $\mu_{\bf p} := \mu _{{\bf p}^{\cal V}_{\bf z}}$.
We can describe $\mu_{\wt{\bf q}}$ and $\mu_{\bf p}$
explicitly using Lagrange multipliers:
they are Boltzmann distributions (see Lemma \ref{boltzmann}). 
In general, it is not possible for one 
Boltzmann distribution to be a marginal of another,
which explains why  Conjecture \ref{kalai} is generally false.
An analysis of the special conditions under it is possible
gives rise to the characterisation of $\GG$ in Theorem \ref{thm: Kalai}.

The outline also suggests a possible characterisation
of the optimal level of supersaturation in all cases
(i.e.\ including those for which Kalai's conjecture fails).
Any choice of $\mu _{\bf q}$ satisfying the hypotheses of 
Theorem \ref{binary} with marginal distributions 
$\mu ^{\cal V}_{\bf z}$ gives a lower bound on
$|({\cal A}\times {\cal A})^{{\cal V}_{\cap }}_{\bf w}|$,
and the optimal such lower bound is obtained
by taking such a measure with maximum entropy. 
Is this essentially tight? We wil give a positive answer
to this question by proving a matching upper bound
in Section \ref{sec:tightsupersat}.

Finally, we remark that our method allows
different vectors defining the sizes of intersections
from those defining the sizes of sets in the family,
i.e.\ $\mc{V}'$-intersections in $(\{0,1\}^n)^{\cal V}_{\bf z}$;
in Section \ref{subsec: FR-patterns} we show such an application
to give a new proof of a theorem of 
Frankl and R\"odl \cite[Theorem 1.15]{FrRo}
on intersection patterns in sequence spaces.

\subsection{Organisation of the paper}

In the next section we collect some probabilistic 
methods that will be used throughout the paper.
We prove the large deviation principle (Theorem \ref{ldp})
in Section \ref{sec:ldp}. In Section \ref{sec:vc} 
we establish the connection between VC-dimension
and boundedness of maximum entropy measures.
Section \ref{sec:counterex} is expository:
we give two concrete counterexamples 
to Kalai's Conjecture \ref{kalai}.
Next we introduce a more general setting in Section \ref{sec:gen},
state our most general result (Theorem \ref{general}), 
and show that it implies our
probabilistic intersection theorem (Theorem \ref{binary}).
In Section \ref{sec:cor} we prove a correlation inequality
needed for the proof of Theorem \ref{general};
as far as we are aware, the inequality is quite unlike other 
such inequalities in the literature.
We prove Theorem \ref{general} in Section \ref{sec:pf},
and then deduce our main theorem (\ref{thm: trichotomy}) 
in Section \ref{sec:trichotomy}. Our corrected form of
Kalai's conjecture (Theorem \ref{thm: Kalai}) 
is proved in Section \ref{sec:concrete description};
we also show here in much more generality that supersaturation 
of the form conjectured by Kalai is rare.
In Section \ref{sec:tightsupersat} we give a complete 
characterisation of the optimal level of supersaturation
in terms of a certain optimisation problem for measures.
Lastly, in section \ref{sec:cts} we recast our results
in terms of `exponential continuity': a notion that 
arises naturally when comparing distributions
according to exponential contiguity,
and may be interpreted in terms 
of robust statistics for social choice:
this point and several potential directions
for future research are addressed in the concluding remarks.

\subsection{Notation}

We identify subsets of a set with their characteristic vectors:
$A \sub X$ corresponds to $\vc{a} \in \{0,1\}^X$, where $a_i=1 \Lra i \in A$.
The Hamming distance between vectors $\vc{a}$ and $\vc{a}'$ in a product space $J^n$
is $d(\vc{a},\vc{a}') = |\{i \in [n]: a_i \ne a'_i\}|$.
Given a set $X$, we write $\tbinom{X}{k} = \{A \sub X: |A|=k\}$.
We write $\dD \ll \eps$ to mean for any $\eps > 0$ there exists $\dD_0 > 0$ such
that for any $\dD \leq \dD_0$ the following statement holds. 
Statements with more constants are defined similarly. 
We write $a = b \pm c$ to mean $b - c \leq a \leq b + c$.
Throughout the paper we omit floor and ceiling symbols where they do not affect the argument. 
All vectors appear in boldface. 

\section{Probabilistic methods} \label{sec:prob}

In this section we gather several probabilistic methods that
will be used throughout the paper: concentration inequalities, 
entropy, an application of Dependent Random Choice
to the independence number of product graphs,
and an alternative characterisation of exponential contiguity.

\subsection{Concentration inequalities}

We start with the well-known Chernoff bound
(see e.g.\ \cite[Appendix A]{aands}).

\begin{lemma}[Chernoff's inequality]
Suppose $t \ge 0$ and $X := \sum_{i\in [n]} X_i$,
where $X_1,\ldots, X_n$ are independent random variables with 
$|X_i - \mb{E}X_i| \leq a_i$ for all $i\in [n]$. Then
$\mb{P}(|X - \mb{E}X| \ge t) \leq 2e^{-t^2/(2\sum_{i=1}^n a_i^2)}$.
\end{lemma}

An easy consequence is the following concentration inequality 
for random sums of vectors.

\begin{lemma} \label{vchern}
Suppose $\mu_{\vc{p}}$ is a product measure on $J^n$, and 
$\mc{V} = (\vc{v}^i_j)$ is an $\vc{R}$-bounded $(n,J)$-array in $\mb{Z}^D$.
Let $X = \mc{V}(\vc{a})$ with $\vc{a} \sim \mu_{\vc{p}}$ and $t \ge 0$.
Then $\mb{P}(\|X - \mb{E}X\|_{\vc{R}} \ge t) \le 2De^{-t^2/8n}$.
\end{lemma}

\nib{Proof.}
For each $d \in D$, we have $X_d = \sum_{i \in [n]} X_{d,i}$,
where $X_{d,i}({\bf a}) = v^i_{a_i,d}$ are independent random variables 
with $|X_{d,i}| \le R_d$ for all $i \in [n]$.
By Chernoff's inequality we have
$\mb{P}(|X_d - \mb{E}X_d| \ge tR_d) \le 2e^{-t^2/8n}$,
so the lemma follows from a union bound. \qed

\medskip

We will also use the following consequence of
Azuma's martingale concentration inequality (see e.g.\ \cite{McD}). We say that $f:J^n \to \mb{R}$ is \emph{$b$-Lipschitz} if for any 
${\bf a}, {\bf a'} \in J^n$ differing only in a single coordinate we have $|f({\bf a})-f({\bf a}')| \le b$. 

\begin{lemma} \label{lip}
Suppose $Z = (Z_1,\dots,Z_n)$ is a sequence of independent random variables,
and $X=f(Z)$, where $f$ is $b$-Lipschitz.
Then $\mb{P}(|X-\mb{E}X|>a) \le 2e^{-a^2/2nb^2}$.
\end{lemma}

\subsection{Entropy}

In this subsection we record some basic properties of entropy
(see \cite{covthom} for an introduction to information theory).
The entropy of a probability distribution $\vc{p}=(p_1,\dots,p_n)$
is $H(\vc{p}) = - \sum_{i \in [n]} p_i \log_2 p_i$.
The entropy of a random variable $X$ taking values 
in a finite set $S$ is $H(X)=H(\vc{p})$,
where $\vc{p} = (p_s: s \in S)$ is the law of $X$,
i.e.\ $p_s = \mb{P}(X = s)$.
When $\vc{p}=(p,1-p)$ takes only two values
we write $H(p)=H(\vc{p})=-p\log_2 p - (1-p)\log_2(1-p)$.

Entropy is subadditive:
if $X=(X_1,\dots,X_n)$ then $H(X) \le \sum_{i=1}^n H(X_i)$,
with equality if and only if the $X_i$ are independent.
An equivalent reformulation is the following lemma.

\begin{lemma} \label{maxent}
Suppose $\mu$ is a probability measure on $\prod_{s \in [S]} J_s$ 
with marginals $(\mu_s: s \in [S])$. Then $H(\mu) \le \sum_{s \in [S]} H(\mu_s)$,
with equality if and only if $\mu = \prod_{s \in [S]} \mu_s$.
\end{lemma}

It is easy to deduce Lemma \ref{maxent2} from Lemma \ref{maxent}.
Indeed, consider $\mu \in \mc{M}^{\mc{V}}_{\vc{w}}$ with maximum entropy.
Let $p^i_j = \mb{P}_{\vc{x} \sim \mu}(x_i = j)$.
Then $\vc{w} = \mc{V}(\mu) = \sum_{i \in [n], j \in J}  p^i_j \vc{v}^i_j$,
so the product measure $\mu_{\vc{p}}$ is in $\mc{M}^{\mc{V}}_{\vc{w}}$,
and $H(\mu) \le H(\mu_{\vc{p}})$, with equality if and only if $\mu=\mu_{\vc{p}}$.
As $\mc{M}^{\mc{V}}_{\vc{w}}$ is convex, uniqueness follows from strict concavity
of the entropy function, which we will now explain.
It is often convenient to use the notation
$H(\vc{p}) = \sum_{i \in [n]} L(p_i)$,
where $L(p) = -p\log_2 p = -p\tfrac{\log p}{\log 2}$.
Note that $L'(p) = -\tfrac{1+\log p}{\log 2}$
and $L''(p) = -\tfrac{1}{p\log 2} < 0$, so $L$ is strictly concave.
The following lemma is immediate from these formulae 
and the mean value form of Taylor's theorem:
$f(a+t)=f(a)+f'(a)t+f''(a+t')t^2/2$
for some $0 < t' < t$.

\begin{lemma} \label{entapprox} If $|t|<\min (p, 1-p)$ then
\begin{enumerate}
\item $L(p+t)-L(p) = - \big (\frac{1+\log p}{\log 2}\big ) t 
\pm (p - |t|)^{-1} t^2$,
\item $L(p+t)+L(p-t)-2L(p) \le -\frac{t^2}{\log 2}$.
\end{enumerate}
\end{lemma}

We deduce the following `stability version'
of the uniqueness of the maximum entropy measure,
which quantifies the decrease in entropy
in terms of distance from the maximiser.

\begin{lemma} \label{perturbmaxent}
Suppose $\mu_{\bf p} = \mu^{\mc{V}}_{\bf z}$
and $\mu_{\wt{\bf p}} \in \mc{M}^{\mc{V}}_{\bf z}$.
If $\|{\bf p}-\wt{\bf p}\|_1 > \dD n$
then $H(\wt{\bf p}) < H({\bf p}) - \dD^2 n$.
\end{lemma}

\nib{Proof.}
Let ${\bf p}' = ({\bf p} + {\wt{\bf p}})/2$ and 
note that ${\mu }_{\bf p '} \in {\cal M}^{\cal V}_{\bf z}$.
By definition of $\mu^{\mc{V}}_{\bf z}$ we have
$H(\mu _{{\bf p}'}) \le H(\mu _{\bf p})$,
so $H(\mu _{\bf p }) - H(\mu _{\wt{\bf p}}) 
\ge 2H(\mu _{{\bf p}'}) - H(\mu _{\bf p}) - H(\mu _{\wt {\bf p}})
\ge \sum _{i\in [n]} (p_i - {\wt p}_i)^2 \ge \dD^2 n$,
by Lemma \ref{entapprox} $ii$ and then Cauchy-Schwarz. \qed

\medskip

We conclude this subsection with a perturbation lemma.

\begin{lemma} \label{perturb} 
Suppose $\mu$ is a probability distribution on $X$ 
and $-\log_2 \mu(x) > H(\mu)$ for some $x \in X$.
Then there is $t>0$ such that $\nu = (1-t)\mu + t1_x$ 
has $H(\nu) > H(\mu)$.
\end{lemma}

\begin{proof}
For $\mu (x) \neq 0$, for small enough $t$ by Lemma \ref{entapprox} $i$ we have
\begin{align*}
L(\nu(x))-L(\mu(x)) & = (\nu(x)-\mu(x))L'(\mu(x)) 
 \pm 2\mu(x)^{-1} (\nu(x)-\mu(x))^2 \\
& = - \Big (\frac {1 + \log_2 \mu(x)}{2}\Big )(1-\mu (x))t - O(t^2).
\end{align*}
If $y \ne x$ and $\mu (y) = 0$ then $L(\nu(y)) - L(\mu (y)) = 0$. Therefore, by Lemma 2.5(i), for all $y\neq x$ we have
\begin{equation}
\label{equation: entropy change calc}
L(\nu(y))-L(\mu(y)) = \Big ( \frac {1 + \log _2 \mu (y)}{2} \Big ) \mu (y) t  - O(t^2).
\end{equation}
All combined, this gives
\begin{align}
H(\nu)-H(\mu) = \sum_{y \in X} L(\nu(y))-L(\mu(y)) 
& = \sum _{y\in X} \Big ( \frac {1 + \log _2 \mu (y)}{2} \Big ) \mu (y) t - \Big ( \frac{1 + \log _2 \mu (x)}{2} \Big )t - O(t^2)\nonumber \\
&= t\Big (- \frac{\log_2 \mu(x)}{2} - \frac{H(\mu)}{2} - O(t) \Big) > 0,\nonumber 
\end{align}
for small  $t>0$. The case $\mu (x) = 0$ is similar, using $L(\nu (x)) - L(\mu (x)) = -t\log _2t$ with \eqref{equation: entropy change calc}.
\end{proof}

\subsection{Dependent Random Choice}

We will use the following version of Dependent Random Choice
(see \cite[Lemma 11]{kl} for a proof and \cite{FS} for a comprehensive survey of the method).
We write $N_G(u,u') := \{v \in V(G): uv, u'v \in E(G)\}$
for the set of common neighbours of $u$ and $u'$ in a graph $G$.

\begin{lemma} \label{preDRC}
Let $t \in \mb{N}$ and $G = (V_1,V_2,E)$ be a bipartite graph 
with $|V_i| = N_i$ and $|E|=\aA N_1N_2$. 
Then there is $U \subset V_1$ with $|U| \geq \aA^t N_1/2$ such that 
$|N_G(u,u')| \geq \aA N_1^{-1/t}N_2$ for all $u,u' \in U$. 
\end{lemma}

The following is an immediate consequence of Lemma \ref{preDRC},
applied with $t = \bcl{2/c\eps}$.

\begin{lemma} \label{DRC}
Let $0 < N^{-1} \ll \dD \ll \eps,c < 1$.
Suppose $G = (V_1,V_2,E)$ is a bipartite graph with each $|V_i|=N_i$,
where $N \le N_1^c \le N_2 \le N_1^{1/c}$ and $e(G) > (N_1N_2)^{1-\dD}$. 
Then there is $U \subset V_1$ with $|U| > N_1^{1-\eps}$ such that 
$|N_G(u,u')| > N_2^{1-\eps}$ for all $u,u' \in U$.
\end{lemma}

We say that $S \subset V(G)$ is independent if it contains no edges of $G$. 
The independence number $\aA(G)$ of $G$ is the maximum size of an independent set in $G$. 
Given graphs $G_1, \ldots, G_k$, we write $G_1 \times \cdots \times G_k$ for the graph  
on vertex set $V(G_1) \times \cdots \times V(G_k)$, 
in which vertices $(u_1,\cdots ,u_k)$ and $(v_1,\cdots ,v_k)$ 
are joined by an edge if $u_iv_i \in E(G_i)$ for all $i \in [k]$. 

\begin{lemma} \label{indep2}
Let $0 < N^{-1} \ll \dD \ll \eps,c < 1$ and $N \le N_1^c \le N_2 \le N_1^{1/c}$.
Suppose for $i=1,2$ we have graphs $G_i$ on $V_i$ with $|V_i| = N_i$ and $\aA(G_i) \leq N_i^{1-\eps}$. 
Then $\aA(G_1 \times G_2) \leq \brac{N_1N_2}^{1-\dD}$.
\end{lemma}

\begin{proof} 
Suppose $E \sub V_1 \times V_2$ with $|E| > (N_1N_2)^{1-\dD}$.
Consider the bipartite graph $G=(V_1,V_2,E)$.
Let $U$ be as in Lemma \ref{DRC}.
As $|U|>\aA(G_1)$, there is an edge $u_1u_2$ of $G_1$ in $U$.
As $|N_G(u_1,u_2)|>\aA(G_2)$, there is an edge $v_1v_2$ of $G_2$ in $N_G(u_1,u_2)$.
Then $(u_1,v_1)(u_2,v_2) \in E$, so $E$ is not independent in $G_1 \times G_2$.
\end{proof}

By repeated application of the previous lemma, we obtain the following corollary. 

\begin{lemma} \label{indep}
Let $0 < N^{-1} \ll \dD \ll \eps,c, k^{-1} < 1$ and $N_1,\dots,N_k \in \mb{N}$ 
with $N \le N_i^c \le N_j \le N_i^{1/c}$ for all $i,j\in [k]$.
Suppose for $i \in [k]$ we have graphs $G_i$ on $V_i$ 
with $|V_i| = N_i$ and $\aA(G_i) \leq N_i^{1-\eps}$. 
Then $\aA(G_1 \times \cdots \times G_k) \leq \brac{N_1\cdots N_k}^{1-\dD}$.
\end{lemma}

\subsection{Exponential Contiguity}

We conclude this section with an alternative characterisation of exponential contiguity.

\begin{lemma} \label{ctg1} 
$\mu \lesssim_\DD \mu'$ if and only if for $n^{-1} \ll \dD \ll \eps \ll 1$
and $B_n = \{ x \in \DD_n: \mu'_n(x) < (1-\eps)^n \mu_n(x)\}$
we have $\mu_n(B_n) \le (1-\dD)^n$.
\end{lemma}

\nib{Proof.}
Let $n^{-1} \ll \dD \ll \eps \ll 1$.
Suppose first that if $A_n \sub \DD_n$ with
$\mu_n(A_n) > (1-\dD)^n$ then $\mu'_n(A_n) > (1-\eps)^n$.
As $\mu'_n(B_n) \le (1-\eps)^n \mu_n(B_n) \le (1-\eps)^n$,
we cannot have $\mu_n(B_n) > (1-\dD)^n$, so we have $\mu_n(B_n) \le (1-\dD)^n$.
Conversely, suppose $\mu_n(B_n) \le (1-\dD)^n$ 
and $A_n \sub \DD_n$ with $\mu_n(A_n) > (1-\dD/2)^n$. Then
$\mu'_n(A_n) \ge \mu'_n(A_n \sm B_n) \ge (1-\eps)^n \mu_n(A_n \sm B_n) > (1-2\eps)^n$.
\qed

\section{Large deviations of fixed sums} \label{sec:ldp}

In this section we prove Theorem \ref{ldp}.
Our first lemma will be used to show that
the maximum entropy measure is exponentially 
dominated by the uniform measure.

\begin{lemma} \label{uctg1}
Let $0 < n^{-1} \ll \eta \ll \kappa , |J|^{-1} \ll 1$.
Suppose $\mu_{\vc{p}}$ is a $\kappa $-bounded product 
measure on $J^n$. Let $\mc{B} = \{ \vc{x} \in J^n: 
\log_2 \mu_{\vc{p}}(\vc{x}) \notin - H(\mu_{\vc{p}}) \pm \eta n \}$.
Then $\mu_{\vc{p}}(\mc{B}) \le (1-\eta^3)^n$. 
\end{lemma}

\nib{Proof.}
Consider $\vc{x} \sim \mu_{\vc{p}}$ and
$X := \log_2 \mu_{\vc{p}}(\vc{x}) = \sum_{i \in [n]} X_i$,
where $X_i = \sum_{j \in J} {\bf 1}_{x_i=j} \log _2(p^i_j)$. 
As $\mu _{\bf p}$ is $\kappa $-bounded, the $X_i$ satisfy $|X_i - {\mathbb E}(X_i)| \leq 2 \log _2(\kappa ^{-1})$ for all $i\in [n]$. As these random variables are independent and $\mb{E}X = -H(\mu _{\vc{p}})$, the bound on ${\mu _{\bf p}}({\cal B})$ follows from Chernoff's inequality. \qed

\medskip

Our next lemma gives a lower bound for point probabilities 
of maximum entropy measures, which implies an upper bound 
on the number of solutions of $\mc{V}(\vc{x})=\vc{w}$.
 
\begin{lemma} \label{uctg2}
Suppose $\mc{V} = (\vc{v}^i_j)$ is an $(n,J)$-array in $\mb{Z}^D$ and $\vc{w} \in \mb{Z}^D$ and
$\mu_{\vc{p}} = \mu_{\vc{p}^{\mc{V}}_{\vc{w}}}$.
Then for all ${\bf x} \in (J^n)^{\cal V}_{\bf w}$ we have $- \log _2 \mu _{\bf p}({\bf x}) \leq H(\mu _{\bf p})$. In particular, $\log_2 |(J^n)^{\cal V}_{\bf w}| \le H(\mu_{\vc{p}})$.
\end{lemma}

\nib{Proof.}
If some ${\bf x} \in (J^n)^{\cal V}_{\bf w}$ satisfies $- \log _2 \mu _{\bf p}({\bf x}) > H(\mu
_{\bf p})$ then Lemma \ref{perturb}(i)  shows that $ \nu = (1-t) \mu _{\bf p} + t1_{\bf x}$
satisfies $H({\nu }) > H({\mu }_{\bf p})$ for some $t > 0$. However as $\nu \in
\mc{M}^{\mc{V}}_{\vc{w}}$ this would contradict the choice of ${\bf p}^{\cal V}_{\bf w}$. 
The second statement now follows as 
$|(J^n)^{\cal V}_{\bf w}| 2^{-H(\mu _{\bf p})} 
\leq \sum _{\bf x} \mu _{\bf p}({\bf x}) \leq 1$. \qed

\medskip

Our final lemma will give an approximate formula for the
number of solutions of $\mc{V}(\vc{x})=\vc{w}$ 
(as mentioned in the introduction, \cite[Theorem 3]{BH} 
gives stronger bounds under different hypotheses).
First we require a small set that efficiently generates $\mb{Z}^D$,
as described by the following definition and associated lemma,
which shows that such a set exists under the mild assumption 
of polynomial growth for the coordinate scale vector $\vc{R}$
(this will also be used later in Theorem \ref{general}).

\begin{dfn}
We say that $\mc{U} \sub \mb{Z}^D$ is $(k,B,\vc{R})$-generating 
if for any $\vc{v} \in \mb{Z}^D$ we have 
$\vc{v} = \sum_{\vc{u} \in \mc{U}} c_{\vc{u}} \vc{u}$, with each 
$c_{\vc{u}} \in \mb{Z}$ with $|c_{\vc{u}}| \le k \|\vc{v}\|_{\vc{R}} + B$.
\end{dfn}

\begin{lemma} \label{u}
If $n^{-1} \ll \bB, C^{-1}, D^{-1}$ and $\max_d R_d < n^C$ then there 
is a $(1,\bB n,\vc{R})$-generating $\vc{R}$-bounded $\mc{U} \sub \mb{Z}^D$ 
with $|\mc{U}| \le D(C+2)$.
\end{lemma}

\nib{Proof.}
Let $\mc{U}$ be the set of all $\vc{u}=(u_1,\dots,u_D)$ such that 
for some $d \in [D]$ we have $u_{d'}=0$ for all $d' \ne d$ and $u_d = R_d$ 
or $u_d = \bfl{\bB n/CD}^{a_d} \in [R_d]$ for some integer $a_d \ge 0$. \qed

\begin{lemma} \label{uctg3}
Let $0 < n^{-1} \ll \lambda \ll \dD \ll \gG, \kK, k^{-1}, D^{-1}, C^{-1} < 1$
and $\vc{R} \in \mb{R}^D$ with $\max_d R_d < n^C$.
Suppose
\begin{enumerate}
\item $\mc{V} = (\vc{v}^i_j)$ is an  $\vc{R}$-bounded
$\gG$-robustly $(\vc{R},k)$-generating $(n,J)$-array in $\mb{Z}^D$, 
\item ${\mu }_{\bf p}$ is a $\kappa$-bounded product measure on $J^n$ 
 with $\| {\mathbb E}{\cal V}({\bf x}) - {\bf  w}\|_{\bf R} \leq \lambda n$,
 where $\vc{w} \in \mb{Z}^D$.
\end{enumerate}
Then $\log_2 |(J^n)^{\mc{V}}_{\vc{w}}| \geq H(\mu_{\vc{p}}) - \dD n$. 

In particular, if $\mu_{\vc{p}} = \mu_{\vc{p}^{\mc{V}}_{\vc{w}}}$ is $\kK$-bounded 
then $\log _2 |(J^n)^{\cal V}_{\bf w}| = H({\mu }_{\bf p}) \pm \delta n$.
\end{lemma}

\nib{Proof.}
We first note that the final statement of the lemma 
follows from the first: the latter gives the lower bound, 
as ${\mathbb E}{\cal V}({\bf x}) = {\bf w}$
when ${\bf x} \sim \mu_{\vc{p}^{\mc{V}}_{\vc{w}}}$,
and the upper bound follows from Lemma \ref{uctg2}.

It remains to prove the first statement of the lemma.
Let $\mc{F}$ be the set of $\vc{x} \in J^n$ 
such that there is $\vc{x}' \in (J^n)^{\mc{V}}_{\vc{w}}$ with Hamming distance
$d(\vc{x},\vc{x}') < \dD^2 n$. We claim that $\mu_{\vc{p}}(\mc{F}) > 1/2$.

First we assume the claim and deduce the lower bound.
By double-counting pairs $(\vc{x},\vc{x}')$ where $\vc{x} \in \mc{F}$
and $\vc{x}' \in (J^n)^{\mc{V}}_{\vc{w}}$ with $d(\vc{x},\vc{x}') < \dD^2 n$
we have $|\mc{F}| \le | (J^n)^{\mc{V}}_{\vc{w}}| \tbinom{n}{\dD^2 n}|J|^{\dD ^2 n}$, and as $|J| \leq \kappa ^{-1} \ll \dD ^{-1}$ this gives
$\log_2 |(J^n)^{\mc{V}}_{\vc{w}}| \ge \log_2 |\mc{F}| - \dD^{3/2} n$.
Now consider $\mc{F}' = \{ \vc{x} \in \mc{F}: 
\log_2 \mu_{\vc{p}}(\vc{x}) \le -H(\mu_{\vc{p}}) + \dD^2 n \}$. Note that 
$\log_2 \mu_{\vc{p}}(\mc{F}') \le \log_2 |\mc{F}'| - H(\mu_{\vc{p}}) + \dD^2 n $, 
and by Lemma \ref{uctg1} and the claim we have $\mu_{\vc{p}}(\mc{F}') > 1/4$.
Thus $\log_2  |\mc{F}'| \ge H(\mu_{\vc{p}}) - \dD^2 n - 2$, 
so $\log_2 |(J^n)^{\mc{V}}_{\vc{w}}| \ge H(\mu_{\vc{p}}) - \dD n$.

To prove the claim, we consider $\vc{x} \sim \mu_{\vc{p}}$
and show that with probability at least $1/2$ there is  
$\vc{x}' \in (J^n)^{\mc{V}}_{\vc{w}}$ with $d(\vc{x},\vc{x}') < \dD^2 n$.
Let $\mc{B}_1$ be the event that 
$\|\mc{V}(\vc{x})-\vc{w}\|_{\vc{R}} \ge \dD^3 n$. 
If $\mc{B}_1$ holds, by the triangle inequality 
$\|\mc{V}(\vc{x}) - \mb{E}\mc{V}(\vc{x})\|_{\vc{R}} 
\geq \dD^3n - \lambda n \geq \dD^3n/2$ 
and so ${\mathbb P}(\mc{B}_1) \le 2De^{-\dD^6 n/2^7}$ by Lemma \ref{vchern}.
Next, by Lemma \ref{u} we can fix some $(1,\dD^2 n,\vc{R})$-generating $\vc{R}$-bounded
$\mc{U} =\{\vc{u}_1,\dots,\vc{u}_M\} \sub \mb{Z}^D$ with $M \le D(C+2)$.
By repeatedly applying Definition \ref{robgen}, we can choose pairwise disjoint 
$S_{mt} \sub [n]$ for each $m \in [M]$ and $t \in [\gG n/kM]$
with each $|S_{mj}| \le k$ and $j_i, j'_i  \in J$ for all $i \in S_{mt}$ such that
$\vc{u}_m = \sum_{i \in S_{mt}} (\vc{v}^i_{j_i}-\vc{v}^i_{j'_i})$.
Let $\mc{B}_2$ be the event that for some $m$ we have
$|\{t: x_i = j_i \ \forall i \in S_{mt}\}| < \kK^k \gG n/2kM$
or $|\{t: x_i = j'_i \ \forall i \in S_{mt}\}| < \kK^k \gG n/2kM$.
Then $\mb{P}(\mc{B}_2) < e^{-\dD n}$ by Chernoff's inequality.

Thus with probability at least $1 - e^{-\dD^7n} > 1/2$
neither $\mc{B}_1$ or $\mc{B}_2$ holds for $\vc{x}$.
As $\mc{U}$ is $(1,\dD^3 n,\vc{R})$-generating and 
$\|\mc{V}(\vc{x})-\vc{w}\|_{\vc{R}} < \dD^3 n$, we have
$\mc{V}(\vc{x})-\vc{w} = \sum_{m \in [M]} c_m \vc{u}_m$, with each 
$c_m \in \mb{Z}$ with $|c_m| \le 2k \dD^3 n$.
Now we modify $\vc{x}$ to obtain $\vc{x}'$, where for each $m \in [M]$,
if $c_m>0$ we fix $c_m$ values of $t$ such that $x_i = j_i$
for all $i \in S_{mt}$ and let $x'_i = j'_i$ for all such $i$,
and if $c_m<0$ we fix $c_m$ values of $t$ such that $x_i = j'_i$
for all $i \in S_{mt}$ and let $x'_i = j_i$ for all such $i$.
Then $\mc{V}(\vc{x}')=\vc{w}$, i.e.\ $\vc{x}' \in (J^n)^{\mc{V}}_{\vc{w}}$,
and $d(\vc{x},\vc{x}') < k\sum_{m \in [M]} |c_m| < \dD^2 n$.
This completes the proof of the claim, and so of the lemma.
\qed

\medskip

We deduce Theorem \ref{ldp}, which states that under the hypotheses
of the above lemmas, we have $\mu_{\vc{p}} \approx_\DD \nu$,
where $\mu_{\vc{p}} = \mu_{\vc{p}^{\mc{V}}_{\vc{w}}}$,
and $\nu$ is the uniform distribution on $\DD_n := (J^n)^{\mc{V}}_{\vc{w}}$.

\medskip

\nib{Proof of Theorem \ref{ldp}.}
Let $0 < n^{-1} \ll \dD \ll \eps \ll \gG, \kK, k^{-1}, D^{-1}, C^{-1} < 1$.
Note that $\nu(\vc{x}) = |(J^n)^{\mc{V}}_{\vc{w}}|^{-1}$ 
for all $\vc{x} \in (J^n)^{\mc{V}}_{\vc{w}}$,
and $\log_2 |(J^n)^{\mc{V}}_{\vc{w}}| = H(\mu_{\vc{p}}) \pm \dD n$ 
by Lemma \ref{uctg3}.

Consider $\mc{C} = \{ \vc{x} \in \DD_n: \mu_{\vc{p}}(\vc{x}) < (1-\eps)^n \nu(\vc{x})\}$. 
If there is $\vc{x} \in \mc{C}$ then $\log_2 \mu_{\vc{p}}(\vc{x}) 
< -\log_2 |(J^n)^{\mc{V}}_{\vc{w}}| - \eps n < H({\mu }_{\bf p})$, 
contradicts Lemma \ref{uctg2}, so $\mc{C} = \emptyset $. 
Thus $\nu \lesssim_\DD \mu_{\vc{p}}$ by Lemma \ref{ctg1}.

Now consider $\mc{C}' = \{ \vc{x} \in \DD_n: 
\nu(\vc{x}) < (1-\eps)^n \mu_{\vc{p}}(\vc{x})\}$.
For any $\vc{x} \in \mc{C}'$ we have
$\log_2 \mu_{\vc{p}}(\vc{x}) > -\log_2 |(J^n)^{\mc{V}}_{\vc{w}}| + \eps n$,
so $\log_2 \mu_{\vc{p}}(\mc{C'}) < - \eps^3 n < -\dD n$ by Lemma \ref{uctg1},
i.e.\ $\mu_{\vc{p}} \lesssim_\DD \nu$. \qed

\section{Boundedness, feasibility
and universal VC-dimension} \label{sec:vc}

In this section we will give several combinatorial characterisations
of the boundedness condition on maximum entropy measures 
required in our probabilistic intersection theorem.
The characterisations hold
under the following `multiscale general position' assumption,
which extends Definition \ref{robust gen} to all finite alphabets
(by `multiscale' we mean that the parameter $\gG$ can be arbitrary,
which is true of the Kalai vectors).

\begin{dfn} (robustly generic)
Suppose $\mc{V} = (\vc{v}^i_j)$ is an $(n,J)$-array in $\mb{Z}^D$. 
Let $I \in \tbinom{[n]}{D}$ and $\vc{c} = (\vc{c}^i: i \in I)$
with $\vc{c}^i \in \mb{R}^J$ and $\sum_{j \in J} c^i_j = 0$ 
for all $i \in I$ and $j \in J$. 

We say that $(I,\vc{c})$
is $(\gG,\vc{R})$-generic for $\mc{V}$ if all $|c^i_j| \le \gG^{-1}$,
and writing $\vc{w}^i = \sum_{j \in J} c^i_j \vc{v}^i_j$
and $W = (w^i_d: i \in I, d \in [D])$, 
we have $|\det(W)| \ge \gG \prod_{d \in [D]} R_d$.

We say that $\mc{V}$ is $\gG'$-robustly $(\gG,\vc{R})$-generic
if for any $X \sub [n]$ with $|X| > \gG' n$
there is some $(I,\vc{c})$ with $I \sub X$
that is $(\gG,\vc{R})$-generic for $\mc{V}$.

We say that a sequence $(\mc{V}_n,\vc{R}_n)$ 
of $(n,J)$-arrays and scalings is robustly generic 
if $\mc{V}_n$ is $\gG'$-robustly $(\gG,\vc{R}_n)$-generic
whenever $n^{-1} \ll \gG \ll \gG'$. 
\end{dfn}

It will also be convenient to use the following
sequence formulation of Definition \ref{robgen}.

\begin{dfn}
We say that $(\mc{V}_n,\vc{R}_n)$ 
is robustly generating 
if there are $\gG>0$ and $k,n_0 \in \mb{N}$ such that
$\mc{V}_n$ is $\gG$-robustly $(\vc{R}_n,k)$-generating
for all $n>n_0$.
\end{dfn}

Next we will define the combinatorial conditions
that appear in our characterisation.
We recall the definition of VC-dimension 
and also define a universal variant
that will be important in the
proof of Theorem \ref{thm: trichotomy}
in section \ref{sec:trichotomy}.

\begin{dfn} \label{uvc}
We say that $\mc{A} \sub J^n$ shatters $X \sub [n]$ 
if for any $(j_x: x \in X) \in J^X$ 
there is $\vc{a} \in \mc{A}$ with $a_x=j_x$ for all $x \in X$.

The VC-dimension $\dim_{VC}({\cal A})$ of ${\cal A}$ 
is the largest natural $k$ such that ${\cal A}$ 
shatters {\em some} subset of $[n]$ of size $k$.

The universal VC-dimension $\dim_{UVC}({\cal A})$ of ${\cal A}$ 
is the largest natural $k$ such that ${\cal A}$ 
shatters {\em every} subset of $[n]$ of size $k$.
\end{dfn}

Next we give a feasibility condition,
which can be informally understood
as saying that we can solve
any small perturbation of the equation
$\mc{V}_n(\vc{x}) = {\bf z}_n$.

\begin{dfn}
Let $(\mc{V}_n,{\bf R}_n,{\bf z}_n)$ be a sequence
of $(n,J)$-arrays, scalings and vectors in $\mb{Z}^D$.
We say $(\mc{V}_n,{\bf R}_n,{\bf z}_n)$ is $\lL$-feasible
if there is $n_0$ such that for any $n>n_0$,
any ${\bf z}'_n \in \mb{Z}^D$ with 
$\|{\bf z}'_n-{\bf z}_n\|_{{\bf R}_n} \le \lL n$, and
any $(n',J)$-array $\mc{V}'_{n'}$ 
obtained from $\mc{V}_n$ by
deleting at most $\lL n$ co-ordinates,
we have $(J^{n'})^{\mc{V}'_{n'}}_{{\bf z}'_n} \ne 0$.
\end{dfn}

Our final property appears to be a substantial weakening
of our $\kK$-boundedness condition, so it is quite
surprising that it also gives a characterisation.

\begin{dfn}
Suppose $\mu_{\bf p}$ is a product measure on $J^n$.
We say that $\mu_{\bf p}$ is $\kK$-dense if
there are at least $\kK n$ coordinates $i \in [n]$
such that $p^i_j \ge \kK$ for all $j \in J$.
\end{dfn}

Now we can state the main theorem of this section.
The sense of the equivalences in the statement is that 
the implied constants are bounded away from zero together. 
For example, the implication $ii \Ra i$ means that
for any $\dD>0$ there is $\eps>0$ such that
if $\mu^{\mc{V}_n}_{{\bf z}_n}$ is $\dD$-dense
then $\mu^{\mc{V}_n}_{{\bf z}_n}$ is $\eps$-bounded. 

\begin{theo} \label{tfae}
Let $(\mc{V}_n,{\bf R}_n)$ be a robustly generic
and robustly generating sequence
of $(n,J)$-arrays and scalings in $\mb{Z}^D$,
and $({\bf z}_n)$ a sequence of vectors in $\mb{Z}^D$.
The following are equivalent:
\begin{enumerate}
\item $\mu^{\mc{V}_n}_{{\bf z}_n}$ is $\OO(1)$-bounded.
\item $\mu^{\mc{V}_n}_{{\bf z}_n}$ is $\OO(1)$-dense.
\item $\dim_{VC}((J^n)^{\mc{V}_n}_{{\bf z}_n})=\OO(n)$.
\item $\dim_{UVC}((J^n)^{\mc{V}_n}_{{\bf z}_n})=\OO(n)$.
\item $(\mc{V}_n,{\bf R}_n,{\bf z}_n)$ is $\OO(1)$-feasible.
\end{enumerate}
\end{theo}

The main step in the proof of Theorem \ref{tfae} is Lemma \ref{bounded},
which provides the implication $iii \Ra i$.
It also implies Lemma \ref{binary bounded}, as for binary vectors
the following coarse version of the Sauer-Shelah theorem
shows that linear VC-dimension is equivalent to exponential growth.

\begin{lemma} \cite{Sa,Sh} \label{SS} 
For $\mc{A} \sub \{0,1\}^n$ we have
$\dim_{VC}(\mc{A}) = \OO(n) \Lra \log_2 |\mc{A}| = \OO(n)$.
\end{lemma}

\begin{lemma} \label{bounded}
Let $0 < n^{-1} \ll \kK \ll \gG, \gG' \ll \lambda \ll D^{-1}, |J|^{-1}$. 
Suppose $\mc{V} = (\vc{v}^i_j)$ is an $\vc{R}$-bounded $\gG'$-robustly 
$(\gG,\vc{R})$-generic $(n,J)$-array in $\mb{Z}^D$.
If $\dim_{VC}((J^n)^{\mc{V}}_{\vc{w}}) \ge \lL n$
then $\mu^{\mc{V}}_{\vc{w}}$ is $\kK$-bounded.
\end{lemma}

The proof of Lemma \ref{bounded} is immediate from the next two
lemmas, which give the implications $iii \Ra ii$ and $ii \Ra i$
of Theorem \ref{tfae}.

\begin{lemma} \label{bounded.i}
Let $0 < n^{-1} \ll \kK \ll \lL \ll D^{-1}, |J|^{-1}$. 
Suppose $\mc{V}$ is an $(n,J)$-array in $\mb{Z}^D$.
Let ${\bf w} \in {\mb Z}^D$ and $\vc{p} = \vc{p}^{\mc{V}}_{\vc{w}}$.
If $\dim_{VC}((J^n)^{\mc{V}}_{\vc{w}}) \ge \lambda n$
then $\mu_{\vc{p}}$ is $\kK$-dense.
\end{lemma}

\nib{Proof.}
Fix $Z$ with $|Z| > \lL n$ 
such that $(J^n)^{\mc{V}}_{\vc{w}}$ shatters $Z$.
Suppose for a contradiction that $\mu_{\vc{p}}$ is not $\kK$-dense.
Then we have $Y \sub Z$ with $|Y| \ge |Z|/2$
and $(j'_y: y \in Y) \in J^Y$ such that $p^y_{j'_y} < \kK'$ for all $y \in Y$.
As $(J^n)^{\mc{V}}_{\vc{w}}$ shatters $Z$, 
we can choose $\vc{j} \in (J^n)^{\mc{V}}_{\vc{w}}$
with $j_y = j'_y$ for all $y \in Y$. Note that 
$\mu_{\vc{p}}(\vc{j}) \le \kK^{|Y|} \le \kK^{\lL n/2}$,
so $-\log_2 \mu_{\vc{p}}(\vc{j}) \ge - (\log \kK) \lL n/2 
> |J|n > H(\mu_{\vc{p}})$.
By Lemma \ref{perturb} we can find 
$\nu = (1-t) \mu_{\vc{p}} + t1_{\vc{j}} \in \mc{M}^{\mc{V}}_{\vc{w}}$
with $H(\nu) > H(\mu_{\vc{p}})$. This contradicts the definition 
of $\mu^{\mc{V}}_{\vc{w}}$. \qed

\begin{lemma} \label{bounded.ii}
Let $0 < n^{-1} \ll \kK \ll \gG, \gG' \ll \lL \ll D^{-1}, |J|^{-1}$. 
Suppose $\mc{V} = (\vc{v}^i_j)$ is an $\vc{R}$-bounded $\gG'$-robustly 
$(\gG,\vc{R})$-generic $(n,J)$-array in $\mb{Z}^D$.
Let $\vc{p} = \vc{p}^{\mc{V}}_{\vc{w}}$.
If $\mu_{\vc{p}}$ is $\lL$-dense
then $\mu_{\vc{p}}$ is $\kK$-bounded.
\end{lemma}

\nib{Proof.}
As $\mu_{\vc{p}}$ is $\lL$-dense, we can fix $Y \sub [n]$
with $|Y| \ge \lL n$ such that $p^i_j \ge \lL$
for all $i \in Y$ and $j \in J$.
As $\mc{V}$ is $\gG'$-robustly $(\gG,\vc{R})$-generic,
we can fix $(I,\vc{c})$ with $I \sub Y$
that is $(\gG,\vc{R})$-generic for $\mc{V}$,
i.e.\ all $|c^i_j| \le \gG^{-1}$, $\sum_{j \in J} c^i_j = 0$ for all $i\in I$, 
and writing $\vc{w}^i = \sum_{j \in J} c^i_j \vc{v}^i_j$
and $W = (w^i_d: i \in I, d \in [D])$, 
we have $|\det(W)| \ge \gG \prod_{d \in [D]} R_d$.

Now we show that $\mu_{\vc{p}}$ is $\kK$-bounded.
For suppose on the contrary that $p^{i'}_{j'}<\kK$ 
for some $i' \in [n]$ and $j' \in J$.
Fix $j'' \in J$ such that $p^{i'}_{j''} \ge |J|^{-1}$.
As $(\vc{w}^i: i \in I)$ are linearly independent, we can 
write $\vc{v}^{i'}_{j''}-\vc{v}^{i'}_{j'} = \sum_{i \in I} b_i \vc{w}^i$.
By Cramer's rule we have $b_i = \det(W)^{-1} \det(W_i)$,
where $W_i$ is the matrix obtained from $W$ 
by replacing $\vc{w}^i$ with $\vc{v}^{i'}_{j''}-\vc{v}^{i'}_{j'}$.
As $\mc{V}$ is $\vc{R}$-bounded, we can write
$\det(W_i) = \det(A_i) \prod_{d \in D} R_d$,
where all entries of $A_i$ have modulus at most 1,
so $|\det(A_i)| \le D!$ (or $D^{D/2}$ by Hadamard's inequality).
Therefore $|b_i| \le D! \gG^{-1}$ for all $i \in I$.

Consider a product measure $\mu_{\vc{p}'}$
where for some $t>0$ we have
$p'{}^i_j = p^i_j + tb_i c^i_j$ for all $i \in I$ and $j \in J$,
$p'{}^{i'}_{j'} = p^{i'}_{j'} + t$, $p'{}^{i'}_{j''} = p^{i'}_{j''} - t$,
and $p'{}^i_j = p^i_j$ otherwise.
Note that $\mb{E}_{\vc{x} \sim \mu_{\vc{p}'}} \sum_{i \in [n]} \vc{v}^i_{x_i}
= \mb{E}_{\vc{x} \sim \mu_{\vc{p}}} \sum_{i \in [n]} \vc{v}^i_{x_i}
+ t \sum_{i \in I, j \in J} b_i c^i_j \vc{v}^i_j
+ t \vc{v}^{i'}_{j'} -  t \vc{v}^{i'}_{j''} 
= \mb{E}_{\vc{x} \sim \mu_{\vc{p}}} \sum_{i \in [n]} \vc{v}^i_{x_i}
= \vc{w}$, so $\mu_{\vc{p}'} \in \mc{M}^{\mc{V}}_{\vc{w}}$.

We claim that we can choose $t>0$
such that $H(\mu_{\vc{p}'}) > H(\mu_{\vc{p}})$.
This will contradict the definition of $\mu^{\mc{V}}_{\vc{w}}$, 
showing that $\mu_{\vc{p}}$ is $\kK$-bounded.
Note that $H(\mu_{\vc{p}'}) - H(\mu_{\vc{p}})
= \sum_{i \in I \cup \{i'\}} (H(p'{}^i) - H(p^i))$.
By Lemma \ref{entapprox} $i$ for each $i$ we have
$H(p'{}^i) - H(p^i) \ge - \sum_{j \in J} ((p'{}^i_j-p^i_j)\log_2(p^i_j)
- c_j^{-1} (p'{}^i_j-p^i_j)^2)$, where $c_j = \min \{ p^i_j, p'{}^i_j \}$.
Thus 
\[ H(p'{}^{i'}) - H(p^{i'}) \ge - t\log_2 \kK - 2t^2 - t\log_2 |J| - 2\kK^{-1} t^2,\] 
as $p^{i'}_{j'}<\kK$ and $p^{i'}_{j''} \ge |J|^{-1}$, and 
\[H(p'{}^i) - H(p^i) \ge 2|J| D!\gG^{-2} t \log \lL
- 2\lL^{-1}|J|(tD!\gG^{-2})^2\] 
for $i \in I$, as all $|b_i c^i_j| \leq  D! \gG^{-2}$ and $p^i_j \in (\lL,1-\lL)$.
The dominant term as $t \to 0$ is $- t\log_2 \kK$, so we can choose $t>0$
so that $H(\mu_{\vc{p}'}) - H(\mu_{\vc{p}}) > 0$, as required. \qed

\medskip

\nib{Proof of Theorem \ref{tfae}.}
It remains to prove the implications 
$i \Ra v$ and  $v \Ra iv$
(note that $iv \Ra iii$ is trivial).

For $i \Ra v$, let $n^{-1} \ll \lL \ll \kK \ll \gG, k^{-1}$,
suppose $\mc{V}_n$ is $\gG$-robustly $(\vc{R}_n,k)$-generating,
$\mu^{\mc{V}_n}_{{\bf z}_n}$ is $\kK$-bounded,
${\bf z}'_n \in \mb{Z}^D$ with 
$\|{\bf z}'_n-{\bf z}_n\|_{{\bf R}_n} \le \lL n$, 
and $\mc{V}'$ is obtained from $\mc{V}_n$
by deleting $S \sub [n]$ with $|S| \le \lL n$.
Then ${\cal V}'$ is ${\bf R}_n$-bounded and 
$(\gamma /2)$-robustly $({\bf R}_n,k)$-generating.
Also, the restriction $\mu_{{\bf p}'}$ of 
$\mu^{\mc{V}_n}_{{\bf z}_n}$ to $J^{[n] \sm S}$ 
is $\kK$-bounded, and
${\mathbb E}_{{\bf x} \sim \mu_{{\bf p}'}}
{\cal V}'({\bf x}) = {\bf z}^*$, where 
$\|{\bf z}^* - {\bf z}'_n\|_{\bf R} \leq 2\lL n$.
Therefore $(J^{[n] \sm S})^{\mc{V}'_{n'}}_{{\bf z}'_n} \ne \es$
by Lemma \ref{uctg3}, as required.

For $v \Ra iv$, let $n^{-1} \ll \kK$, and
suppose $(\mc{V}_n,{\bf R}_n,{\bf z}_n)$ is $\kK$-feasible.
Fix $S \sub [n]$ with $|S|=\kK n$ and ${\bf y} \in J^S$.
We need to show that there is 
${\bf x} \in (J^n)^{\mc{V}_n}_{{\bf z}_n}$ 
with ${\bf x}|_S = {\bf y}$.
Let $\mc{V}'$, $\mc{V}^0$ be obtained from $\mc{V}_n$ 
by respectively deleting, retaining the coordinates of $S$.
Let ${\bf z}'_n = {\bf z}_n - \mc{V}^0({\bf y})$.
Then $\|{\bf z}'_n-{\bf z}_n\|_{{\bf R}_n} \le \kK n$,
so by definition of $\kK$-feasibility we can find
${\bf x}' \in (J^{[n] \sm S})^{\mc{V}'_{n'}}_{{\bf z}'_n}$.
Then ${\bf x}={\bf x}'$ is as required. \qed

\medskip

We conclude this section by noting the following lemma
which is immediate from the preceding proof and Lemma \ref{bounded}.

\begin{lemma} \label{universal VC-dimension}
Let $0 < n^{-1} \ll \lambda \ll \gG, \gG' \ll \alpha, D^{-1}, J^{-1}$. 
Suppose $\mc{V} = (\vc{v}^i_j)$ is an $\vc{R}$-bounded, 
$\gamma $-robustly $({\bf R},k)$-generating,
$\gG'$-robustly $(\gG,\vc{R})$-generic $(n,J)$-array in $\mb{Z}^D$.

Let ${\bf w} \in {\mathbb Z}^D$ with 
$\dim _{VC} ((J^n)^{\mc{V}}_{\vc{w}}) \ge \alpha n$. 
Then $\dim _{UVC} ((J^n)^{\mc{V}}_{\vc{w}}) \geq \lambda n$.
\end{lemma}

\section{Counterexamples to Conjecture \ref{kalai}} \label{sec:counterex}

Theorem \ref{thm: Kalai} will precisely describe the conditions under which the conclusion of
Conjecture \ref{kalai} is valid, and so show the existence of counterexamples in most
cases. However, our proof is not constructive, so for expository purposes,
in this section we will present two concrete counterexamples, 
each illustrating a different `breaking point' of Theorem \ref{thm: trichotomy}. 
The first will illustrate case $ii$ 
by showing that we may have $[n]_{k,s} \times _{(t,w)} [n]_{k,s}$ large,
but very few sets in $[n]_{k,s}$ are involved in any $(t,w)$-intersection.
The second will illustrate case $iii$
by showing that we may have almost all sets in $[n]_{k,s}$ 
involved in some $(t,w)$-intersection,
but a large subset of $[n]_{k,s}$ containing no $(t,w)$-intersections.
The first example also shows that cases $ii$ and $iii$
can hold simultaneously.
\vspace{1mm}

\noindent \textbf {Counterexample 1:} Set $\alpha _1 = 1/2$ and $\alpha _2 = {7}/{16}$ and $\beta _1 = 1/4 $ and $\beta _2 = 1/16  + \zeta $, where $\zeta  >0$ is small to be selected. Given $n\in {\mathbb N}$, take $k$, $s$, $t$ and $w$ as in Conjecture \ref{kalai}. 
We will show that $|[n]_{k,s} \times _{(t,w)} [n]_{k,s} | > (1 + c)^n$ for some constant $c>0$ but that there is a set ${\cal A}\subset [n]_{k,s}$ with $|{\cal A}| \geq (1-o(1))\big |[n]_{k,s} \big |$ satisfying ${\cal A} \times _{(t,w)} {\cal A} = \emptyset $.

First we show that $|[n]_{k,s} \times _{(t,w)} [n]_{k,s}|$ is large. 
To see this, we start by finding $C \subset [(1/4 + \zeta ^{1/2})n]$ 
with $(|C|, \sum (C)) = (t, w)$. 
Let $C_0 = [t]$ and note that $\sum (C_0) = \tbinom {t+1}{2} <w$. 
By a sequence of moves, each removing some $i$ and adding $i+1$, 
we can obtain $C_1 = [(1/4 + \zeta ^{1/2})n - t +1, (1/4 + \zeta ^{1/2})n]$, 
with $\sum (C_1) > w$. Clearly some intermediate set has $(C, \sum (C)) = (t,w)$. 
A similar argument gives a set $S \subset [(1/4  + \zeta ^{1/2})n + 1, n]$ 
with $\big (|S|, \sum (S) \big ) = 2(k-t, s-w)$. 
Next we note that the maximum entropy measure $\mu_{\wt{\bf p}}$
on $\{0,1\}^S$ with $\sum_{i \in S} \wt{p}_i (1,i) = (k-t,s-w)$
is the constant vector $(1/2)_{i\in S}$.
By Theorem \ref{ldp} we deduce
$|\{D \subset S: (|D|,\sum (D)) = (k-t,s-w)\}| 
= 2^{(1-o(1))|S|} = 2^{(1/2-o(1))n}$. 
However, for any $D \subset S$ with $(D,\sum (D)) = (k-t, s-w)$ 
we have $(E,\sum (E)) = (k-t,s-w)$, where $E = S \setminus D$. 
Taking $A = C \cup D$ and $B = C \cup E$ we find 
$(A,B) \in [n]_{k,s} \times _{(t,w)}[n]_{k,s}$. 
We have at least as many $(t,w)$-intersections
as choices of $D$, so
$|[n]_{k,s} \times _{(t,w)} [n]_{k,s}| \ge 2^{(1/2 - o(1))n}$.

Next we show that sets in $[n]_{k,s}$ involved in
any $(t,w)$-intersection are very restricted. 
Let $A,B \in [n]_{k,s}$ with $( |A|, \sum (A) ) = ( |B|, \sum (B) ) = 
( \alpha _1n , 	\alpha _2 \tbinom {n}{2} ) = ( n/2, \tfrac{7}{16}\tbinom {n}{2} )$.
Suppose $A$ and $B$ are $(t,w)$-intersecting. Let $\ell := |A \cap B \cap [n/4]|$. Then 
\begin{equation*}
	\Big (\frac{1}{16} + \zeta \Big ) \binom {n}{2}= w = \sum _{i\in A \cap B} i = 
	\sum _{\substack {i\in A \cap B:\\ i< n/4}} i + 
	\sum _{\substack {i\in A \cap B:\\ i\geq n/4}} i 
	\geq 
	\binom {\ell+1}{2} + 
	\Big (\frac{n}{4} - \ell\Big ) \frac{n}{4}.
\end{equation*}
Rearranging gives $(\ell-n/4)^2 - \zZ n^2 + (1/16+\zeta)n+\ell \le 0$,
so $\ell \ge n/4 - \sqrt{\zZ} n$.
In particular, $|A \cap [n/4]| \ge n/4 - \sqrt{\zZ} n$.

We now show that almost all elements of $[n]_{k,s}$ 
do not have this restricted form. 
Fix constants $ \dD_2 \ll \dD_1 \ll \kK \ll 1$.
Let ${\mu }_{\bf p}$ be the maximum entropy measure 
with $\sum p_i (1,i) = (k,s)$. Then ${\mu }_{\bf p}$
is $\kappa $-bounded by Lemma \ref{binary bounded}.
Let ${\cal E} := \{A \subset [n]: \big |A \cap [\tfrac{n}{4}] \big | 
\geq \big (1-\tfrac {\kappa }{2} \big )\tfrac{n}{4} \}$. 
Then ${\mu }_{{\bf p}}({\cal E}) \leq (1-\dD_1)^n$
by Chernoff's inequality, so
$\big |[n]_{k,s} \cap {\cal E} \big |
\leq (1-\delta _2)^n|[n]_{k,s}|$
by Theorem \ref{ldp}. Choosing $\zZ < (\kK/2)^2$,
all $(t,w)$-intersecting pairs from $[n]_{k,s}$ lie within ${\cal E}$, 
which illustrates case $ii$ of Theorem \ref{thm: trichotomy}.
Furthermore, $[n]_{k,s} \setminus {\cal E}$ 
is a set of size $(1-o(1))|[n]_{k,s}|$ containing no $(t,w)$-intersections,
which illustrates case $iii$ of Theorem \ref{thm: trichotomy}.
\vspace{1mm}

\noindent \textbf {Counterexample 2:} 
This counterexample is a modification of the family in Section 1.4, related to VC-dimension. 
Let $\aA_1 = \aA_2 = 2/3$ and $\bB_1 = \bB_2 = 1/3 + \zeta $, where $\zeta  >0$ is small. 
Let $n$, $k$, $s$, $t$ and $w$ be as in Conjecture \ref{kalai}. 
It is not hard to see that 
$|[n]_{k,s} \times _{(t,w)} [n]_{k,s}| 
= (1-o_{\zeta }(1))^n \tbinom {n}{t, k-t, k-t, n-2k +t} 
= (1-o_{\zeta }(1))^n 3^n$ 
and almost all elements of $[n]_{k,s}$ are involved in a $(t,w)$-intersection. 
However, for any set $U \subset [n]$ with $|U| = 2\zeta n + 1$, 
taking ${\cal A}_U = \{A \in [n]_{k,s}:  A \cap U = \emptyset \}$, 
we have $|A \cap B| \geq t+1$ for all $A,B \in {\cal A}_U$ and
so ${\cal A}_U \times _{(t,w)} {\cal A}_U = \emptyset $. 
On the other hand, if we select such a set $U$ 
uniformly at random we find 
${\mathbb E}_U(|{\cal A}_U|) \geq (1-o_{\zeta }(1))^n |[n]_{k,s}|$. 
Thus for some $U$ we have $|{\cal A}_U| \geq (1-o_{\zeta }(1))^n |[n]_{k,s}|$ 
and ${\cal A}_U \times _{(t,w)} {\cal A}_U = \emptyset $.

\section{The general setting} \label{sec:gen}

In this section we state our most general result, Theorem \ref{general};
we will defer the proof to section \ref{sec:pf}.
This is in fact the main result of the paper in some sense,
as we will show in this section that it implies Theorem \ref{binary} 
(in a more general cross-intersection form).
However, the hypothesis of `transfers' in Theorem \ref{general}
appears to be quite strong at first sight, and it will take some
work to show that it follows from the hypotheses of Theorem \ref{binary}
(it is here that the idea of enlarging the alphabet comes into play).
We state our result in the next subsection and then
deduce Theorem \ref{binary} in the following subsection. A second application of Theorem \ref{general} is given in subsection \ref{subsec: FR-patterns}, where we use it to give a short proof of a theorem of Frankl and R\"odl on forbidden intersection patterns.
 
\subsection{Statement of the general theorem}

Before stating our theorem, we require the following definition,
which describes a situation when for any vector $\vc{u}$
in some specific set (which will be given by the following definition),
there are many ways of choosing a coordinate 
and two particular alterations of its value:
one does not change the associated vector,
and the other changes it by $\vc{u}$.

\begin{dfn} \label{transfer}
Suppose $\mc{V} = (\vc{v}^i_{j,\ell})$ is an $(n,J \times L)$-array in $\mb{Z}^D$.
We say that $\vc{u}$ is an $i$-transfer in $\mc{V}$
(via $(j,j')$ and $(\ell,\ell')$)
if there are $j,j'$ in $J$ and $\ell,\ell'$ in $L$
with $\vc{v}^i_{j,\ell}-\vc{v}^i_{j',\ell}=\vc{u}$
and $\vc{v}^i_{j',\ell'}=\vc{v}^i_{j,\ell'}$.

Let $\mc{U}=\{\vc{u}_1,\dots,\vc{u}_M\} \sub \mb{Z}^D$
and $\mc{P}=(P_m: m \in [M])$ for some disjoint subsets $P_m$ of $[n]$.
We say that $\mc{V}$ has transfers for $(\mc{P},\ \mc{U})$
if $\vc{u}_m$ is an $i$-transfer in $\mc{V}$
for each $m \in M$ and $i \in P_m$.

We say that $\mc{V}$ has $\gG$-robust transfers for $\mc{U}$
if it has transfers for $(\mc{P},\ \mc{U})$ for some $\mc{P}$
such that $|P_m| \ge \gG n$ for all $m \in [M]$.
\end{dfn}

Note that an $(n,\prod_{s \in S} J_s)$-array in $\mb{Z}^D$ has transfers 
for $(\mc{P},\ \mc{U})$ if it has them as an $(n,J \times L)$-array,
where $J = \prod_{s \in S'} J_s$ and $L = \prod_{s \in S \sm S'} J_s$ for some $S' \sub S$.

We can now state our general theorem.
(Recall that $\mc{U}$ exists by Lemma \ref{u}.)

\begin{theo} \label{general}
Let $0 < n^{-1} \ll \dD \ll \zZ \ll \eps, \kK, \gG \ll D^{-1}, M^{-1}, k^{-1}, C^{-1}$.
Let $S$ and $(J_s: s \in S)$ be sets of size at most $C$,
and $\vc{R}=(R_1,\dots,R_D)$ with $\max_d R_d < n^C$. Suppose
\begin{enumerate}
\item $\mu_{\vc{q}}$ is a $\kK$-bounded product measure 
on $(\prod_{s \in S} J_s)^n$ with marginals $(\mu_{\vc{p}_s}: s \in S)$,
\item $\mc{V} = (\vc{v}^i_{j_1,\dots,j_S})$ 
is an $\vc{R}$-bounded $(n,\prod_{s \in S} J_s)$-array in $\mb{Z}^D$,
\item $\mc{U}=\{\vc{u}_1,\dots,\vc{u}_M\} \sub \mb{Z}^D$ 
is $\vc{R}$-bounded and $(k,k\zZ n,\vc{R})$-generating, 
\item $\mc{V}$ has $\gG$-robust transfers for $\mc{U}$,
\item $\vc{w} \in \mb{Z}^D$ with $\|\vc{w}-\mc{V}(\mu_{\bf q})\|_{\vc{R}} < \zZ n$. 
\end{enumerate}
Suppose $\mc{A}_s \sub J_s^n$ for $s \in S$ with 
$\prod_{s \in S} \mu_{\vc{p}_s}(\mc{A}_s) > (1-\dD)^n$. Then $\mu_{\vc{q}}((\prod_{s \in S} \mc{A}_s)^{\mc{V}}_{\vc{w}}) > (1-\eps)^n$.
\end{theo}

\subsection{Proof of Theorem \ref{binary}}

Now we assume Theorem \ref{general} and  prove Theorem \ref{binary};
in fact we prove the more general cross-intersection theorem.
The strategy is to fuse together suitable co-ordinates and enlarge the alphabet.

\begin{theo} \label{xbinary}
Let $0 < n^{-1}, \dD \ll \zZ \ll \kK, \gG, \eps \ll D^{-1}, C^{-1}, k^{-1}$
and $\vc{R}=(R_1,\dots,R_D)$ with $\max_d R_d < n^C$. Suppose
\begin{enumerate}
\item $\mu_{\vc{q}}$ is a $\kK$-bounded product measure 
on $(\{0,1\} \times \{0,1\})^n$ with marginals $(\mu_{\vc{p}_1}, \mu_{\vc{p}_2})$,
\item $\mc{V} = (\vc{v}_i: i \in [n])$ is $\vc{R}$-bounded
 and $\gG$-robustly $(\vc{R},k)$-generating in $\mb{Z}^D$,
\item $\vc{w} \in \mb{Z}^D$ with $\|\vc{w}-\mc{V}_\cap(\mu_{\bf q})\|_{\vc{R}} < \zZ n$.
\end{enumerate}
Then any $\mc{A}_1, \mc{A}_2 \sub \{0,1\}^n$ with 
$\mu_{\vc{p}_1}(\mc{A}_1) \mu_{\vc{p}_2}(\mc{A}_2) > (1-\dD)^n$ satisfy $\mu_{\vc{q}}((\mc{A}_1 \times \mc{A}_2)^{\mc{V}_\cap}_{\vc{w}}) > (1-\eps)^n$.
\end{theo}

\nib{Proof.} 
By Lemma \ref{u} we can fix some $(1,\zZ n,\vc{R})$-generating $\vc{R}$-bounded
$\mc{U} =\{\vc{u}_1,\dots,\vc{u}_M\} \sub \mb{Z}^D$ with $M \le D(C+2)$.
By repeatedly applying Definition \ref{robgen}, we can choose pairwise disjoint 
$S_{mj}, S'_{mj} \sub [n]$ for each $m \in [M]$ and $j \in [\gG n/kM]$
with each $|S_{mj}|+|S'_{mj}| \le k$ and 
$\vc{u}_m = \sum_{i \in S_{mj}} \vc{v}_i - \sum_{i \in S'_{mj}} \vc{v}_i$.
We let $N=\bfl{n/k}$ and partition $[n]$ into sets $T_1,\dots,T_N$
each of size $k$ and a remainder set $R$ with $0 \le |R| \le k-1$,
such that each $S_{mj} \cup S'_{mj}$ is contained in some $T_i$.
We let $\mc{P}=(P_m: m \in [M])$, where each $P_m$ is the set of $i \in [N]$ 
such that $T_i$ contains some $S_{mj} \cup S'_{mj}$.

We start by reducing to the case $R=\es$ and $k| n$.
For $R_s \sub R$ we let 
$\mc{A}_s^{R_s} = \{ A \in \mc{A}_s: A \cap R = R_s\}$ for $s=1,2$.
By the pigeonhole principle we can fix $(R_1,R_2)$ so that 
$\mu_{\vc{p}_1}(\mc{A}_1^{R_1}) \mu_{\vc{p}_2}(\mc{A}_2^{R_2}) 
> 2^{-2K} (1-\dD)^n > (1-2\dD)^n$.
Let $\mc{A}'_s = \{ A_s \sm R_s: A \in \mc{A}_s^{R_s}\}$
and $\mc{V}' = (\vc{v}_i: i \in [n] \sm R')$.
Note that for $A_s \in \mc{A}'_s$ we have $A_s \cup R_s \in \mc{A}_s$ with
$\mc{V}_\cap(A_1 \cup R_1,A_2 \cup R_2) = \mc{V}'_\cap(A_1,A_2) + \vc{v}'$,
where $\vc{v}' = \sum_{i \in R_1 \cap R_2} \vc{v}_i$.
Writing $\vc{w}'=\vc{w}-\vc{v}'$, we have
$\mu_{\vc{q}}((\mc{A}_1 \times \mc{A}_2)^{\mc{V}_\cap}_{\vc{w}})
> \kK^k \mu_{\vc{q}}((\mc{A}'_1 \times \mc{A}'_2)^{\mc{V}'_\cap}_{\vc{w}'})$,
so to prove the theorem it suffices to show
$\mu_{\vc{q}}((\mc{A}'_1 \times \mc{A}'_2)^{\mc{V}'_\cap}_{\vc{w}'}) > (1-\eps/2)^n$.
 
We can naturally identify $\{0,1\}^n$ with $(\{0,1\}^k)^N$,
where $A \in \{0,1\}^n$ corresponds to $(A \cap T_i: i \in [N])$
according to some fixed bijection of $T_i$ with $[k]$.
We will apply Theorem \ref{general} with $N$ in place of $n$,
with $S=\{1,2\}$ and $J_1=J_2=\{0,1\}^k$, and 
$\mc{A}'_s$ (naturally identified) in place of $\mc{A}_s$.
We let $\mc{W} = (\vc{w}^i_{J_1,J_2})$ be the $(N,\{0,1\}^k \times \{0,1\}^k)$-array in $\mb{Z}^D$
defined by $\vc{w}^i_{J_1,J_2} = \sum_{j \in J_1 \cap J_2} \vc{v}_j$ for $J_1,J_2 \sub T_i$.
Note that $\mc{V}_\cap(\vc{x},\vc{y}) = \mc{W}(\vc{x},\vc{y})$
for all $\vc{x},\vc{y}$ in $\{0,1\}^n$ (naturally identified).

We also note that $\mc{W}$ has transfers for $(\mc{U}, \mc{P})$.
To see this, consider $i \in P_m$ with $S_{mj} \cup S'_{mj} \sub T_i$.
Let $J = S_{mj}$, $J' = S'_{mj}$, $L = S_{mj} \cup S'_{mj}$ and $L' = \es$.
Then $\vc{w}^i_{J,L'}=\vc{w}^i_{J',L'}=0$ and
$\vc{w}^i_{J,L}-\vc{w}^i_{J',L} = \sum_{i \in S} \vc{v}_i - \sum_{i \in S'} \vc{v}_i = \vc{u}_m$.

We let $\mu_{\vc{q}'}$ be the corresponding product measure on 
$(\{0,1\}^k \times \{0,1\}^k)^N$, defined by 
$q'{}^i_{\vc{j}^1,\vc{j}^2} = \prod_{i' \in T_i} q^{i'}_{j^1_{i'},j^2_{i'}}$
for $\vc{j}^1$ and $\vc{j}^2$ in $\{0,1\}^k$,
noting that $\mu_{\vc{q}'}$ is $\kK^k$-bounded,
and let $(\mu_{\vc{p}'_1},\mu_{\vc{p}'_2})$ be its marginals on $(\{0,1\}^k)^N$.
By construction we have $\mu_{\vc{p}'_s}(\vc{x}) = \mu_{\vc{p}_s}(\vc{x})$
and $\mu_{\vc{q}'}(\vc{x},\vc{y}) = \mu_{\vc{q}}(\vc{x},\vc{y})$
for all $\vc{x},\vc{y}$ in $\{0,1\}^n$ (naturally identified).

To summarise, after the above reductions, we have 
$\mu_{\vc{p}'_1}(\mc{A}'_1) \mu_{\vc{p}'_2}(\mc{A}'_2) > (1-2\dD)^n$, and it suffices to show
$\mu_{\vc{q}'}((\mc{A}'_1 \times \mc{A}'_2)^{\mc{W}}_{\vc{w}'}) > (1-\eps/2)^n$.
For $\vc{r} \sim \mu_{\vc{q}'}$ we have $\|\vc{w}'-\mb{E} \mc{W}(\vc{r}')\|_{\vc{R}}
\le \|\vc{w}-\mb{E} \mc{V}_\cap(\vc{r})\|_{\vc{R}} + 2|R| < 2\zZ n$,
so the theorem follows from Theorem \ref{general}. \qed

\subsection{Application to a theorem of Frankl and R\"odl}\label{subsec: FR-patterns}

In this subsection we give another application of Theorem \ref{general}, 
which illustrates an additional flexibility, namely that our method allows
different vectors defining the sizes of intersections
from those defining the sizes of sets in the family.
We will give a  new proof of a theorem of 
Frankl and R\"odl \cite[Theorem 1.15]{FrRo}
on intersection patterns in sequence spaces.
(To align with notation from the rest of the paper, 
our notation differs from that of \cite{FrRo}.)

Given non-negative integers $l_1,\ldots ,l_s$ with $\sum _{i} l_i = n$, let $\tbinom {[n]}{l_1,\ldots ,l_s}$ denote the set of elements ${\bf x} \in [s]^n$ with $|\{i\in [n]: x_i = j\}| = l_j$ for all $j\in [s]$. Given ${\bf x} \in \tbinom {[n]}{l_1,\ldots ,l_s}$ and ${\bf y} \in \tbinom {[n]}{k_1,\ldots ,k_t}$, the intersection pattern of ${\bf x}$ and ${\bf y}$ is given by an $s$ times $t$ matrix $M$, with $M_{j_1,j_2} = |\{i\in [n]: x_i = j_1, y_i = j_2\}|$ for $(j_1,j_2) \in [s]\times [t]$. For ${\cal A}_1 \subset \tbinom {[n]}{l_1,\ldots ,l_s}$ and ${\cal A}_2 \subset \tbinom {[n]}{k_1,\ldots ,k_t}$ we let ${\cal A}_1 \times _M {\cal A}_2$ denote the set of pairs $({\bf x},{\bf y}) \in {\cal A}_1 \times {\cal A}_2$ with intersection pattern $M$.

We say that $M$ is an intersection pattern for $(l_1,\ldots ,l_s)$ and $(k_1,\ldots,k_t)$
if each $\sum _{j_2 \in [t]} M_{j_1,j_2} = k_{j_1}$,
each $\sum _{j_1\in [s]} M_{j_1,j_2} = l_{j_2}$,
and $\sum _{(j_1,j_2) \in [s]\times [t]} M_{j_1,j_2} = n$.
The following result of Frankl and R\"odl 
is the analogue of Theorem \ref{FranklRodl}
for intersection patterns.

\begin{theo}[Frankl-R\"odl]\label{FRintpatt}
 Given $\eps , \kappa >0$ and $s,t \in {\mathbb N}$ 
 there is $\delta >0$ such that the following holds. 
 Suppose that $M$ is an intersection pattern 
 for $(l_1,\ldots ,l_s)$ and $(k_1,\ldots,k_t)$
 with all $M_{j_1,j_2} \geq \kappa n$.
 Let ${\cal A}_1 \subset \tbinom {[n]}{l_1,\ldots ,l_s}$ 
 with $|{\cal A}_1| \geq (1-\delta )^n \tbinom {n}{l_1,\ldots ,l_s}$ 
 and ${\cal A}_2 \subset \tbinom {[n]}{k_1,\ldots ,k_t}$ 
 with $|{\cal A}_2| \geq (1-\delta )^n \tbinom {n}{k_1,\ldots ,k_t}$.
 Then $|{\cal A}_1 \times _{M} {\cal A}_2| \geq 
(1-\eps )^n |\tbinom {n}{l_1,\ldots ,l_s} \times _M \tbinom {n}{k_1,\ldots ,k_t}|$.
\end{theo}

\begin{proof}
	Fix $0 < \delta \ll \delta ' \ll \eps ' \ll \eps ,\kappa $,
	and let $J_1 = [s]$. 
Let ${\bf e}_1, \ldots ,{\bf e}_s$ denote the standard basis for ${\mathbb Z}^s$, 
and let ${\cal V}_1 = ({\bf v}^{i}_j)$ denote the 
	$(n,J_1)$-array, where each ${\bf v}^i_j = {\bf e}_j$.
We can naturally identify $\tbinom {[n]}{l_1,\ldots, l_{s}}$
with $(J_1^n)^{{\cal V}_1}_{{\bf z}_1}$,
where ${\bf z}_1 = (l_1,\ldots, l_s) \in {\mathbb Z}^s$.
	The maximum entropy measure 
	$\mu _{{\bf p}_1} = \mu ^{{\cal V}_1}_{{\bf z}_1}$ on $J_1^n$ 
	is then given by $(p_1)^{i}_j = l_j/n$ for all $i\in [n]$ 
	and $j\in J_1$. Indeed, as ${\bf v}^{i}_j$ is independent 
	of $i\in [n]$, by strict concavity of entropy 
	(Lemma \ref{entapprox}) so is $(p_1)^{i}_j = p_{1,j}$, 
	and $n(p_{1,1},\ldots , p_{1,s}) = 
	{\mathbb E}{\cal V}_1({\bf x}) = (l_1,\ldots ,l_s)$. 
	As $\mu _{{\bf p}_1}$ is $\kappa $-bounded 
	we can apply Theorem \ref{ldp} to find 
	${\mu }_{{\bf p}_1} \approx_{\triangle _1} \nu _1$, where 
	$\nu _1$ is uniform measure on 
	$(J_1^n)^{{\cal V}_1}_{{\bf z}_1} = \tbinom {[n]}{l_1,\ldots, l_{s}}$. 
Similarly, taking $J_2 = [t]$, we have a $\kappa $-bounded product 
	measure $\mu _{{\bf p}_2}$ on $J_2^n$, with 
	$\mu _{{\bf p}_2} \approx _{\triangle _2} \nu _2$, where 
	$\nu _2$ is uniform measure on 
	$\tbinom {[n]}{k_1,\ldots ,k_t}$.
Therefore $\mu _{{\bf p}_i}({\cal A}_i) \geq (1-\delta ')^n$ for $i = 1,2$.
	
Similarly, we let ${\bf e}_{1,1}, \ldots ,{\bf e}_{s-1,t-1}$ 
denote the standard basis for ${\mathbb Z}^{(s-1)(t-1)}$,
and let ${\cal V} = ({\bf v}^i_{j_1,j_2})$ denote the $(n, J_1 \times J_2)$-array, 
where ${\bf v}^{i}_{j_1,j_2} = {\bf e}_{j_1,j_2}$ if $(j_1,j_2) \in [s-1] \times [t-1]$ 
and $\bf 0$ otherwise.	We also let  
${\bf w} = \sum _{(j_1,j_2) \in [s-1] \times [t -1]}
M_{j_1,j_2}{\bf e}_{j_1,j_2}$.
Note that for ${\bf x} \in \tbinom {n}{l_1,\ldots ,l_s}$
and ${\bf y} \in \tbinom {n}{k_1,\ldots ,k_t}$, 
we have ${\cal V}({\bf x},{\bf y}) = {\bf w}$ if and
only if ${\bf x}$ and ${\bf y}$ have intersection pattern $M$. 
Therefore $({\cal A}_1 \times {\cal A}_2)^{\cal V}_{\bf w} 
= {\cal A}_1 \times _{M} {\cal A}_2$.

We will apply Theorem \ref{general} to estimate
$({\cal A}_1 \times {\cal A}_2)^{\cal V}_{\bf w}$
under the product measure $\mu _{\bf q}$ on $(J_1 \times J_2)^n$
defined by $q^i_{j_1,j_2} = M_{j_1,j_2}/n$.
	By hypothesis, $\mu _{\bf q}$ is $\kappa $-bounded, 
	with marginals $\mu _{{\bf p}_1}$ and $\mu _{{\bf p}_2}$, and 
	${\mathbb E}_{({\bf x},{\bf y}) \sim \mu _{\bf q}}
	{\cal V}({\bf x},{\bf y}) = {\bf w}$. Taking $\bf R$ to 
	be the constant ${\bf 1}$ vector in 
	${\mathbb Z}^{(s-1)(t-1)}$ we see that 
	 ${\cal V}$ is ${\bf R}$-bounded, and 
	${\cal U} = \{{\bf e}_{j_1,j_2}\}$ is 
	$(st,0,{\bf R})$-generating. Lastly, 
	${\cal V}$ has $1$-robust transfers for ${\cal U}$, 
	as for any $(j_1,j_2) \in [s-1] \times [t-1]$ 
	we have ${\bf v}^i_{j_1,j_2} - {\bf v}^i_{s,j_2} 
	= {\bf e}_{j_1,j_2}$ and ${\bf v}^i_{j_1,t} - 
	{\bf v}^i_{s,t} = {\bf 0}$. 
As $\mu _{{\bf p}_i}({\cal A}_i) \geq (1-\delta ')^n$ for $i = 1,2$, 
Theorem \ref{general} gives 
	$\mu _{\bf q} ( ({\cal A}_1 \times {\cal A}_2)^{\cal V}_{\bf w}) 
 = \mu _{\bf q} ( {\cal A}_1 \times _M{\cal A}_2) \geq (1-\eps ')^n$. 
The theorem follows from a final application of Theorem \ref{ldp}.
\end{proof}

We wish to emphasize two aspects of the above proof. 
Firstly, it is crucial that the arrays
${\cal V}_1, {\cal V}_2$ and ${\cal V}$ can differ.
Secondly, the arrays ${\cal V}_i$ are not $|J_i|^{-1}$-robustly
$(\gamma, {\bf R})$-generic for any $\gamma >0$ for $i = 1,2$, 
so we cannot apply Lemma \ref{bounded},
but we were able to see directly that
$\mu _{{\bf p}_1}$ and ${\mu }_{{\bf p}_2}$ are $\kappa $-bounded. 
Thus Theorem \ref{general} has useful consequences 
even for arrays that are not robustly generic.

\section{Correlation on product sets} \label{sec:cor}

In this section we will prove the following correlation inequality
which will be used in the proof of Theorem \ref{general};
it can also be interpreted as an exponential contiguity
result for product measures (see Theorem \ref{corctg}).

\begin{theo} \label{gencor} 
Let $0 < n^{-1} , \dD \ll \kK, \eps < 1$ 
and $\mu_{\vc{q}}$ be a $\kK$-bounded product measure
on $(\prod_{s \in S} J_s)^n$ with marginals $(\mu_{\vc{p}_s}: s \in S)$.
Suppose $\mc{A}_s \sub J_s^n$ for $s \in S$ 
with $\prod_{s \in S} \mu_{\vc{p}_s}(\mc{A}_s) > (1-\dD)^n$.
Then $\mu_{\vc{q}}(\prod_{s \in S} \mc{A}_s) > (1-\eps)^n$.
\end{theo}

\begin{proof}[Proof of Theorem \ref{gencor}]
	We first consider the case $S = \{1,2\}$. 
	Define $f:(J_1)^n \to {\mathbb R}$ by 
	$$f({\bf x}_1) = \log _e\big ({\mu _{{\bf p}_1}({\bf x}_1)} \big ) - 
	\log _e\big ({\mu }_{\bf q} (\{{\bf x}_1\} \times  {\cal A}_2) \big ).$$ 
	As ${\mu }_{{\bf p}_1}$ is a marginal of ${\mu }_{\bf q}$,
	we have ${\mu _{{\bf p}_1}({\bf x}_1)} 
\geq {\mu }_{\bf q} (\{{\bf x}_1\} \times  {\cal A}_2)$  
	for all ${\bf x}_1 \in (J_1)^n$, and so $f({\bf x}_1) \geq 0$ for all 
	${\bf x}_1 \in (J_1)^n$. Note also that $f$ is 
	$2\log ({\kappa }^{-1})$-Lipschitz, as $\mu _{\bf q}$ is ${\kappa }$-bounded.
		
Let $M = {\mathbb E}_{\mu _{\bf q}}(f)$. We claim that 
$M \leq (2\delta + \alpha )n \leq 2\alpha n$. 
To see this, we apply a well-known concentration argument.
	For $I \subset {\mathbb R}$, let 
\[{\cal B}_{I} = \{ {\bf x}_1 \in (J_1)^n: f({\bf x}_1) \in I \}.\] 
By Lemma \ref{lip}, letting $\alpha  = 4\delta^{1/2} \log \kappa^{-1}$,
we have 
	$\mu _{{\bf p}_1}({\cal B}_{[M - \alpha n, M + \alpha n]}) 
	\geq 1 - 2e^{{-\alpha ^2n}/(8 \log^2 (\kappa ^{-1}))} > 1 - (1-\delta )^n/2$.
Now let 
\[{\cal C} = \{{\bf x}_1: \mu _{\bf q}(\{{\bf x}_1\} \times{\cal A}_2) 
\geq (1-\delta )^n \mu _{{\bf p}_1}({\bf x}_1)/2\}.\] 
	Then $f({\bf x}_1) \leq 2\delta n$ for ${\bf x}_1 \in {\cal C}$, 
	so ${\cal C} \subset {\cal B}_{[0, 2\delta n]}$. However,
	$(1-\delta )^n \leq \mu _{{\bf p}_2}({\cal A}_2) 
	= \mu _{{\bf q}} ((J_1)^n \times {\cal A}_2) \leq \mu _{{\bf p}_1}({\cal C}^{c}) 
	(1-\delta )^n/2 + \mu _{{\bf p}_1} ({\cal C})$, and so 
$\mu _{{\bf p}_1}({\cal B}_{[0,2\delta n]}) 
	\geq \mu _{{\bf p}_1} ({\cal C}) \geq (1-\delta )^n/2$. Thus
	${\cal B}_{[0,2\delta n]} \cap {\cal B}_{[M-\alpha n, M+\alpha n]} 
	\neq \emptyset $, which gives 
$M \leq (2\delta + \alpha )n \leq 2\alpha n$, as claimed.
	
	Now set ${\cal B} = {\cal A}_1 \cap {\cal B}_{[0,3\alpha n]}$. As $\mu _{{\bf p}_1} ({\cal A}_1) \geq (1-\delta )^n$ and $\mu _{{\bf p}_1} ({\cal B}_{[0,3\alpha n]}) \geq \mu _{{\bf p}_1}({\cal B}_{[M - \alpha n, M + \alpha n]}) \geq 1 - (1-\delta )^n/2$ we have $\mu _{{\bf p}_1}({\cal B}) \geq (1-\delta )^n/2$. Therefore 
		\begin{align*}
			\mu _{\bf q}({\cal A}_1 \times {\cal A}_2) 
				\geq 
			\sum _{{\bf x}_1 \in {\cal B}} 
				\mu _{\bf q} (\{{\bf x}_1\} \times  {\cal A}_2) 
				= 
			\sum _{{\bf x}_1 \in {\cal B}} 
				\mu _{{\bf p}_1} ({\bf x}_1) e^{-f({\bf x}_1)} 
				\geq 
			\mu _{{\bf p}_1} ({\cal B}) e^{-3\alpha n} 
				\ge (1-\eps )^n.
		\end{align*}
	This completes the proof in this case.
	
	Now we deduce the general case by induction on $|S|$. 
	Suppose the theorem is known for $|S| = k-1$ and we wish to prove it for 
	$|S| = k$. Fix $s \in S$ and let $S' = S \setminus \{s\}$.
We view $(\prod _{s \in S} J_s)^n$ as $(J_s \times J')^n$,
where $J' = \prod _{s' \in S'} J_{s'}$.
Let $\mu _{{\bf p}'}$ be the product measure on $J'{}^n$
defined by $\mu _{{\bf p}_{S'}}({\bf x}') 
= \mu _{\bf q}((J_1)^n \times \{{\bf x}'\})$.
Then $\mu _{{\bf p}'}$ is $\kappa $-bounded 
and has marginals $(\mu _{{\bf p}_{s'}})_{s' \in S'}$,
so by induction hypothesis, as 
$\prod _{s' \in S'} \mu _{{\bf p}_{s'}}({\cal A}_{s'}) \geq (1-\delta )^n$ 
we have $\mu _{{\bf p}_{S'}}(\prod _{s'\in S'} {\cal A}_{s'}) \geq (1-\dD')^n$,
where $\dD \ll \dD' \ll \eps$.

Also, we can view $\mu _{\bf q}$ as a product measure on $(J_s \times J')^n$,
with marginals $\mu _{{\bf p}_s}$ and $\mu _{{\bf p}'}$. 
Since $\mu _{{\bf p}_s}({\cal A}_s) 
\mu _{{\bf p}'} (\prod _{s' \in S'} {\cal A}_{s'}) 
\geq (1-\delta )^n(1-\dD')^n \ge (1-2\dD')^n$,
from the $|S| =2$ case of the theorem we obtain 
$\mu _{\bf q}(\prod _{s \in S} {\cal A}_s) \geq (1-\eps )^n$, as required.
\end{proof}	

Next we will apply Theorem \ref{gencor} to show exponential contiguity
of $\mu_{\vc{q}}$ and $\prod_{s \in S} \mu_{\vc{p}_s}$,
defined by $(\prod_{s \in S} \mu_{\vc{p}_s})(\vc{x}_s: s \in S)
= \prod_{s \in S} \mu_{\vc{p}_s}(\vc{x}_s)$. Here the subscript $\Pi$ 
indicates exponential contiguity relative to product sets,
i.e.\ we apply Definition \ref{expctg} in the case
$\OO_n = \big (\prod_{s \in S} J_{s}\big )^n$ and $\mc{F} = \Pi = (\Pi_n)_{n \in \mb{N}}$,
where $\Pi_n = \{ (\mc {A}_{n,s}: s \in S): \text{ all } \mc {A}_{n,s} \in J_{s}^n \}$.

\begin{theo} \label{corctg} 
Let $0 < n^{-1} \ll \kK \ll 1$ 
and $\mu_{\vc{q}}$ be a $\kK$-bounded product measure
on $(\prod_{s \in S} J_s)^n$ with marginals $(\mu_{\vc{p}_s}: s \in S)$.
Then $\mu_{\vc{q}} \approx_\Pi \prod_{s \in S} \mu_{\vc{p}_s}$.
\end{theo}

\nib{Proof.}
As in the proof of Theorem \ref{gencor}, 
it suffices to consider the case $S=[2]$. By Theorem \ref{gencor} 
we have $\mu_{\vc{p}_1} \times \mu_{\vc{p}_2} \lesssim_\Pi \mu_{\vc{q}}$.
Conversely, consider $\mc{A}_s \sub J_s^n$ for $s \in [2]$.
By the Cauchy-Schwarz inequality, writing $\sum$ for
$\sum_{\vc{x}_1 \in J_1^n, \vc{x}_2 \in J_2^n}$, we have
\[\mu_{\vc{q}}(\mc{A}_1 \times \mc{A}_2)^2 
= \Big( \sum \mu_{\vc{q}}(\vc{x}_1,\vc{x}_2)
  \prod_{s \in [2]} 1_{\vc{x}_s \in \mc{A}_s} \Big)^2
\le \prod_{s \in [2]} \sum  \mu_{\vc{q}}(\vc{x}_1,\vc{x}_2) 1_{\vc{x}_s \in \mc{A}_s}
= \prod_{s \in [2]} \mu_{\vc{p}_s}(\mc{A}_s),\]
so $\mu_{\vc{q}} \lesssim_\Pi \mu_{\vc{p}_1} \times \mu_{\vc{p}_2}$. \qed
\medskip

We conclude this section by giving the easy deduction
of Theorem \ref{binary approx} from Theorem \ref{gencor}.

\begin{proof}[Proof of Theorem \ref{binary approx}]
	Given $\zeta >0$, taking ${\cal D} = \{ {\bf r}:\|{\cal V}({\bf r}) - {\mathbb E}{\cal V}\|_{\bf R} 
	\geq \zeta n \} \subset (\{0,1\} \times \{0,1\})^n$, by Lemma \ref{vchern} we have 
	$\mu _{\bf q}\big ( {\cal D}\big ) \leq 2D e^{-{\zeta }^2n/8} \leq e^{-{\zeta }^2n/16}$. 
	However, provided 
	$\delta , n^{-1} \ll \zeta ,\eps ,\kappa $, by Theorem \ref{gencor} any ${\cal A} \subset \{0,1\}^n$ with 
	$\mu _{\bf p}({\cal A}) \geq (1-\delta )^n$ satisfies
	$\mu _{\bf q} ({\cal A} \times {\cal A}) \geq  e^{-{\zeta }^2n/16}+ 
	(1-\eps )^n$. Since $({\cal A} \times {\cal A}) \cap {\cal D}^c = 
	({\cal A} \times {\cal A})^{{\cal V}_{\cap }}_{L}$ the result 
	follows.
\end{proof}

\section{Proof of the general theorem} \label{sec:pf}

In this section we prove Theorem \ref{general}.
We start by reducing to the case $|S|=2$.

\begin{lemma}
Theorem \ref{general} follows from the case $|S|=2$.
\end{lemma}

\nib{Proof.}
First note that if $\mc{V}$ has $\gG$-robust transfers for $\mc{U}$
then it has them as an $(n,L_1 \times L_2)$-array, where each
$L_j = \prod_{s \in S_j} J_s$ for some partition $(S_1,S_2)$ of $S$. 

Now let $\mu _{{\bf p}_{S_1}}$ denote the 
product measure on $L_1^n$
defined by $\mu _{{\bf p}_{S_1}}({\bf x}_1) = 
\mu _{\bf q} (\{{\bf x}_1\} \times L_2^n)$;
then  $\mu _{{\bf p}_{S_1}}$ is $\kappa $-bounded.
Similarly, we obtain  
$\mu _{{\bf p}_{S_2}}$ on $L_2^n$
that is $\kappa $-bounded.

Let $\mc{A}_{S_j} = \prod_{s \in S_j} \mc{A}_s$ for $j=1,2$.
As $\prod _{s \in S_j} \mu _{{\bf p}_s} ({\cal A}_s) \geq (1-\delta )^n$ 
and $\delta \ll \delta '$, by Theorem \ref{gencor} each
$\mu _{{\bf p}_{S_j}}({\cal A}_{S_j}) \geq (1-\delta ')^n$, 
so the case $S = \{1,2\}$ of Theorem \ref{general}
applies to $(\mc{A}_{S_1},\mc{A}_{S_2})$. \qed
\vspace{2mm}
 
Now we will prove a succession of special cases of Theorem \ref{general},
where the proof of each case builds on the previous cases,
culminating in the proof of the general case.
We assume without further comment that $S=\{1,2\}$.

\begin{lemma} \label{2x2=hs0}
Suppose the assumptions of Theorem \ref{general} hold,
we also have $J_1=J_2=\{0,1\}$, $\mc{A}_1=\mc{A}_2=\mc{A}$, 
$q^i_{j_1,j_2}=q_{j_1,j_2}$ for all $i \in [n]$ and $q_{0,1}=q_{1,0}=q_{1,1}=\aA$, 
and $\mc{V}$ has transfers for $(\mc{P},\ \mc{U})$, 
where $\mc{P}=(P_m: m \in [M])$ is a partition of $[n]$.
Then $(\mc{A} \times \mc{A})^{\mc{V}}_{\vc{w}} \ne \es$.
\end{lemma}

\nib{Proof.}
The idea of the proof is to reduce the required statement
to finding two sets $A$ and $B$ in $\mc{A}$ with prescribed values 
of $|A \cap B \cap P_m|$ for all $m \in [M]$,
where we identify $\{0,1\}^n$ with subsets of $[n]$;
this will be achieved by the Frankl-R\"odl theorem 
and Dependent Random Choice.

First we introduce some notation. We write
$\mu_{\vc{p}} = \mu_{\vc{p}_1} = \mu_{\vc{p}_2}$,
and note for all $i \in [n]$ that 
$p^i_0 = 1-2\aA$ and $p^i_1 = 2\aA$, where $\aA \ge \kK$.
For $K = (K_m: m \in [M])$ we let $\mc{B}^K$ denote 
the set of all $\vc{a} \in \{0,1\}^n$ such that 
$\sum_{i \in P_m} \vc{a}_i = K_m$ for all $m \in [M]$.
 
We claim that we can fix $K$ with $K_m = (2\aA \pm \kK/4) |P_m|$ 
for all $m \in [M]$ such that $\mu_{\vc{p}}(\mc{A} \cap \mc{B}^K) > (1-\dD)^n$.
Indeed, by assumption we have $\mu_{\vc{p}}(\mc{A}) > (1-\dD)^{n/2}$.
Also, for $\vc{a} \sim \mu_{\vc{p}}$ and $X_m = \sum_{i \in P_m} \vc{a}_i$
we have $\mb{E}X_m = 2\aA |P_m|$, so by Chernoff's inequality
$\mb{P}(|X_m-\mb{E}X_m| > \kK|P_m|/4) \le 2e^{-(\kK |P_m|/4)^2/2|P_m|} \le 2e^{-\kK^2\gG n/32}$.
There are at most $n^M$ choices of $K$, so by a union bound and the 
pigeonhole principle there is some $K$ with all $K_m = (2\aA \pm \kK/4) |P_m|$
such that $\mu_{\vc{p}}(\mc{A} \cap \mc{B}^K) > n^{-M} ((1-\dD)^{n/2} - 2M e^{-\kK^2\gG n/32})
>  (1-\dD)^n$, as claimed. 

Now for $\vc{z} = (\vc{z}_m: m \in [M])$ with all $\vc{z}_m \in \mb{Z}^D$
we let $\mc{B}^{K,\vc{z}}$ denote the set of all $\vc{a} \in \mc{B}^K$
with $\sum_{i \in P_m} \vc{v}^i_{a_i,a_i} = \vc{z}_m$ for all $m \in [M]$.
We can fix $\vc{z}$ such that $\mu_{\vc{p}}(\mc{A} \cap \mc{B}^{K,\vc{z}}) > (1-2\dD)^n$.
Indeed, such a $\vc{z}$ exists by the pigeonhole principle,
as all $\|\vc{v}^i_{j_0,j_1}\|_{\vc{R}} \le 1$ and $\max_d R_d < n^C$,
so there are at most $(2n^{C+1})^M$ possible values of $\vc{z}$.

We note for any $\vc{a}$ and $\vc{a}'$ in $\mc{B}^{K,\vc{z}}$
that $\mc{V}(\vc{a},\vc{a}')$ is determined 
by the values $t_m = \sum_{i \in P_m} a_i a'_i$. 
Indeed, for each $i \in P_m$, as $\vc{u}_m$ is an $i$-transfer,
we may suppose that $\vc{v}^i_{1,0}-\vc{v}^i_{0,0}=\vc{u}_m$ and $\vc{v}^i_{1,1}=\vc{v}^i_{0,1}$.
Then 
\begin{align*} \mc{V}(\vc{a},\vc{a}') & = \sum_{m \in [M]} \sum_{i \in P_m} \vc{v}^i_{a_i,a'_i} 
= \sum_{m \in [M]} \Big ( \sum_{i \in P_m: a_i'=0} \vc{v}^i_{a_i,a'_i} + \sum_{i \in P_m: a_i'=1} \vc{v}^i_{a_i,a'_i} \Big ) \\
& = \sum _{m\in [M]} \Big ( \sum_{i \in P_m: a_i'=0} (\vc{v}^i_{a'_i,a'_i} + 1_{a_i=1} \vc{u}_m) + \sum_{i \in P_m: a_i'=1} \vc{v}^i_{a'_i,a'_i} \Big )
= \sum_{m \in [M]} (\vc{z}_m + (K_m-t_m) \vc{u}_m). \end{align*}

Next we claim that there are $\vc{b}, \vc{b}' \in \mc{B}^{K,\vc{z}}$
such that $\vc{v}^* = \mc{V}(\vc{b},\vc{b}')$ and $\wt{\vc{v}} = \mb{E} \mc{V}(\vc{r})$
satisfy $\|\vc{v}^*-\wt{\vc{v}}\|_{\vc{R}} < \zZ n$,
and $\vc{v}^* = \sum_{m \in [M]} (\vc{z}_m + c_m \vc{u}_m)$,
where $c_m \in \mb{Z}$ with $c_m = (\aA \pm \kK/4) |P_m|$ for all $m \in [M]$.
Indeed, for $X=\mc{V}(\vc{r})$ with $\vc{r}=(\vc{b},\vc{b'}) \sim \mu_{\vc{q}}$ we have
$\mb{P}(\|X - \mb{E}X\|_{\vc{R}} \ge \zZ n) \leq 2De^{-\zZ^2 n/2}$ by Lemma \ref{vchern}.
Also, $Y_m = \sum_{i \in P_m} b_i b'_i$ satisfies $\mb{E} Y_m = \aA |P_m|$ and 
$\mb{P}(|Y_m-\mb{E}Y_m| > \kK|P_m|/2) < 2e^{-\kK^2\gG n/128}$ by Chernoff's inequality.
As $\mu_{\vc{p}}(\mc{B}^{K,\vc{z}}) \ge \mu_{\vc{p}}(\mc{A} \cap \mc{B}^{K,\vc{z}}) > (1-2\dD)^n$,
by Theorem \ref{gencor} we have $\mu_{\vc{q}}(\mc{B}^{K,\vc{z}} \times \mc{B}^{K,\vc{z}}) > (1-\dD')^n$,
where $\dD \ll \dD' \ll \zZ$, so we can choose $\vc{b}$ and $\vc{b}'$ as claimed.

Now we can determine values $t_m$ for $m \in [M]$
such that for any $\vc{a}$ and $\vc{a}'$ in $\mc{A} \cap \mc{B}^{K,\vc{z}}$
with $\sum_{i \in P_m} a_i a'_i = t_m$ for all $m \in [M]$
we have $\mc{V}(\vc{a},\vc{a}') = \vc{w}$. Indeed,
$\|\vc{w}-\vc{v}^*\|_{\vc{R}} \le \|\vc{w}-\wt{\vc{v}}\|_{\vc{R}}
+ \|\vc{v}^*-\wt{\vc{v}}\|_{\vc{R}} < 2\zZ n$,
so as $\mc{U}$ is $(k,k\zZ n,\vc{R})$-generating, 
we have $\vc{w}-\vc{v}^* =  \sum_{m \in [M]} e_m \vc{u}_m$, 
with each $e_m \in \mb{Z}$ and $|e_m| \le 3k\zZ n$.
Thus $\vc{w} = \sum_{m \in [M]} (\vc{z}_m + (c_m+e_m) \vc{u}_m)$,
so we take $t_m = K_m - (c_m+e_m)$ for all $m \in [M]$.

It remains to show that we can find such $\vc{a}$ and $\vc{a}'$.
We consider the graph $G = G_1 \times \dots \times G_M$ on $\mc{B}^K$,
where each $G_m$ is the graph on $\tbinom{P_m}{K_m}$
with $A_m A'_m \in E(G_m) \Lra |A_m \cap A'_m| = t_m$. As 
\begin{align*}
\max (2K_m - |P_m|,0) + \tfrac{\kK}{4} |P_m| \leq 
(2\aA - \tfrac{\kK}{4})|P_m| - (\aA + \tfrac{\kK}{4})|P_m| - 3k\zZ n 
\leq t_m \leq K_m - \tfrac{\kK}{4} |P_m|,
\end{align*}
we have $\aA(G_m) < (1-\dD')^n |V(G_m)|$ by Theorem \ref{FranklRodl}.
As $|V(G_m)| \ge \tbinom{\gG n}{\kK \gG n/2} \ge |V(G_{m'})|^{\gG\kK}$ for all $m' \in [M]$,
by Lemma \ref{indep} we have $\aA(G) < (1-2\dD)^n |V(G)|$.
But $|\mc{A} \cap \mc{B}^{K,\vc{z}}|/|\mc{B}^K| 
= \mu_{\vc{p}}(\mc{A} \cap \mc{B}^{K,\vc{z}})/\mu_{\vc{p}}(\mc{B}^K) > (1-2\dD)^n$,
so $\mc{A} \cap \mc{B}^{K,\vc{z}}$ contains an edge of $G$, as required. \qed

\begin{lemma} \label{2x2=hs}
Theorem \ref{general} holds under the additional assumptions 
that $J_1=J_2=\{0,1\}$, $\mc{A}_1=\mc{A}_2=\mc{A}$, 
$q^i_{j_1,j_2}=q_{j_1,j_2}$ for all $i \in [n]$ and $q_{0,1}=q_{1,0}=q_{1,1}=\aA$. 
\end{lemma}

\nib{Proof.}
Let $\mc{P}=(P_m: m \in [M])$ with $|P_m| = \gG_0 n$ for all $m \in [M]$
be such that $\mc{V}$ has transfers for $(\mc{P},\ \mc{U})$,
where $\zZ \ll \gG_0 \ll \eps$.
Let $B_2 = \cup_{m \in [M]} P_m$ and $B_1 = [n] \sm B_2$.
Write $\mc{F}^{K_1,K_2}$ for the set of all $\vc{a} \in \{0,1\}^n$
such that $\sum_{i \in B_j} \vc{a}_i = K_j$ for $j=1,2$.
As in the proof of Lemma \ref{2x2=hs0}, we write 
$\mu_{\vc{p}} = \mu_{\vc{p}_1} = \mu_{\vc{p}_2}$, note that
$p^i_0 = 1-2\aA$ and $p^i_1 = 2\aA$, where $\aA \ge \kK$,
and fix $K_j = (2\aA \pm \kK/4) |B_j|$ for $j=1,2$
such that $\mu_{\vc{p}}(\mc{A} \cap \mc{F}^{K_1,K_2}) > (1-\dD)^n$.

Consider the bipartite graph $G$ with parts $(\tbinom{B_1}{K_1},\tbinom{B_2}{K_2})$ 
where $(\vc{b}_1,\vc{b}_2) \in E(G) \Lra \vc{b}_1\vc{b}_2 \in \mc{A}$.
By Lemma \ref{indep2} there is $\mc{B} \sub \tbinom{B_1}{K_1}$ 
with $|\mc{B}| > (1-\dD')^n \bsize{\tbinom{B_1}{K_1}}$,
where $\dD \ll \dD' \ll \zZ$,
such that for any $\vc{b}_1, \vc{b}'_1$ in $\mc{B}$
we have $|N_G(\vc{b}_1, \vc{b}'_1)| > (1-\dD')^n \bsize{\tbinom{B_2}{K_2}}$.

We will now find $\mc{F} \sub \mc{B} \times \mc{B}$
with $\mu_{\vc{q}}(\mc{F}) > (1-\eps/2)^n$
(also writing $\mu_{\vc{q}}$ for its restriction to $(\{0,1\} \times \{0,1\})^{B_1}$)
such that for any $(\vc{b}_1,\vc{b}'_1) \in \mc{F}$
there are $\vc{b}_2$ and $\vc{b}'_2$ in $N_G(\vc{b}_1, \vc{b}'_1)$,
such that $\mc{V}(\vc{b}_1\vc{b}_2,\vc{b}'_1\vc{b}'_2) = \vc{w}$.
This will suffice to prove the lemma, as then
$\mu_{\vc{q}}( (\mc{A} \times \mc{A})^{\mc{V}}_{\vc{w}} )
\ge \kK^{|B_2|} \mu_{\vc{q}}(\mc{F}) > (1-\eps)^n$, using $\gG_0 \ll \eps$.

Let $\mc{V}_j = \{\vc{v}^i_{j_1,j_2}: i \in B_j\}$ for $j=1,2$
and $\wt{\vc{v}}_j = \mb{E} \mc{V}_j(\vc{r})$, where $\vc{r} \sim \mu_{\vc{q}}$,
so $\wt{\vc{v}}_1 + \wt{\vc{v}}_2 = \wt{\vc{v}} = \mb{E} \mc{V}(\vc{r})$.
Let $\mc{E}$ be the set of $(\vc{b}_1,\vc{b}'_1) \in (\{0,1\} \times \{0,1\})^{B_1}$
such that $\|\mc{V}_1(\vc{b}_1,\vc{b}'_1) - \wt{\vc{v}}_1 \|_{\vc{R}} > \zZ n$.
Then $\mu_{\vc{q}}(\mc{E}) < 2De^{-\zZ^2 n/2}$ by Lemma \ref{vchern}.
We choose $\mc{F} = (\mc{B} \times \mc{B}) \sm \mc{E}$. 
By Theorem \ref{gencor} we have $\mu_{\vc{q}}(\mc{B} \times \mc{B}) > (1-\dD'')^{|B_1|}$ 
with $\dD' \ll \dD'' \ll \zZ  \ll \eps $, so $\mu_{\vc{q}}(\mc{F}) > (1-\eps/2)^n$.
 
It remains to show for fixed $(\vc{b}_1,\vc{b}'_1) \in \mc{F}$
that there is $\vc{b}_2$ and $\vc{b}'_2$ in $N_G(\vc{b}_1, \vc{b}'_1)$
such that $\mc{V}_2(\vc{b}_2,\vc{b}'_2) = \vc{w}' := \vc{w} - \mc{V}_1(\vc{b}_1,\vc{b}'_1)$.
To see this, it suffices to verify the hypotheses of Lemma \ref{2x2=hs0},
applied with $N_G(\vc{b}_1, \vc{b}'_1)$ in place of $\mc{A}$,
restricting $\mu_{\vc{q}}$ to $(\{0,1\} \times \{0,1\})^{B_2}$,
and with $\mc{V}_2$ in place of $\mc{V}$.
We note that $\mc{V}_2$ has transfers for the same $(\mc{P},\ \mc{U})$, 
and $\mc{P}=(P_m: m \in [M])$ is a partition of $B_2$.
As $\| \vc{w}' - \wt{\vc{v}}_2 \|_{\vc{R}} \le \|\vc{w}-\wt{\vc{v}}\|_{\vc{R}} +
\|\mc{V}_1(\vc{b}_1,\vc{b}'_1) - \wt{\vc{v}}_1\|_{\vc{R}} \le 2\zZ n$,
replacing $\zZ$ by $2\zZ$ we see that all hypotheses hold,
so the proof of the lemma is complete. \qed

\begin{lemma} \label{2x2xhs}
Theorem \ref{general} holds under the additional assumptions that $J_1=J_2=\{0,1\}$,
$q^i_{j_1,j_2}=q_{j_1,j_2}$ for all $i \in [n]$ and $q_{0,1}=q_{1,0}=q_{1,1}=\aA$. 
\end{lemma}

\nib{Proof.}
Let $\mc{A}'_1$ be the set of $\vc{a}_1 \in \mc{A}_1$ such that 
there is some $\vc{a}_2=\vc{a}_2(\vc{a}_1) \in \mc{A}_2$ with Hamming distance
$d(\vc{a}_1,\vc{a}_2) \le 2\dD' n$, where $\dD \ll \dD' \ll \zZ$.
We claim that $\mu_{\vc{p}_1}(\mc{A}'_1) > (1-2\dD)^n$.
This follows from the same concentration argument 
used in the proof of Theorem \ref{gencor}.
Indeed, consider $\vc{r}_1 \sim\mu_{\vc{p}_1}$ and
$X = d(\vc{r}_1,\mc{A}_2) = \min_{\vc{a}_2 \in \mc{A}_2} d(\vc{r}_1,\vc{a}_2)$.
As $X$ is $1$-Lipschitz, by Lemma \ref{lip} we have 
$\mb{P}(|X-\mb{E}X|>\dD' n) < e^{-(\dD' n)^2/2n}$.
This implies $\mb{E}X \le \dD' n$, otherwise we would have 
$\mb{P}(|X-\mb{E}X|>\dD' n) \ge \mb{P}(X=0) = \mu_{\vc{p}_1}(\mc{A}_1) > (1-\dD)^n$.
Therefore $\mb{P}(X > 2\dD' n) < e^{-(\dD' n)^2/2n}$, so the claim holds. 

By the pigeonhole principle, we can fix $T \sub [n]$ with $|T| \le 2\dD' n$,
a partition $T = T_{0,1} \cup T_{1,0}$
and $\mc{A}''_1 \sub \mc{A}'_1$ with $\mu_{\vc{p}_1}(\mc{A}''_1) > (1-3\dD)^n$
such that for every $\vc{a}_1 \in \mc{A}''_1$ we have $T_{1,0} = \{ i \in [n]: (a_{1i},a_2(\vc{a}_1)_i) = (1,0) \}$ and $T_{0,1} = \{ i \in [n]: (a_{1i},a_2(\vc{a}_1)_i) = (0,1) \}$.

Now let $\vc{w}' = \sum_{i \in T_{0,1}} (\vc{v}^i_{0,1}-\vc{v}^i_{0,0})
+  \sum_{i \in T_{1,0}} (\vc{v}^i_{1,0}-\vc{v}^i_{1,1})$
and note that for any $\vc{a}_1$ and $\vc{a}'_1$ in $\mc{A}''_1$ 
with $\mc{V}(\vc{a}_1,\vc{a}'_1) = \vc{w} + \vc{w}'$
we have $\mc{V}(\vc{a}_1,\vc{a}_2(\vc{a}'_1)) = \vc{w}$.
Note also that $\|\vc{w}'\|_{\vc{R}} \le |T| \le 2\dD' n$,
so $\| \vc{w} + \vc{w}' - \mb{E} \mc{V}(\vc{r})\|
\le \|\vc{w}'\| + \| \vc{w} - \mb{E} \mc{V}(\vc{r})\| < 2\zZ n$.
Then $\mu_{\vc{q}}( (\mc{A}''_1 \times \mc{A}''_1)^{\mc{V}}_{\vc{w} + \vc{w}'} )
> (1-\eps/2)^n$ by Lemma \ref{2x2=hs}, 
so $\mu_{\vc{q}}( (\mc{A}_1 \times \mc{A}_2)^{\mc{V}}_{\vc{w}} )
\ge \kK^{|T|} \mu_{\vc{q}}( (\mc{A}''_1 \times \mc{A}''_1)^{\mc{V}}_{\vc{w} + \vc{w}'} ) 
> (1-\eps)^n$, as required. \qed \medskip

\nib{Proof of Theorem \ref{general}.}
As noted earlier, we may assume $S=\{1,2\}$. 
By relabelling, we can also assume $\{0,1\} \subset J_1, J_2$. 
As ${\cal V}$ has $\gamma $-robust transfers for ${\cal U}$, 
there are disjoint subsets $P_m$ of $[n]$, with $|P_m| \geq \gamma n$, 
so that ${\bf u}_m$ is an $i$-transfer for all $i\in P_m$. 
By relabelling, we can assume ${\bf v}^i_{1,1} - {\bf v}^i_{0,1} = {\bf u}_m$
and ${\bf v}^i_{1,0} = {\bf v}^i_{0,0}$ for all $i\in P_m$.

Next we describe an alternative construction for the measure $\mu_{\vc{q}}$. 
To begin, we select a random partition $[n] = S \cup T$,
where each $i\in [n]$ appears in $S$ independently with
probability $\kappa $. 
Secondly, we randomly select 
${\bf r}' = ({\bf r}_1', {\bf r}_2') \in (J_1)^{T}
\times (J_2)^{T} = (J_1 \times J_2)^{T}$ 
according to a product measure $\mu _{\bf q '}$ 
on $(J_1 \times J_2)^{T}$, which will be defined below. 
Lastly, we select ${\bf s} = ({\bf s}_1, {\bf s}_2)
\in \{0,1\}^{S} \times \{0,1\}^{S} = (\{0,1\} \times \{0,1\})^{S}$,
according to the uniform measure
$\nu$ on $(\{0,1\} \times \{0,1\})^{S}$.
(We will also write $\nu$ for the uniform measure on $\{0,1\}^{S}$.)
We obtain a random element
${\bf r} = {\bf r}' \circ {\bf s} \in (J_1\times J_2)^n$. 

Note that $r_1,\dots,r_n$ are independent, so ${\bf r}$
defines a product measure ${\mu }_{\bf q ''}$ on $(J_1 \times J_2)^n$.
To determine ${\bf q}''$, note that if $j,j' \in \{0,1\}$ 
then $(q'')^i_{j,j'} = \kappa / 4 + (1-\kappa ) (q')^i_{j,j'}$,
and otherwise $(q'')^i_{j,j'} = (1-\kappa ) (q')^i_{j,j'}$. 
Thus we can obtain ${\bf q}'' = {\bf q}$ by 
setting $(q')^i_{j,j'} = (q^i_{j,j'}-\kappa /4)/(1-\kappa )$ 
for $i \in [n]$,  $j,j'\in \{0,1\}$ 
and $(q')^i_{j,j'} = (q)^i_{j,j'}/(1-\kappa)$ otherwise.
(Note that $\kappa $-boundedness ensures ${q'}^i_{j,j'} \in [0,1]$.)

For fixed $\vc{h}:=(S,\vc{r}')$ and $j=1,2$ let
\begin{align*} 
\mc{F}^{\vc{h}}_j & =  \{ \vc{s}_j \in \{0,1\}^S: {\bf r}'_j\circ \vc{s}_j \in \mc{A}_j \}
\text{ and } \\
\mc{F}^{\vc{h}}_{\vc{w}} & = \{ \vc{s} \in \mc{F}^{\vc{h}}_1 \times \mc{F}^{\vc{h}}_2:
\mc{V}(\vc{r}' \circ \vc {s}) = \vc{w} \}.
\end{align*}
Since ${\bf q}'' = {\bf q}$, we have 
$\mu _{{\bf p}_j}({\cal A}_j) 
= {\mathbb E}_{\bf h}(\nu({\cal F}^{\bf h}_{j}))$ for $j = 1,2$ 
and $\mu _{{\bf q}}(({\cal A} \times {\cal A})^{\cal V}_{\bf w}) 
= {\mathbb E}_{\bf h}(\nu({\cal F}^{\bf h}_{\bf w}))$.

In the remainder of the proof we will show that
$\mb{P}_{\bf h}( \nu(\mc{F}^{\vc{h}}_{\vc{w}}) > (1-\eps/2)^n ) > (1-\dD')^n$,
where $\dD \ll \dD' \ll \zZ$.
This will imply the Theorem, as then 
$\mu_{\vc{q}}( (\mc{A}_1 \times \mc{A}_2)^{\mc{V}}_{\vc{w}} ) 
= {\mathbb E}_{\bf h}(\nu({\cal F}^{\bf h}_{\bf w})) > (1-\dD')^n (1-\eps/2)^n > (1-\eps)^n$.
To achieve this, we will show that for 
`good' ${\bf h}$ we can apply Lemma \ref{2x2xhs}
to $\mc{F}^{\vc{h}}_1$ and $\mc{F}^{\vc{h}}_2$,
with uniform product measure and the array
${\cal X}^{\bf h} := ({\bf v}^i_{j,j'}: i\in S, j,j'\in \{0,1\})$. 
As $\mc{V}(\vc{r}) = \mc{X}^{\vc{h}}(\vc{s}) + {\cal Y}^{\bf h}({\bf r}')$,
we have \[\mc{F}^{\vc{h}}_{\vc{w}}
= (\mc{F}^{\vc{h}}_1 \times \mc{F}^{\vc{h}}_2)^{\mc{X}^{\vc{h}}}_{\vc{w}'},\]
where ${\bf w}' := {\bf w} - {\cal Y}^{\bf h}({\bf r}')$ with 
${\cal Y}^{\bf h} = ({\bf v}^i_{j,j'}: i\in T, j\in J_1, j'\in J_2)$.

First we define some bad events for ${\bf h}$ and show that they are unlikely.
Let $\vc{v}^{\vc{h}} = \mb{E}[ \mc{X}^{\vc{h}}(\vc{s}) \mid \vc{h} ]$
and $\mc{B}_1$ be the event that
$\| \vc{v}^{\vc{h}} - \mb{E} \vc{v}^{\vc{h}}\|_{\vc{R}} > \zZ n$.
Then $\mb{P}(\mc{B}_1) \le 2De^{-\zZ^2 n/8}$ by Lemma \ref{vchern}.
Similarly, the bad event $\mc{B}_2$ that
$\| {\cal Y}^{\bf h}({\bf r}') - \mb{E} {\cal Y}^{\bf h}({\bf r}')\|_{\vc{R}} > \zZ n$
has $\mb{P}(\mc{B}_2) \le 2De^{-\zZ^2 n/8}$.
Note that if $\mc{B}_1 \cup \mc{B}_2$ does not hold,
as $\|\vc{w}-\mb{E} \mc{V}(\vc{r})\|_{\vc{R}} \le \zZ n$,
we have $\| \vc{v}^{\vc{h}} - \vc{w} '\|_{\vc{R}} \le 3\zZ n$.

The last bad event is that we do not have robust transfers.
Let $\mc{P}^{\vc{h}}=(P^{\vc{h}}_m: m \in [M])$,
where $P^{\vc{h}}_m$ is the set of $i \in P_m$
such that $\vc{u}_m$ is an $i$-transfer in $\mc{X}^{\vc{h}}$. 
Recalling that $\vc{u}_m$  is an $i$-transfer in $\mc{V}$ 
via $(0,1)$ and $(0,1)$ for all $i \in P_m$,
we have $i\in P^{\bf h}_m$ whenever $i\in S$,
so ${\mathbb E}|P^{\bf h}_m| \geq \kappa \gamma n$.
By Chernoff's inequality, the bad event $\mc{B}_3$ 
that some $|P^{\vc{h}}_m| < \kK \gG n/2$  satisfies
$\mb{P}(\mc{B}_3) < 2Me^{-\kK^2 \gG^2 n/8}$.

Now let $\mc{G}$ be the good event for $\vc{h}$ that  
$\nu(\mc{F}^{\vc{h}}_1) \nu(\mc{F}^{\vc{h}}_2) > (1-\dD')^n$. 
By Cauchy-Schwarz and Theorem \ref{gencor} we have
\[\mb{E}_{\vc{h}} \nu(\mc{F}^{\vc{h}}_1) \nu(\mc{F}^{\vc{h}}_2)
\geq ( \mb{E}_{\vc{h}} \nu(\mc{F}^{\vc{h}}_1 \times \mc{F}^{\vc{h}}_2) )^2 
= \mu_{\vc{q}}(\mc{A}_1 \times \mc{A}_2)^2 > (1-\dD'/4)^{2n},\] 
so $(1-\mb{P}(\mc{G})) (1-\dD')^n + \mb{P}(\mc{G}) \geq (1-\dD'/2)^{n}$,
giving $\mb{P}(\mc{G}) > (1-\dD'/2)^n/2$.
Thus with probability at least $(1-\dD')^n$
the event $\mc{G} \sm \cup_{i=1}^3 \mc{B}_i$ holds,
so we can apply Lemma \ref{2x2xhs} to obtain
$\nu(\mc{F}^{\vc{h}}_{\vc{w}}) 
= \nu ((\mc{F}^{\vc{h}}_1 \times \mc{F}^{\vc{h}}_2)^{\mc{X}^{\vc{h}}}_{\vc{w}'})
> (1-\eps/2)^n$, as required to prove the theorem. \qed

\section{Proof of Theorem \ref{thm: trichotomy}}\label{sec:trichotomy}

In this section we will prove Theorem \ref{thm: trichotomy}. 
Let ${\cal X} = (\{0,1\}^n)^{\cal V}_{\bf z}$, 
as in the statement of Theorem \ref{thm: trichotomy}. 
The proof will split naturally into two pieces 
according to the VC-dimension of 
$({\cal X} \times {\cal X})^{\cal V_{\cap}}_{\bf w}$.
The next subsection shows that for high VC-dimension
cases $i$ or $ii$ of Theorem \ref{thm: trichotomy} hold;
the following subsection shows
that case $iii$ holds in the case of small VC-dimension.

\subsection{Large VC-dimension}
\label{subsect: large VC-dim case}

Here we implement the strategy discussed in
subsection \ref{subsec: supersaturation}:
we consider the maximum entropy measure $\mu_{\wt{\bf q}}$ that represents
$({\cal X} \times {\cal X})^{{\cal V}_{\cap }}_{\bf w}$,
and distinguish cases $i$ or $ii$ from Theorem \ref{thm: trichotomy}
according to whether its marginals $\mu_{\wt{\bf p}}$ are close to 
$\mu_{\bf p} := \mu _{{\bf p}^{\cal V}_{\bf z}}$.
Throughout this subsection we use the following notation.

\begin{dfn} \label{tilde}
Let $J = \{0,1\}\times \{0,1\}$ and let 
$\widetilde {\cal V} = (\widetilde {\bf v}^i_j)$ 
denote the $(n,J)$-array in ${\mathbb Z}^{3D}$ with
\begin{gather*}
\widetilde {\bf v}^i_{1,1} = ({\bf v}_i, {\bf v}_i, {\bf v}_i), \qquad 
\widetilde {\bf v}^i_{1,0} = ({\bf v}_i, {\bf 0}, {\bf 0}), \\
\widetilde {\bf v}^i_{0,1} = ({\bf 0}, {\bf v}_i,  {\bf 0}), \qquad 
\widetilde {\bf v}^i_{0,0} = ({\bf 0,0,0})\in {\mathbb Z}^{3D},
\end{gather*}
where ${\bf 0}$ denotes the zero vector in ${\mathbb Z}^D$. 
Let ${\bf z}, {\bf w} \in {\mathbb Z}^D$ 
and ${\cal X} = (\{0,1\}^n)^{\cal V}_{\bf z}$. 

We identify $({\cal X} \times {\cal X})^{{\cal V}_{\cap }}_{\bf w}$ 
with $(J^n)^{ \widetilde {\cal V}}_{ \widetilde {\bf x}}$,
where $\widetilde {\bf x} := ({\bf z}, {\bf z}, {\bf w})$.
We define
\[{\mu }_{\widetilde {\bf q}} := \mu ^{\widetilde {\cal V}}_{\widetilde {\bf x}} 
\quad \text{ and } \quad
\mu _{{\bf p}} := \mu ^{\cal V}_{\bf z}. \]
We denote the marginals of ${\mu }_{\widetilde {\bf q}}$
by $\mu _{\widetilde {\bf p}}$ (both marginals are equal).
\end{dfn}

Next we show $\kK$-boundedness of the above measures
under our usual assumptions on $\mc{V}$
(and justify the final statement
of the above definition).

\begin{lemma}
	\label{lem: general transfer to measure}
Let $0<n^{-1} \ll \kappa \ll \gamma, \gamma ' \ll \lambda \ll \eps , D^{-1}, C^{-1}, k^{-1}$. 
Let ${\bf R} \in {\mathbb R}^D$ with $\max _{d} R_d \leq n^C$. 
and $\widetilde {\bf R} = ({\bf R},{\bf R}, {\bf R}) \in {\mathbb Z}^{3D}$.
Suppose ${\cal V} = ({\bf v}_i : i\in [n])$ is an
${\bf R}$-bounded, $\gamma '$-robustly $(\gamma, {\bf R})$-generic 
$\gamma $-robustly $({\bf R},k)$-generating array in $\mb{Z}^D$. 
Then $\widetilde {\cal V}$ is $\widetilde {\bf R}$-bounded,
$(\gamma /2)$-robustly $({\widetilde {\bf R}},3k)$-generating
and $(\gamma '/2)$-robustly $(\gamma ^3, {\widetilde {\bf R}})$-generic.
Suppose also that
$\dim_{VC}( ({\cal X} \times {\cal X})^{{\cal V}_{\cap }}_{\bf w} ) \geq \lambda n$. 
Then $\mu _{{\bf p}}$ and ${\mu }_{\widetilde {\bf q}}$ are $\kappa$-bounded,
both marginals of $\mu _{\widetilde {\bf q}}$ are $\mu _{\widetilde {\bf p}}$,
and $\mu _{\widetilde {\bf p}} \in {\cal M}^{\cal V}_{\bf z}$.
\end{lemma}	

\begin{proof} 
The proof of the first statement is
`definition chasing' so we omit it.
The last statement follows from symmetry and 
strict concavity of $L(p)$ (see Lemma \ref{entapprox} $ii$).
For $\kappa$-boundedness
of $\mu _{{\bf p}}$ and ${\mu }_{\widetilde {\bf q}}$ 
we apply Lemma \ref{bounded}. For $\mu _{{\bf p}}$ this is valid
as $\dim_{VC}({\cal X}) \geq \lambda n$ and
${\cal V}$ is $\gamma $-robustly $(\gamma ,{\bf R})$-generic.
For ${\mu }_{\widetilde {\bf q}}$ this is valid
as $\dim _{VC}(({\cal X} \times {\cal X})^{{\cal V}_{\cap }}_{\bf w}) \geq \lambda n$
and $\widetilde {\cal V}$ is
$(\gamma '/2)$-robustly $(\gamma ^3, {\widetilde {\bf R}})$-generic.
\end{proof}

Now we prove the main lemma of this subsection,
which distinguishes cases $i$ and $ii$ according to
$\|{\bf p} - \widetilde {\bf p} \|_1 
:= \sum _{i\in [n], j\in J} |p^i_j - \widetilde {p}^i_{j}|$.

\begin{lemma}
\label{lem: high vc-dim case of trichotomy}
	Let $0<n^{-1} \ll \delta \ll \delta _1 \ll  
\gamma, \gamma ' \ll \lambda \ll \eps , D^{-1}, C^{-1}, k^{-1}$ 
	and let ${\bf R} \in {\mathbb R}^D$ with $\max _{d} R_d \leq n^C$. 
Suppose ${\cal V} = ({\bf v}_i : i\in [n])$ is an
${\bf R}$-bounded, $\gamma '$-robustly $(\gamma, {\bf R})$-generic 
$\gamma $-robustly $({\bf R},k)$-generating array in $\mb{Z}^D$. 
Fix notation as in Definition \ref{tilde} and suppose
$\dim_{VC}( ({\cal X} \times {\cal X})^{{\cal V}_\cap}_{\bf w} ) \geq \lambda n$.
\begin{enumerate}
	\item Suppose
	$\|{\bf p} - {\widetilde {\bf p}}\|_1 \leq \delta _1n$. 
	If  ${\cal A} \subset {\cal X}$ with 
	$|{\cal A}| \geq (1-\delta )^n|{\cal X}|$ then
	$|({\cal A} \times {\cal A})^{{\cal V}_\cap}_{\bf w}| \geq 
	(1-\eps )^n |({\cal X} \times {\cal X})^{{\cal V}_\cap}_{\bf w}|$.
	\item Suppose $\|{\bf p} - {\widetilde {\bf p}}\|_1 
	\geq \delta _1n$. Then there is 
	${\cal B}_{full} \subset {\cal X}$ with 
	$|{\cal B}_{full}| \leq (1-\delta )^n|{\cal X}|$ and
	$$|({\cal X} \times {\cal X})^{{\cal V}_\cap}_{\bf w} \setminus 
	({\cal B}_{full} \times {\cal B}_{full})^{{\cal V}_{\cap}}_{\bf w}| 
	\leq (1-\delta )^n |({\cal X} \times {\cal X})^{{\cal V}_\cap}_{\bf w}|.$$
	Furthermore, if ${\cal B} \subset {\cal B}_{full}$ with 
	$|{\cal B}| \geq (1-\delta )^n|{\cal B}_{full}|$ then 
	$|({\cal B} \times {\cal B})^{{\cal V}_{\cap }}_{\bf w}|
	\geq (1-\eps )^n|({\cal X} \times {\cal X})^{{\cal V}_{\cap }}_{\bf w}|$.	
\end{enumerate}
\end{lemma}	

\begin{proof} 
By Lemma \ref{lem: general transfer to measure}, 
${\mu }_{\bf p}$ and $\mu _{\widetilde {\bf q}}$ 
(and so $\mu _{\widetilde {\bf p}}$) are $\kappa $-bounded,
where $\kappa \ll \gamma , \gamma '$.
Then by Theorem \ref{ldp}, $\mu _{\bf p} \approx _{\triangle} \nu$,
where $\nu $ is the uniform distribution on 
$\triangle _n= (\{0,1\}^n)^{\cal V}_{\bf z} = {\cal X}$. 
Also by Theorem \ref{ldp}, $\mu _{\widetilde {\bf q}} \approx _{\triangle '} \nu '$,
where $\nu '$ is the uniform distribution on 
$\triangle _n'= (J^n)^{\widetilde {\cal V}}_{\widetilde {\bf x}} 
= ({\cal X} \times {\cal X})^{{\cal V}_{\cap }}_{{\bf w}}$. 

Fix constants $\delta \ll \dD_0 \ll \delta _1 \ll \delta _2 \ll \eps _1 \ll \kappa $.

\underline{\textbf{Case $i$:}} $\|{\bf p} - \widetilde {\bf p}\|_1 \leq \delta _1 n$.

Given ${\cal A} \subset {\cal X}$ with $|{\cal A}| \geq (1-\delta )^n|{\cal X}|$, we have $\mu _{\bf p}({\cal A}) \geq (1-\delta _1)^n$ by Theorem \ref{ldp}. As both ${\bf p}$ and $\widetilde {\bf p}$ are $\kappa $-bounded, and $\|{\bf p} - {\widetilde {\bf p}}\|_1 \leq \delta _1 n$, we have $\mu _{\widetilde {\bf p}}({\cal A}) \geq (1-\delta _1)^n(\kappa )^{\delta _1 n} \geq (1-\delta _2)^n$. As the hypotheses of Theorem \ref{binary} hold, we find $\mu _{\widetilde {\bf q}}(({\cal A} \times {\cal A})^{{\cal V}_\cap }_{\bf w}) \geq (1-\eps _1)^n$. Theorem \ref{ldp} applied once again for $\mu _{\widetilde {\bf q}}$ gives $|({\cal A} \times {\cal A})^{{\cal V}_\cap }_{\bf w}| \geq (1-\eps )^n|(J^n)^{\widetilde {\cal V}}_{\widetilde {\bf x}}| = (1-\eps )^n|({\cal X} \times {\cal X})^{{\cal V}_{\cap }}_{\bf w}|$. 

\textbf{\underline{Case $ii$:} $\|{\bf p} - \widetilde {\bf p}\|_1 \geq \delta _1 n$.}

In this case, we let \[{\cal B}_{full} = \big \{ {\bf x} \in {\cal X}: 
- \log _2 \mu _{\widetilde {\bf p}}({\bf x}) =
H(\mu _{\widetilde {\bf p}}) \pm \delta _1^2 n/2 \big \}.\] 
Note that $H(\mu _{\widetilde {\bf p}}) \leq H(\mu _{\bf p}) - \delta_1^2n$
by Lemma \ref{perturbmaxent}, so by Lemma \ref{uctg3} we have
$|{\cal B}_{full}| \leq 2^{H(\mu _{\bf p}) -\delta _1^2 n/2} 
\leq (1-\delta )^n|{\cal X}|$.

Next we show that almost all ${\bf w}$-intersections
in $\mc{X}$ are contained in ${\cal B}_{full}$. 
We require an upper bound on the size of
${\cal Y} := ({\cal X} \times {\cal X})^{{\cal V}_\cap}_{\bf w} 
\setminus ({\cal B}_{full} \times {\cal B}_{full})^{{\cal V}_{\cap}}_{\bf w}$. 
Note that ${\cal Y} \subset ({\cal X}  \times ({\cal X}  
\setminus {\cal B}_{full}) )^{{\cal V}_\cap}_{\bf w} \cup
(({\cal X}  \setminus {\cal B}_{full} ) \times {\cal X}  )^{{\cal V}_\cap}_{\bf w}$. 
As $\mu _{\widetilde {\bf p}}$ is a marginal of $\mu _{\widetilde {\bf q}}$, 
this gives $\mu _{\widetilde {\bf q}} ({\cal Y}) \leq 2 \mu _{\widetilde {\bf q}}(({\cal X}  \times
({\cal X} \setminus {\cal B}_{full}) )^{{\cal V}_\cap}_{\bf w}) \leq 2 \mu _{\widetilde {\bf
    p}}({\cal X}  \setminus {\cal B}_{full})$. 
However, ${\mu }_{\widetilde {\bf p}}$ is $\kappa$-bounded, 
so Lemma \ref{uctg1} gives $\mu _{\widetilde {\bf q}}({\cal Y}) 
\leq 2\mu _{\widetilde {\bf p}}({\cal X} \setminus {\cal B}_{full}) 
\leq (1-\dD_0)^n$. Then Lemma \ref{uctg3} gives
$|{\cal Y}| \leq (1-\delta )^n |({\cal X} \times {\cal X})^{{\cal V}_{\cap }}_{\bf w}|$, 
as required.

It remains to show supersaturation relative to ${\cal B}_{full}$;
the proof is similar to that of case $i$.
As ${\mu }_{\widetilde {\bf p}}$ is $\kappa $-bounded, 
Lemma \ref{uctg3} gives $\log _2 |{\cal B}_{full}| 
\geq H(\mu _{\widetilde {\bf p}}) - \delta _1n$. 
Suppose ${\cal B} \subset {\cal B}_{full}$ with 
$|{\cal B}| \geq (1-\delta )^n|{\cal B}_{full}|$.
Then $\mu _{{\widetilde {\bf p}}}({\cal B}) \geq 
|{\cal B}|2^{-H(\mu _{\widetilde {\bf p}}) - \delta _1^2n/2} \geq (1-\delta _1)^n$. 
Theorem \ref{general} gives 
$\mu _{\wt{\bf q}} ( ({\cal B} \times {\cal B})^{{\cal V}_{\cap }}_{\bf w}) \geq (1-\eps _1)^n$. 
A final application of Theorem \ref{ldp} gives 
$|({\cal B} \times {\cal B})^{{\cal V}_{\cap }}_{\bf w}| 
\geq (1-\eps )^n|({\cal X} \times {\cal X})^{{\cal V}_{\cap }}_{\bf w}|$.\end{proof}

\subsection{Small VC-dimension}

To complete the proof of Theorem \ref{thm: trichotomy},
it remains to show the negative result in the case that
$({\cal X} \times {\cal X})^{{\cal V}_{\cap }}_{\bf w}$ has small VC-dimension,
i.e.\ that there is a large subset of $\mc{X}$ with no ${\bf w}$-intersection.

First we use universal VC-dimension (see Definition \ref{uvc})
to give a criterion for $({\cal X} \times {\cal X})^{{\cal V}_{\cap }}_{\bf w}$ 
to have large VC-dimension (which will be used in contrapositive form).
We require the following notation. 
Given ${\bf x} \in {\cal X}$, $j \in \{0,1\}$, $\aA>0$ let
\begin{gather*}
S_j({\bf x}) = \{i\in [n]: x_i = j\}, \qquad
{\cal V}^j_{\bf x} = ({\bf v}_i:  {i\in S_j({\bf x})}), \\
N^1_{{\bf w}}({\bf x}) = (\{0,1\}^{S_1({\bf x})})^{{\cal V}^1_{\bf x}}_{\bf w}, 
\qquad N^0_{{{\bf z} - {\bf w}}}({\bf x}) 
= (\{0,1\}^{S_0({\bf x})})^{{\cal V}^0_{\bf x}}_{{\bf z}-{\bf w}}, \\
{\cal X}^\aA_{\bf w} = \{ {\bf x}\in {\cal X} : 
|N^0_{{\bf z} - {\bf w}}({\bf x})| \geq (1+\alpha )^n \text{ and } 
|N^1_{{\bf w}}({\bf x})| \geq (1+\alpha )^n \}. 
\end{gather*}
An important observation is
\[ \brac{ {\bf x}' \in {\cal X} \text{ and }
{\cal V}_{\cap }({\bf x},{\bf x}') = {\bf w} } \Lra
\brac{ {\bf x}' = {\bf y}_0 \circ {\bf y}_1 \text{ with }
{\bf y}_0 \in N^0_{{{\bf z} - {\bf w}}}({\bf x})
\text{ and } {\bf y}_1 \in N^1_{{\bf w}}({\bf x}) }.\] 

\begin{lemma}
	\label{lem: char of low VC-dim}
Let $n^{-1} \ll \lambda \ll \gamma , \gamma ' \ll \alpha \ll \eps , D^{-1}, C^{-1}, k^{-1}$
       and let ${\bf R} \in {\mathbb R}^D$ with $\max _{d} R_d \leq n^C$. 
Suppose $\mc{V} = (\vc{v}^i_j)$ is an $\vc{R}$-bounded, 
$\gamma $-robustly $({\bf R},k)$-generating,
$\gG'$-robustly $(\gG,\vc{R})$-generic $(n,J)$-array in $\mb{Z}^D$.
Suppose $|{\cal X}| \geq (1+\eps )^n$ and $|{\cal X}^\aA_{\bf w}| \geq |{\cal X}|/2$.
Then $\dim _{VC}(({\cal X} \times {\cal X})^{{\cal V}_{\cap}}_{\bf w}) \geq \lambda n$.
\end{lemma}

\begin{proof} 
The strategy of the proof is to find a large set $S$
that is shattered by a subset $\mc{X}'$ of $\mc{X}$,
such that if ${\bf x} \in \mc{X}'$ then
$\dim _{UVC}(N^0_{{\bf z} - {\bf w}}({\bf x}))$ and 
$\dim _{UVC}(N^1_{{\bf w}}({\bf x}))$ are large.
Then the definition of universal VC-dimension
will imply that $S$ is shattered by 
$({\cal X} \times {\cal X})^{{\cal V}_{\cap}}_{\bf w}$,
as $N^0_{{\bf z} - {\bf w}}({\bf x})$ shatters $S_0({\bf x})$
and $N^1_{{\bf w}}({\bf x})$ shatters $S_1({\bf x})$.
First we note by Lemma \ref{binary bounded}
that ${\mu }_{\bf p}$ is $\kK$-bounded,
where $\aA \ll \kK \ll \eps$.

Let ${\cal X}'$ be the set of ${\bf x} \in {\cal X}^\aA_{\bf w}$
such that ${\cal V}^1_{\bf x}$ and ${\cal V}^0_{\bf x}$
are $(\kappa ^k\gamma /2k)$-robustly $({\bf R},k)$-generated.
We claim that $|{\cal X}'| \geq |{\cal X}| /4$.
To see this, note that for any 
${\bf w} \in {\mathbb Z}^D$ with $\| {\bf w}\|_{\bf R} \leq 1$, 
as ${\cal V}$ is $\gamma $-robustly $({\bf R},k)$-generating, 
there are $L \geq \gamma n/k$ disjoint sets $S_1,\ldots ,S_L$, 
with $|S_{\ell }| \leq k$ for all $\ell \in [L]$, such that 
for all $\ell \in [L]$ there is a partition 
$S_{\ell } = S^1_{\ell } \cup S^{0}_{\ell }$ with 
$\sum _{i\in S^1_{\ell }} {\bf v}_i 
- \sum _{i\in S^0_{\ell }} {\bf v}_i= {\bf w}$.
Given ${\bf x} \in \{0,1\}^n$ and $j\in \{0,1\}$,
let $L_{\bf w}^j({\bf x}) = 
\{\ell \in [L]: S_{\ell } \subset S_j({\bf x})\}$.
As ${\mu }_{\bf p}$ is $\kK$-bounded,
each $\mb{E}_{{\bf x} \sim {\mu }_{\bf p}}(|L_{\bf w}^j({\bf x})|) 
\geq \kappa ^{k} L$.
Let ${\cal B}_{\bf w}$ be the event that either
$|L_{\bf w}^j({\bf x})| \leq \kappa ^{k} L/2$,
and ${\cal B}$ be the union of ${\cal B}_{\bf w}$ 
over all ${\bf w} \in {\mathbb Z}^D$ with $\|{\bf w}\|_{\bf R} \leq 1$.
There are at most $(2n+1)^{CD}$ choices of ${\bf w}$,
so by Chernoff's inequality and a union bound
$\mb{P}_{{\bf x} \sim {\mu }_{\bf p}}({\cal B}) = (1-c_{\kappa })^n$ for some $c_{\kappa } >0$.
By Theorem \ref{ldp}, we deduce
$|{\cal X}^\aA_{\bf w} \sm \mc{B}| \geq |{\cal X}| /4$.
As ${\cal X}^\aA_{\bf w} \sm \mc{B} \sub \mc{X}'$
this proves the claim.

Next we claim that if ${\bf x} \in {\cal X}'$ then 
$\dim _{UVC}(N^0_{{\bf z} - {\bf w}}({\bf x})) \geq \lambda n$ 
in $\{0,1\}^{S_0({\bf x})}$ and 
$\dim _{UVC}(N^1_{{\bf w}}({\bf x})) \geq \lambda n$ 
in $\{0,1\}^{S_1({\bf x})}$.
Indeed, as ${\bf x} \in {\cal X}^\aA_{\bf w}$ we have
$|S_1({\bf x})|, |S_0({\bf x})| \geq \log _2( 1 + \alpha ) n$, 
so ${\cal V}^0_{\bf x}$ and ${\cal V}^1_{\bf x}$ are 
$(\gamma '/ \log _2(1 +\alpha ))$-robustly $(\gamma ,{\bf R})$-generic,
and by Lemma \ref{SS} both $N^0_{{\bf z} - {\bf w}}({\bf x})$
and $N^1_{{\bf w}}({\bf x})$ have VC-dimension
at least $\aA' n$, where $\gG, \gG' \ll \aA' \ll \aA$.
They are clearly ${\bf R}$-bounded,
and by definition of ${\cal X}'$ they are
$(\kappa ^k\gamma /2k)$-robustly $({\bf R},k)$-generated,
so the claim follows from Lemma \ref{universal VC-dimension}.

Now we can implement the strategy outlined at the start of the proof.
As $|{\cal X}'| \geq |{\cal X}|/4 \geq (1+\eps )^n/4$, 
we have $\dim _{VC}({\cal X}') \geq \lambda n$ by Lemma \ref{SS}. 
Let $S \subset [n]$ with $|S| \geq \lambda n$ be shattered by ${\cal X}'$. 
We will show that $S$ is also shattered by 
$({\cal X} \times {\cal X})^{{\cal V}_{\cap }}_{\bf w} 
\subset (\{0,1\}\times \{0,1\})^n$. 
Indeed, suppose that we are given a partition 
$\cup _{j_1,j_2 \in \{0,1\}} S_{j_1,j_2}$ of $S$, 
and wish to find sets ${\bf x},{\bf x}' \in {\cal X}$ 
such that ${\cal V}_{\cap }({\bf x},{\bf x}') = {\bf w}$ 
and $\{i\in S: x_i = j_1, x'_i = j_2\} = S_{j_1,j_2}$. 
As ${\cal X}'$ shatters $S$, there is ${\bf x} \in {\cal X}'$ 
with $\{i\in S: x_i = 1\} = S_{1,0} \cup S_{1,1}$. 
Furthermore, using the universal VC-dimension
of $N^1_{{\bf w}}({\bf x})$ and $N^0_{{\bf z}-{\bf w}}({\bf x})$,
we have ${\bf y}_1 \in N^1_{{\bf w}}({\bf x})$ with 
$\{i\in S_{1,0} \cup S_{1,1}: ({\bf y}_1)_i = 1\} = S_{1,1}$
and ${\bf y}_0 \in  N^0_{{\bf z}-{\bf w}}({\bf x})$ with 
$\{i\in S_{0,0} \cup S_{0,1}: ({\bf y}_2)_i = 1\} = S_{0,1}$. 
Now ${\bf x}' = {\bf y}_0 \circ {\bf y}_1 \in {\cal X}$ 
with ${\cal V}_{\cap }({\bf x},{\bf x}') = {\bf w}$ 
and $\{i\in S: x_i = j_1, x'_i = j_2\} = S_{j_1,j_2}$ 
for all $j_1, j_2 \in \{0,1\}$. Thus $S$ is shattered by 
$({\cal X} \times {\cal X})^{{\cal V}_{\cap }}_{\bf w}$, and so 
$\dim _{VC}(({\cal X} \times {\cal X})^{{\cal V}_{\cap }}_{\bf w}) 
\geq \lambda n$, as required.
\end{proof}

We conclude with the main result of this subsection,
that there is a large subset of $\mc{X}$ with no ${\bf w}$-intersection.

\begin{lemma}
\label{lem: low vc-dim case of trichotomy}
Let $n^{-1} \ll \lambda \ll \gamma , \gamma ' \ll \eps , D^{-1}, C^{-1}, k^{-1}$ 
and let ${\bf R} \in {\mathbb R}^D$ with $\max _{d} R_d \leq n^C$.
Suppose $\mc{V} = (\vc{v}^i_j)$ is an $\vc{R}$-bounded, 
$\gamma $-robustly $({\bf R},k)$-generating,
$\gG'$-robustly $(\gG,\vc{R})$-generic $(n,J)$-array in $\mb{Z}^D$. Suppose ${\bf z} \neq {\bf w}$ and that $\dim_{VC}( ({\cal X} \times {\cal X})^{{\cal V}_{\cap }}_{\bf w} ) 
\leq \lambda n$. Then there is ${\cal B}_{empty} \subset {\cal X}$ 
with $|{\cal B}_{empty} | \geq \lfloor (1-\eps)^n|{\cal X}|\rfloor $ and
$({\cal B}_{empty} \times {\cal B}_{empty})^{{\cal V}_{\cap }}_{\bf w} = \emptyset $.
\end{lemma}	

\begin{proof}
We may assume $|{\cal X}| \geq (1+\eps)^n$ as otherwise we 
can take ${\cal B}_{empty} = \emptyset$.  
Take $\alpha $ and $\xi $ such that 
$\gamma , \gamma ' \ll \alpha \ll \xi \ll \eps ,D^{-1}, C^{-1},k^{-1}$. 
Let ${\cal X}_0 = \{{\bf x} \in {\cal X}: 
|N^0_{{\bf z} - {\bf w}}({\bf x})| \leq (1+\alpha )^n\}$ 
and ${\cal X}_1 = \{{\bf x} \in {\cal X}: 
|N^1_{{\bf w}}({\bf x})| \leq (1+\alpha )^n\}$.
Then $|{\cal X}_0 \cup {\cal X}_1| \geq |{\cal X}|/2$
by Lemma \ref{lem: char of low VC-dim}.
The remainder of the proof splits into two similar cases
according to which $\mc{X}_j$ is large;
we will give full details for the case $j=1$
and then indicate the necessary modifications for $j=0$.

Suppose $|{\cal X}_1| \ge  |{\cal X}|/4$.
By the pigeonhole principle, we can fix
${\cal X}' \subset {\cal X}_1$ and $t \in [n]$
such that $|{\cal X}'| \geq |{\cal X}| /4n$ 
and $|S_1({\bf x})| = t$ for all ${\bf x} \in {\cal X}'$.
As $\tbinom {n}{t} \geq |{\cal X}'| \geq (1+\eps )^n /4n$
we have $\xi n \leq t \leq n-\xi n$.
Next we can pass to a subset ${\cal X}'' \subset {\cal X}'$ 
with $|{\cal X}''| \geq |{\cal X}'|/2\tbinom {n}{2\xi n} 
\geq (1-\xi ^{1/2})^n|{\cal X}|$ that is `well-separated',
in that the Hamming distance $d({\bf  x}, {\bf x}') \geq 2\xi n$ 
for all distinct ${\bf x}, {\bf x}' \in {\cal X}''$. 
Indeed, we can select ${\cal X}''$ greedily, 
noting that each element of ${\cal X}''$ forbids at most 
$\sum _{i\in  [0,2\xi n]} \tbinom {n}{i} 
\leq 2\tbinom {n}{2\xi n}$ elements from ${\cal X}''$.
As $|S_1({\bf x})| = |S_1({\bf x}')|= t$, 
this gives $|S_1({\bf x})\setminus S_1({\bf x}')| \geq \xi n$ 
for all distinct ${\bf x}, {\bf x}' \in {\cal X}''$.

Next we will define ${\cal B}_{empty}$.
We randomly select $S \subset [n]$ with $|S| = \xi n$, and let 
${\cal C} = \{{\bf x} \in {\cal X}'': S \subset S_1({\bf x})\}$. 
We say that ${\bf x}\in {\cal C}$ is \emph{isolated} 
if there is no ${\bf x}' \in {\cal C}$ 
with ${\cal V}_{\cap }({\bf x}, {\bf x}') = {\bf w}$. 
We let ${\cal B}_{empty}$ be the set of
isolated ${\bf x}\in {\cal C}$.
Then by definition we have
$({\cal B}_{empty} \times {\cal B}_{empty})^{{\cal V}_{\cap}}_{\bf w} = \es$.

Now we will show that 
${\mathbb E}|{\cal B}_{empty}| \geq (1-\eps )^n |{\cal X}|$.
As ${\mathbb E}(|{\cal C}|) 
= \tbinom {t}{\xi n} \tbinom {n}{\xi n}^{-1}  |{\cal X}''|$,
$|{\cal X}''| \geq (1-\xi ^{1/2})^n|{\cal X}|$
and $\xi \ll \eps$, it suffices to show 
${\mathbb P}({\bf x} \mbox { is isolated}
\mid {\bf x} \in {\cal C}) \ge 1/2$
for all ${\bf x} \in {\cal X}'$.
To see this, we condition on ${\bf x} \in {\cal C}$
and note that $S$ is equally likely 
to be any subset of $S_1({\bf x})$ of size $\xi n$.
Consider any ${\bf x}' \in {\cal X}''$ with 
${\cal V}_{\cap}({\bf x}, {\bf x}') = {\bf w}$. 
Note that ${\bf x} \neq {\bf x}'$ since ${\cal V}({\bf x},{\bf x}') = 
{\bf z} \neq {\bf w}$, and so $S_1({\bf x}) \neq S_1({\bf x}')$ as 
both sets have size $t$. Furthermore ${\bf y} := {\bf x}'|_{S_1({\bf x})} \in N^1_{{\bf w}}({\bf x})$,
we have $|S_1({\bf x}) \sm S_1({\bf y})| \geq \xi n$
by definition of ${\cal X}''$, and 
${\bf x}' \in {\cal C} \Lra S \sub S_1({\bf y})$.
For fixed ${\bf y}$ we have $\mb{P}(S \sub S_1({\bf y}))
\le \tbinom {t - \xi n}{\xi n}  \tbinom {t}{\xi n}^{-1}
\le (1-\xi)^{\xi n}$. By definition of ${\cal X}_1$ 
we have a union bound over at most $(1+\alpha )^n$ 
choices of ${\bf y} \in N^1_{{\bf w}}({\bf x})$,
so as $\alpha \ll \xi $, the probability that
${\bf x}$ is not isolated given ${\bf x} \in {\cal C}$
is $o(1)$, so at most $1/2$, as required.

Similarly, if $|{\cal X}_0| \ge  |{\cal X}|/4$,
we define ${\cal X}'$ and ${\cal X}''$
in the same way for ${\cal X}_0$, and let 
${\cal C} = \{{\bf x} \in {\cal X}'': 
S \subset S_0({\bf x})\}$.
We use the same definition
of ${\cal B}_{empty}$ as before,
and bound the probability that
${\bf x}$ is not isolated given ${\bf x} \in {\cal C}$
by taking a union bound over at most $(1+\alpha )^n$ 
choices of ${\bf y} := {\bf x}'|_{S_0({\bf x})} 
\in N^0_{{\bf z}- {\bf w}}({\bf x})$. 
The remaining details of this case
are the same, so we omit them. 
\end{proof}

\subsection{Proof of Theorem \ref{thm: trichotomy}}

\begin{proof}[Proof of Theorem \ref{thm: trichotomy}]
Take $\lambda $ with 
$\gamma _1 , \gamma _1' \ll \lambda \ll \gamma _2, \gamma _2'$. 
If $\dim _{VC}(({\cal X} \times {\cal X})^{\cal V_{\cap }}_{\bf w}) \geq \lambda n$, 
then we can apply Lemma \ref{lem: high vc-dim case of trichotomy}
        with $\gamma = \gamma _1$ and $\gamma ' = \gamma _1'$ 
to obtain case $i$ or $ii$ of Theorem \ref{thm: trichotomy}. 
On the other hand, if $\dim _{VC}(({\cal X} \times {\cal X})^{\cal V_{\cap }}_{\bf w}) \leq \lambda
n$ then we apply Lemma \ref{lem: low vc-dim case of trichotomy} with $\gamma = \gamma _2$ and
$\gamma ' = \gamma _2'$ to obtain case $iii$ of Theorem \ref{thm: trichotomy}.
\end{proof}

\section{Solution of Kalai's Conjecture} \label{sec:concrete description}

In this section we prove Theorem \ref{thm: Kalai},
which is our solution to Kalai's Conjecture \ref{kalai}.
We give the proof in the first subsection,
then generalise it in the following subsection
to show that supersaturation of the type
conjectured by Kalai is quite rare.

\subsection{Proof of Theorem \ref{thm: Kalai}} \label{subsect: kalai conjecture}

As described in subsection \ref{subsec: supersaturation},
the supersaturation conclusion desired by Conjecture \ref{kalai} 
(case $i$ of Theorem \ref{thm: trichotomy}) needs the
maximum entropy measure $\mu_{\wt{\bf q}}$ that represents
$({\cal X} \times {\cal X})^{{\cal V}_{\cap }}_{\bf w}$
to have marginals $\mu_{\wt{\bf p}}$ 
close to $\mu_{\bf p} := \mu _{{\bf p}^{\cal V}_{\bf z}}$.
Recall that in Definition \ref{tilde} we constructed $\mu_{\wt{\bf q}}$
as $\mu ^{\wt{\cal V}}_{\wt{\bf x}}$, 
where $\widetilde {\cal V}$ is a certain
$(n,\{0,1\}\times \{0,1\})$-array in ${\mathbb Z}^{3D}$
and $\wt{\bf x} := ({\bf z}, {\bf z}, {\bf w})$.
In this subsection we work with the Kalai vectors 
${\cal V} = ({\bf v}_i)_{i\in [n]}$ 
with ${\bf v}_i = (1,i)$, so $D=2$.
In the notation of Conjecture \ref{kalai} 
we have ${\bf z}=(k,s)$ and ${\bf w}=(t,w)$.
Sometimes we will indicate the dependence on $n$ 
as a subscript in our notation,
e.g.\ writing ${\bf z}_n=(k_n,s_n)= (\bfl{\aA_1 n}, 
\bfl{ \aA_2 \tbinom {n}{2} })$.
Our proof will use the following concrete description 
of the maximum entropy measures as Boltzmann distributions.

\begin{lemma} \label{boltzmann}
Let $\mc{V}=(\vc{v}^i_j)$ be an $(n,J)$-array in $\mb{Z}^D$ and $\vc{z} \in \mb{Z}^D$.
Suppose $\vc{p} = \vc{p}^{\mc{V}}_{\vc{z}}$ has all $p^i_j \ne 0$.
Then there is $\blL \in (\mb{R}^D)^J$ such that 
all $p^i_j = Z_i^{-1} e^{\blL_j \cdot \vc{v}^i_j}$,
where $Z_i = \sum_{j \in J} e^{\blL_j \cdot \vc{v}^i_j}$.
\end{lemma}

\nib{Proof.}
By the theory of Lagrange multipliers, $\vc{p}$ is a stationary point of
\[ L(\vc{p},\blL) = H(\vc{p}) -  (\log 2)^{-1} \sum_{d \in [D]} \sum_{j \in J} \lL_{j,d} 
  (\sum_{i \in [n]} p^i_j (v^i_j)_d - z_d),\]
so $0 = -1+\log(p^i_j)-\sum_{d \in [D]} \lL_{j,d} (v^i_j)_d$,
which gives the stated formula. \qed  

\medskip

When $\vc{p} = {\bf p}^{\cal V}_{\bf z}$ and
$\mu_{\wt{\bf q}} = \mu ^{\widetilde {\cal V}}_{\widetilde {\bf x}}$
are $\kK$-bounded we can describe them
explicitly using Lemma \ref{boltzmann}.
For $\vc{p}$ we obtain
${\blL } = (\lL_1,\lL_2) \in {\mathbb R}^2$ 
such that $\vc{p} = {\bf p}^{\cal V}_{\bf z}$ 
is given by $p^i_1 = e^{\lL _1 + \lL _2 (i /n)} (1 + e^{\lL _1 + \lL _2 (i/n)})^{-1}$ 
(it is convenient to rescale, using $\lL _2/n$ in place of $\lL _2$). 
To determine whether $\vc{p}$ is close to $\wt{\bf p}$,
it will be more convenient to pass to a limit problem
in which closeness is replaced by equality. 
With this in mind, we write
 \[ p^i_1 = (p^{(n)}_{\blL})^i_1
:= f_{\blL}(i/n), \text{ where } f_{\blL}(x) 
= e^{\lL _1 + \lL _2 x} (1 + e^{\lL _1 + \lL _2 x})^{-1}.\]
Similarly, Lemma \ref{boltzmann} gives
$\blP = (\pi _1, \pi _1', \pi _2, \pi _2') \in {\mathbb R}^4$
(using the symmetry between $(0,1)$ and $(1,0)$)
such that $\mu_{\wt{\bf q}} = \mu ^{\widetilde {\cal V}}_{\widetilde {\bf x}}$
is given by $\wt{q}^i_{j,j'} = ({\bf q}^{(n)}_{\blP})^i_{j,j'}
:= g^{\blP}_{j,j'}(i/n)$, where
\begin{gather*}
g^{\blP}_{0,0}(x) = Z_{\blP}(x)^{-1}, \qquad
g^{\blP}_{0,1}(x) = g^{\blP}_{1,0}(x) 
= e^{\pi_1 + \pi_2 x} Z_{\blP}(x)^{-1}, \\
g^{\blP}_{1,1}(x) 
= e^{\pi'_1 + \pi'_2 x} Z_{\blP}(x)^{-1}, \qquad \text{ with }
Z _{\blP}(x) = 1 + 2e^{\pi _1 + \pi _2x} + e^{\pi _1' + \pi _2'x}.
\end{gather*}
The limit marginal problem 
is to characterise $\blL$ and $\blP$ such that 
$f_{\blL}(x) = g^{\blP}_{0,1}(x) + g^{\blP}_{1,1}(x)$.

Next we formulate the constraints on $\blL$ and $\blP$
defined by the parameters $\aA_1,\aA_2,\bB_1,\bB_2$ 
of Conjecture \ref{kalai}, namely
$ \sum_{i \in [n]} p^i_1 (1,i)
= (k,s) = (\aA_1 n, \aA_2 \tbinom{n}{2})$
and $\sum_{i \in [n]} \wt{q}^i_1 (1,i)
= (t,w) = (\bB_1 n, \bB_2 \tbinom{n}{2})$.
The limit versions of these constraints
are $h(\blL) = (\aA_1, \aA_2)$
and $h^*(\blP) = (\bB_1, \bB_2)$, where
\[ h(\blL) = \int _0^1 (1,2x) f_{\blL}(x) dx 
\qquad \text{ and } \qquad
h^*(\blP) = \int _0^1 (1,2x) g^{\blP}_{1,1}(x) dx.\]
The following lemma shows that we can think of
$\blL$ as a reparameterisation of $(\aA_1,\aA_2)$,
and that large finite instances of $[n]_{k,s}$
are well-approximated by the limit.
Recall that a homeomorphism is a continuous bijection
with a continuous inverse.

\begin{lemma} \label{reparam} $ $
\begin{enumerate}
\item $h$ is a homeomorphism between $\mb{R}^2$ and $\LL$.
\item For $(\aA_1,\aA_2) \in \LL$ and large $n$ we have
$\mu^{{\cal V}_n}_{{\bf z}_n} = \mu_{\bf p}$, 
for some ${\bf p} = {\bf p}^{(n)}_{{\blL}^{(n)}}$
where $\blL^{(n)} \to \blL = h^{-1}(\aA_1,\aA_2)$. 
\end{enumerate}
\end{lemma}

\begin{proof} 
We start by noting that $h$ is continuous.
Next we claim that $h(\blL) \in \LL$ 
for all $\blL \in \mb{R}^2$.
To see this, note that $0 \le f_{\blL}(x) \le 1$
for all $x \in [0,1]$. Then given 
$\aA_1 = \int_0^1 f_{\blL}(x) dx$, 
we can bound
$\aA_2 = \int_0^1 2x f_{\blL}(x) dx$
below by
$ \int_0^1 2x 1_{[0,\aA_1]} dx = \aA_1^2$
and above by
$\int_0^1 2x 1_{[1-\aA_1,1]} dx = 2\aA_1-\aA_1^2$,
so $(\aA_1,\aA_2) \in \LL$, as claimed.

Next we claim that the principal minors 
of the Jacobian of $h$ are positive;
this gives injectivity of $h$
by the Gale-Nikaido theorem \cite{GN}, 
and also continuity of $h^{-1}$
by the Inverse Function Theorem.
The Jacobian of $h$ is 
\begin{equation*}
\begin{pmatrix} I(1) & 2I(x)\\ I(x) & 2I(x^2) \end{pmatrix},
\text{ where }
I(g) = \int_0^1 g(x) \Big ( \frac {f_{\blL}(x)}{1 + e^{\lL _1 + \lL _2x}} \Big ) dx.
\end{equation*}
All entries are positive as $f_{\blL}(x)$ is positive.
The determinant $2(I(1)I(x^2)-I(x)^2)$ is positive 
by the Cauchy-Schwarz inequality.
Thus the claim holds.

It remains to prove statement (ii) of the lemma.
Fix $(\aA_1,\aA_2) \in \LL$.
We claim that $|[n]_{k_n,s_n}| \geq (1 + \gamma )^n$ 
for $n^{-1} \ll \gG \ll \aA_1, \aA_2$.
To see this, we fix $\gG \ll \zZ \ll \tT \ll \aA_1, \aA_2$
and construct a $\tT$-bounded measure $\mu_{{\bf p}}$
on $\{0,1\}^S$ for some $S \sub [n]$ such that
$\sum_{i \in S} p_i (1,i) = (k_n,s_n) \pm \zZ (n,n^2)$;
the claim then follows by Lemma \ref{uctg3}.
We let $S = [a-\tT k_n,a+k_n+\tT k_n]$,
for some $a \in [n]$ such that 
$\sum_{i=a+1}^{a+k_n} = s_n \pm n$;
as $\alpha _1 ^2 < \alpha _2 < 2\alpha _1 - \alpha _1^2$
we have $S \sub [n]$ for small $\tT$.
Note that $\aA_2 n^2 = 2s_n + O(n)
= k_n(a+k_n/2) + O(n)$.
We let $p_i = (1+2\tT)^{-1}$ for $i \in S$.
Then $\sum_{i \in S} p_i = k_n + O(1)$
and $\sum_{i \in S} p_i i = s_n + O(n)$,
as required to prove the claim.

Now by Lemma \ref{binary bounded}, 
${\bf p}^{{\cal V}_{n}}_{{\bf z}_n}$ is $\kK$-bounded, 
where $n^{-1} \ll \kK \ll \gG$,
so Lemma \ref{boltzmann} gives 
${\bf p}^{{\cal V}_n}_{{\bf z}_n} = {\bf p}^{(n)}_{\blL ^{(n)}}$.
for some $\blL ^{(n)}$. By $\kK$-boundedness,
$\kappa /2 \leq ({\bf p}^{{\cal V}_n}_{{\bf z}_n})^{i}_1
/ ({\bf p}^{{\cal V}_n}_{{\bf z}_n})^{i}_0 
= e^{\lL ^{(n)}_1 + \lL ^{(n)}_2(i/n)} 
\leq 2\kappa ^{-1}$ for all $i\in [n]$,
so ${\blL}^{(n)} \in [-C,C]^2$,
where $n^{-1} \ll C^{-1} \ll \kK$.
Then $({\blL}^{(n)})$ has a convergent subsequence 
by compactness of $[-C,C]^2$.
Furthermore, any convergent subsequence
of ${\blL}^{(n)}$ has a limit $\blL$
that satisfies $h(\blL)=\aA$,
so is uniquely determined by injectivity of $h$. 
\end{proof}

Next we show a limit theorem for 
the maximum entropy measures for $(t,w)$-intersections
which is somewhat analogous
that in Lemma \ref{reparam} $ii$
for the maximum entropy measures for $[n]_{k,s}$.

\begin{lemma} \label{limitpi}
Suppose ${\bf g} = (\alpha _1, \alpha _2, \beta _1, \beta _2) \in [0,1]^4$
with $(\alpha _1, \alpha _2) \in \Lambda $. 
Write ${\mu }_{\wt{\bf q}^{(n)}} = \mu ^{\wt{\cal V}_n}_{\wt{\bf x}_n}$.
Then either 
\begin{enumerate}
\item there is $(j,j') \in \{0,1\}^2$ such that
 $\min_{i \in [n]} (\wt{q}^{(n)})^i_{j,j'} \to 0$, or 
\item for large $n$ we have
$\wt{\bf q}^{(n)} = {\bf q}^{(n)}_{{\blP}^{(n)}}$,
where $\blP^{(n)}$ converges to some $\blP \in \mb{R}^4$.
\end{enumerate}
Furthermore, the following are equivalent to case $ii$:
\begin{enumerate}
\item there is $\kK>0$ such that
${\mu }_{\wt{\bf q}^{(n)}}$ is $\kK$-bounded for large $n$,
\item there is $\lL>0$ such that
${\mbox {dim}}_{VC}([n]_{k,s}\times _{(t,w)}[n]_{k,s}) 
> \lambda n$ for large $n$.
\end{enumerate}
\end{lemma}

\begin{proof}
We suppose that case (i) does not hold and prove that case (ii) holds.
We can fix $\kK>0$ and a sequence $n_m \to \infty$ such that 
each ${\mu }_{\wt{\bf q}^{(n_m)}}$ is $\kK$-bounded.
By Lemma \ref{boltzmann}, we have $\blP^{n_m} \in {\mathbb R}^4$ 
such that $\wt{\bf q}^{(n_m)} = {\bf q}^{(n_m)}_{\blP^{n_m}}$. 
By $\kK$-boundedness, each $\blP^{n_m} \in [-C,C]^4$
for some $C=C(\kK) \in \mb{R}$. By compactness of $[-C,C]^4$,
we can pass to a convergent subsequence,
so by relabelling we can assume $\blP^{n_m} \to \blP \in \mb{R}^4$.

Note that $n_m^{-1} H({\mu }_{\wt{\bf q}^{(n_m)}})
= \mb{E}_{i \in [n_m]} \sum_{j,j' \in \{0,1\}}
g^{\blP^{n_m}}_{j,j'}(i/n_m) \log_2 g^{\blP^{n_m}}_{j,j'}(i/n_m)
\to H^*(\blP)$, where
\[H^*(\blP) := \sum_{j,j' \in \{0,1\}} \int_0^1 
g^{\blP}_{j,j'}(x) \log_2  g^{\blP}_{j,j'}(x) dx.\]
Furthermore, we claim that $\blP$ is the limit of
any convergent sequence $\blP^{n'_m}$ such that 
$\wt{\bf q}^{(n'_m)} = {\bf q}^{(n'_m)}_{\blP^{n'_m}}$.
To see this, suppose for a contradiction that
$\blP^{n'_m} \to \blP' \ne \blP$. Consider any 
$n^{-1} \ll \lL \ll \kK \ll \|\blP'-\blP\|_1$
and let ${\bf q}' = {\bf q}^{(n)}_{\blP '}$. 
Then $n^{-1} H(\mu_{{\bf q}'}) > H^*(\blP) - \lL$,
$\mu_{{\bf q}'}$ is $(\kK/2)$-bounded and
$\wt{\mc{V}}(\mu_{\bf q}) = \wt{\bf x}'$ with 
$\|\wt{\bf x}' - \wt{\bf x}\|_{\wt{\bf R}} < \lL n$.

As in the proof of Lemma \ref{bounded},
we can modify ${\bf q}'$ to obtain ${\bf q}$
with $\wt{\mc{V}}(\mu_{\bf q}) = \wt{\bf x}$
and $\|{\bf q}-{\bf q}'\|_1 < \lL' n$,
where $\lL \ll \lL' \ll \kK$,
so $n^{-1} H(\mu_{\bf q}) > H^*(\blP) - \kK$.
We deduce that $H^*(\blP') > H^*(\blP) - 2\kK$,
and by symmetry $H^*(\blP) > H^*(\blP') - 2\kK$.
As $\kK$ is arbitrary, $H^*(\blP) = H^*(\blP')$.
Now suppose $n \in (n'_m)_{m \ge 1}$. 
Then $\wt{\bf q} = {\bf q}^{(n)}_{\blP^n}$
and $\mu_{\bf q} \in \mu^{\wt{\mc{V}}}_{\wt{\bf x}}$
with $H(\mu_{\bf q}) > H(\mu_{\wt{\bf q}})-2\kK n$.
However, $n^{-1} \|{\bf q}-\wt{\bf q}\|_1
> \tfrac{1}{2} \|\blP'-\blP\|_1 \gg \kK$,
which contradicts Lemma \ref{perturbmaxent}.
The claim follows.

Now consider any $n^{-1} \ll \lL \ll \kK$
and let ${\bf q} = {\bf q}^{(n)}_{\blP}$.
Then $\mu_{\bf q}$ is $(\kK/2)$-bounded and
$\wt{\mc{V}}(\mu_{\bf q}) = \wt{\bf x}'$
with $\|\wt{\bf x}' - \wt{\bf x}\|_{\wt{\bf R}} < \lL$,
so $\wt{\bf q}^{(n)} = \mu^{\wt{\mc{V}}_n}_{\wt{\bf x}_n}$
is $\lL$-bounded by Theorem \ref{tfae}.
By Lemma \ref{boltzmann}, we have $\blP^n \in {\mathbb R}^4$ 
such that $\wt{\bf q}^{(n)} = {\bf q}^{(n)}_{\blP^n}$. 
By the claim, any convergent subsequence of $(\blP^n)$
converges to $\blP$, so $\blP^n \to \blP$, 
as required for $ii$. 

The first equivalence is immediate from the above proof,
and the second from Theorem \ref{bounded}.
\end{proof}

Our next lemma explains the characterisation 
of the set $\GG$ that appears in Theorem \ref{thm: Kalai}:
it is the set of ${\bf g} = (\aA_1, \aA_2, \bB_1, \bB_2)$
with $(\aA_1,\aA_2) \in \LL$ such that the limit
marginal problem has a solution.
First we complete the definition of $\GG$
by defining the functions $\bB_1,\bB_2:\LL \to \mb{R}$
that appear in the definition of $\GG_1$.
Suppose $(\aA_1,\aA_2) \in \LL$ with $\aA_1 \ne \aA_2$
and  let $\blL = h^{-1}(\aA_1,\aA_2)$.
We define 
\[ (\bB_1(\aA_1,\aA_2),\bB_2(\aA_1,\aA_2))
= \int _0^1 (1,2x) f_{\blL}(x)^2 dx.\]

\begin{lemma}
\label{limiting distributions for (t,w)-intersections in the (k,s)-sets}
Suppose ${\bf g} = (\alpha _1, \alpha _2, \beta _1, \beta _2) \in [0,1]^4$
with $(\alpha _1, \alpha _2) \in \Lambda $. 
Let ${\blL } = h^{-1}(\alpha _1, \alpha _2)$. 
Then ${\bf g} \in \Gamma$ if and only if there is 
$\blP = (\pi _1, \pi _1', \pi _2, \pi _2') \in {\mathbb R}^4$ 
with $h^*(\blP) = (\bB_1, \bB_2)$ and
$f_{\blL}(x) = g^{\blP}_{0,1}(x) + g^{\blP}_{1,1}(x)$.
Furthermore, if ${\bf g} \in \Gamma$
there is a unique such $\blP$,
which we denote $\blP^{\blL}$, and
$\log_2 \bsize{ [n]_{k,s} \times_{(t,w)} [n]_{k,s} }
= H(\mu_{{\bf q}'}) + o(n)$, where 
${\bf q}' = {\bf q}^{(n)}_{\blP^{\blL}}$.
\end{lemma}

\begin{proof}
First suppose that there is $\blP$ with $h^*(\blP) = (\bB_1, \bB_2)$ 
and $f_{\blL}(x) = g^{\blP}_{0,1}(x) + g^{\blP}_{1,1}(x)$.
Then we also have 
$1-f_{\blL}(x) = g^{\blP}_{0,0}(x) + g^{\blP}_{1,0}(x)$.
Setting $y = e^x$ and rearranging, 
we find the polynomial equality
\begin{equation}\label{poly}
e^{\pi _1} y^{\pi _2} + e^{\pi _1'}y^{\pi _2'} = 
e^{\lambda _1} y^{\lambda _2} + e^{\lambda _1 + \pi _1} y^{\lambda _2 + \pi _2}
\end{equation}
for $y\in [1,e]$. Thus one of the following two conditions holds:

	(a) $\pi _2 = \lambda _2$ and $\pi _2' = \lambda _2 + \pi _2$,
	(b) $\lambda _2 = \pi _2' = 0$.

Suppose that $\lambda _2 \neq 0$ and so case (a) holds, giving 
$\pi _2 = \lambda _2$ and  $\pi _2' = 2 \lambda _2$. 
Equating coefficients in (\ref{poly}) gives 
$\lambda _1 = \pi _1$ and $\pi _1' = 2 \lambda _1$,
so $\blP = \blP^{\blL} = (\lL_1,2\lL_1,\lL_2,2\lL_2)$.
Then $g^{\blP}_{1,1}(x) = (e^{\pi _1 + \pi _2x})^2
(1 + e^{\pi _1 + \pi _2x})^{-2} = f_{\blL}(x)^2$,
so ${\bf q}'{}^i_{1,1}=g^{\blP}_{1,1}(i/n)
= f_{\blL}(i/n)^2=(({\bf p}^{(n)}_{\blL})^i_1)^2$,
i.e.\ $\mu_{{\bf q}'}$ is the product
of its two marginals $\mu_{{\bf p}^{(n)}_{\blL}}$.
Note that $\alpha _1 \neq \alpha _2$, as $h$ is injective and 
$h(\lambda _1,0) = (\alpha_1, \alpha _1)$. Thus ${\bf g} \in \Gamma _1$.

Now suppose that $\lambda _2 = \pi _2 =0$. Then the right hand side 
of (\ref{poly}) is constant, so $\pi _2' = 0$, and so 
$e^{\pi _1} + e^{\pi _1'} = e^{\lambda _1} + e^{\lambda _1 + \pi _1}$. 
Thus $\aA_1=\aA_2=f_{\blL}(x) = e^{\lambda _1}(1 + e^{\lambda _1})^{-1}$,
and $\bB_1=\bB_2=g^{\blP}_{1,1}(x) 
= e^{\pi _1'} (1 + 2e^{\pi _1} + e^{\pi _1'})^{-1}$.
Furthermore, $\bB_1 = g^{\blP}_{1,1}(x) 
< g^{\blP}_{1,1}(x) + g^{\blP}_{0,1}(x) = \aA_1$
and $2\alpha _1 - 1 = 
2(e^{\pi _1} + e^{\pi _1'})(1 + 2e^{\pi _1} + e^{\pi _1'})^{-1} - 1 
< e^{\pi _1'}(1 + 2e^{\pi _1} + e^{\pi _1'})^{-1} = \beta _1$.
Thus ${\bf g} \in \Gamma _2$. 

It remains to consider $\lambda _2 = 0$ and $\pi _2 \neq 0$.  
Then case (b) must hold, so $\pi _2' = 0$. 
Equating coefficients in (\ref{poly}) gives 
$\lambda _1 = \pi _1'$ and $\lambda _1 + \pi _1 = \pi _1$, 
so $\lambda _1 = \pi _1' = 0$. 
Then $\alpha _1= \alpha _2 = f_{\blL}(x) = 1/2$.
Furthermore, $(\bB_1,\bB_2) = h^*(\blP)
= h(\wt{\blP})/2$, where $\wt{\blP} = (\pi _1,\pi _2)$,
so $2(\bB_1,\bB_2) \in \LL$. Thus ${\bf g} \in \Gamma _3$.

We conclude that if $h^*(\blP) = (\bB_1, \bB_2)$ 
and $f_{\blL}(x) = g^{\blP}_{0,1}(x) + g^{\blP}_{1,1}(x)$
then ${\bf g} \in \Gamma$. Conversely, if ${\bf g} \in \Gamma$
then the analysis of each case above exhibits 
the unique $\blP=\blP^{\blL}$ satisfying these conditions. 
Indeed, if ${\bf g} \in \GG_1$ we have
${\blP} = (\lambda _1, 2\lambda _1, \lambda _2 , 2\lambda _2)$,
if ${\bf g} \in \GG_2$ we have
$\blP = (\pi _1, \pi _1', 0,0)$
where $g^{\blP}_{1,1} = \bB_1$
and $g^{\blP}_{0,1}(x) = \aA_1-\bB_1$
(when $\aA_1 \ne \bB_1$
this gives two linear equations for
$e^{\pi _1}$ and $e^{\pi' _1}$
that have a unique solution),
and if ${\bf g} \in \GG_3$ we have
$\blP = (\pi_1,0,\pi_2, 0)$, 
where $(\pi_1,\pi_2)=h^{-1}(2\beta _1, 2\beta _2)$

Finally, let 
${\bf q}' = {\bf q}^{(n)}_{\blP^{\blL}}$, 
and note that 
$\mu_{{\bf q}'} \in \mc{M}^{\wt{\cal V}}_{\wt{\bf x}}$, 
so $H(\mu_{{\bf q}'})
\le H(\mu ^{\wt{\cal V}}_{\wt{\bf x}})
\le \log_2 \bsize{ [n]_{k,s} \times_{(t,w)} [n]_{k,s} }
+ o(n)$ by Lemma \ref{uctg3}.
For the inequality in the other direction
we consider each $\GG_i$ separately.

If ${\bf g} \in \GG_1$ we let ${\bf p}'={\bf p}^{(n)}_{\blL}$,
note that $H(\mu_{{\bf q}'})=2 H(\mu_{{\bf p}'})$
and $\log_2 |[n]_{k,s}| = H(\mu_{{\bf p}'}) +o(n)$
by Lemma \ref{uctg3}. Then
$\log_2 \bsize{ [n]_{k,s} \times_{(t,w)} [n]_{k,s} }
\le 2 \log_2 |[n]_{k,s}| = H(\mu_{{\bf q}'}) +o(n)$.

If ${\bf g} \in \GG_2$ we have
$\bsize{ [n]_{k,s} \times_{(t,w)} [n]_{k,s} }
\le \bsize{ \tbinom{[n]}{k} \times_t \tbinom{[n]}{k} }
= \tbinom{n}{t,k-t,k-t,n-2k+t} = 2^{H(\mu_{{\bf q}'})+o(n)}$.

If ${\bf g} \in \GG_3$ we note that if 
$(A,B) \in [n]_{k,s} \times_{(t,w)} [n]_{k,s}$,
where $k = \bfl{ \tfrac{1}{2} n } $, 
$s = \bfl{ \tfrac{1}{2} \tbinom{n}{2} } $, 
$t = \bfl{ \bB_1 n } $ and $w = \bfl{ \bB_2 \tbinom{n}{2} }$,
then $C := (A \cap B) \cup (\ov{A} \cap \ov{B}) \in [n]_{k',s'}$,
where $k' = 2\bB_1 n + O(1)$ and 
$s' = 2\bB_2 \tbinom{n}{2}  + O(n)$.
By Lemma \ref{uctg3}, $\log_2|[n]_{k',s'}|=H(\mu_{{\bf p}''})$,
where ${\bf p}''={\bf p}^{(n)}_{\blP}$
with $\blP=h^{-1}(2\beta _1, 2\beta _2)$.
Note also that $H(\mu_{{\bf q}'})=H(\mu_{{\bf p}''})+n$.
Given $C$, there are at most $2^n$ choices for $(A,B)$,
so $\log_2 \bsize{ [n]_{k,s} \times_{(t,w)} [n]_{k,s} }
\le n + \log_2|[n]_{k',s'}| = H(\mu_{{\bf q}'}) + o(n)$.
In all cases we have the required bound.
\end{proof}

We conclude this subsection 
with the solution to Kalai's conjecture.

\begin{proof}[Proof of Theorem \ref{thm: Kalai}] 
Suppose ${\bf g} = (\aA_1, \aA_2, \bB_1, \bB_2) \in [0,1]^4$
with $(\aA_1,\aA_2) \in \LL$ and  
$n^{-1} \ll \dD \ll \dD' \ll \dD'' \ll \eps \ll \eps' 
\ll \aA_1, \aA_2, \bB_1, \bB_2$.
Let ${\bf g}' \in \GG$ minimise $\|{\bf g} - {\bf g}'\|_1$.
Let $\blL = h^{-1}(\aA_1,\aA_2)$.
By Lemma \ref{reparam} $ii$,
the maximum entropy measure $\mu_{\bf p}$
for $[n]_{k,s}$ is given by
${\bf p} = {\bf p}^{(n)}_{{\blL}^{(n)}}$, 
where $\|\blL - \blL^{(n)}\|_1 \leq \dD$, say.

For $i$, we suppose $\|{\bf g} - {\bf g}'\|_1 \le \dD$ 
and show that ${\bf g}$ is $(n,\dD,\eps)$-Kalai.
Suppose ${\cal A} \sub [n]_{k,s}$ with
$|{\cal A}| \geq (1-\dD)^n|[n]_{k,s}|$.
Then $\mu _{{\bf p }}({\cal A}) \geq (1-\dD')^n$ by Theorem \ref{ldp},
so $\mu _{{\bf p}'}({\cal A}) \geq (1-2\dD')^n$,
where ${\bf p}' = {\bf p}^{(n)}_{\blL }$.

Let ${\bf q}' = {\bf q}^{(n)}_{\blP^{\blL}}$ be given by Lemma 
\ref{limiting distributions for (t,w)-intersections in the (k,s)-sets}.
Then $\mu _{{\bf q}'}$ is $\kK$-bounded, with $\dD \ll \kK \ll 1$,
has marginals $\mu _{{\bf p}'}$ and
${\mathbb E} _{(A,B) \sim \mu _{{\bf q}'}}(|A\cap B|, \sum (A\cap B)) 
= (t,w) \pm 2\dD (n,n^2)$. By Theorem \ref{binary} we have
$\mu _{{\bf q}'}({\cal A} \times _{(t,w)} {\cal A} ) \geq (1-\dD'')^n$. 
By Lemma \ref{uctg3} and the last part of Lemma 
\ref{limiting distributions for (t,w)-intersections in the (k,s)-sets},
this gives $|{\cal A} \times _{(t,w)} {\cal A}| 
\geq (1-2 \dD'')^n 2^{H({\bf q}')} 
\geq (1-\eps )^n |[n]_{k,s} \times _{(t,w)} [n]_{k,s}|$, 
which completes the proof of $i$.

For $ii$, we suppose $\|{\bf g} - {\bf g}'\|_1 \ge \eps'$
and show that ${\bf g}$ is not $(n,\dD,\eps)$-Kalai.
By Lemma \ref{lem: low vc-dim case of trichotomy},
if $\lambda \ll \delta $ and
${\mbox {dim}}_{VC}([n]_{k,s}\times _{(t,w)}[n]_{k,s}) \leq \lambda n$ 
then there is ${\cal A} \subset [n]_{k,s}$ 
with ${\cal A } \times _{(t,w)} {\cal A} = \emptyset $ 
and $|{\cal A}| \geq (1-\delta )^n|[n]_{k,s}|$. 
Thus we may assume that case $ii$ of Lemma \ref{limitpi} holds,
so there is $\kK>0$ such that
${\mu }_{\wt{\bf q}}$ is $\kK$-bounded for large $n$. By Lemma \ref{boltzmann}, there is $\blP \in {\mathbb R}^4$ 
such that $\wt{\bf q} = {\bf q}^{(n)}_{\blP}$. 
As $\mu _{\wt{\bf q}}$ is $\kappa $-bounded, 
we have ${\blP} \in [-C, C]^4$,
where $n^{-1} \ll C^{-1} \ll \kK$.

Consider $\phi:[-C,C]^4 \to \mb{R}$ defined by
$\phi(\blS) = \|(\bB_1,\bB_2) - h^*(\blS)\|_1 
+ \int_0^1 |f_{\blL}(x) - g^{\blS}_{1,0}(x) - g^{\blS}_{1,1}(x)| dx$.
As ${\bf g} \notin \Gamma$,
we have $\phi(\blS)>0$ for all $\blS$ by Lemma 
\ref{limiting distributions for (t,w)-intersections in the (k,s)-sets}
so by compactness there is some $n^{-1} \ll \tT \ll C^{-1}$
such that  $\phi(\blS) > \tT$ for all $\blS \in [-C,C]^4$;
in particular, this holds for $\blS = \blP$.
As $\mu_{\wt{\bf q}} \in \mc{M}^{\wt{\cal V}}_{\wt{\bf x}}$
we have $\|(\bB_1,\bB_2) - h^*(\blP)\|_1 = O(1/n)$,
so $\int_0^1 |f_{\blL}(x) - g^{\blP}_{1,0}(x) - g^{\blP}_{1,1}(x)| dx 
> \tT - O(1/n)$. 

Now we translate back from the limit to the finite setting. 
The previous inequality implies
$\sum_{i=1}^n |f_{\blL}(i/n) - g^{\blP}_{1,0}(i/n) - g^{\blP}_{1,1}(i/n)|
> \tT n - O(1)$. Recalling that $\wt{\bf q} = {\bf q}^{(n)}_{\blP}$,
$\mu_{\wt{\bf q}}$ has marginals $\mu_{\wt{\bf p}}$, 
and ${\bf p}' = {\bf p}^{(n)}_{\blL }$ we have
$g^{\blP}_{1,0}(i/n) - g^{\blP}_{1,1}(i/n)
= \wt{q}^i_{1,0} + \wt{q}^i_{1,1} = \wt{p}^i_1$
and $f_{\blL}(i/n) = (p')^i_1$.
As $\tT$ depends only on ${\bf g}$,
we can assume $n^{-1} \ll \dD \ll \eps \tT$.
Then ${\bf p} = {\bf p}^{(n)}_{{\blL}^{(n)}}$, 
where $\|\blL - \blL^{(n)}\|_1 \leq \dD$,
and so $\| {\bf p} - \wt{\bf p} \| > \tT n/2$.

Applying Lemma \ref{lem: high vc-dim case of trichotomy} $ii$, with 
$\delta _1 = \theta /2$ and $\delta = \eps $, 
we find ${\cal B}_{full} \subset [n]_{k,s}$ 
so that ${\cal A} = [n]_{k,s} \setminus {\cal B}_{full}$ satisfies 
$|{\cal A}| \geq (1-o(1))|[n]_{k,s}| > (1-\dD)^n |[n]_{k,s}|$
and $|{\cal A} \times _{(t,w)} {\cal A}| 
< (1-\eps)^n|[n]_{k,s} \times _{(t,w)}[n]_{k,s}|$.
This completes the proof. \end{proof}

\subsection{Uniqueness in higher dimensions}
\label{subsect: extension of Kalai}

In this subsection we illustrate how the method used
to prove Theorem \ref{thm: Kalai} can be applied in a broader context.
Throughout this subsection we work with the following setting.
\begin{itemize}
\item
Fix $\blA = (\aA_d) _{d\in [D]}$
and $\blB = (\bB_d) _{d\in [D]}$ in $(0,1)^D$. \\
For all $n \in \mb{N}$ 
let ${\bf z}_n = (\bfl{\aA_d n^2} )_{d\in [D]}$ 
and ${\bf w}_n = (\bfl{\aA_d n^2} )_{d\in [D]}$.
\item
Suppose $\mc{V}_n = ({\bf v}^{(n)}_i)_{i\in [n]}$ 
are $(n,\{0,1\})$-arrays in $[n]^D$ such that
$(\mc{V}_n,{\bf z}_n)$ is robustly generated 
and robustly generic.
\item
Write $\mc{X}_n = (\{0,1\}^n)^{\mc{V}_n}_{{\bf z}_n})$ 
and suppose that $|\mc{X}_n| > (1 + \eta )^n$,
where $\eta = \eta (\blA)>0$ is fixed.
\item
The arrays $\mc{V}_n$ have a `scaling limit':
there is a positive measurable function $p:[0,1]^D \to \mathbb R$
with $\int _{[0,1]^D} p({\bf x}) d{\bf x}  = 1$ 
such that for any measurable set $B \subset [0,1]^D$ we have
\[ \lim_{n \to \infty} n^{-1} \big | \{ i \in [n]:
  n^{-1} {\bf v}^{(n)}_i \in B \} \big | = \int_B p({\bf x}) d{\bf x}. \]
\end{itemize}
The assumption that $(\mc{V}_n,{\bf z}_n)$ is robustly generic
is in fact redundant, as it can be shown to follow from the
scaling limit assumption, but for the sake of brevity 
we omit this deduction.

We say that $(\blA,\blB)$ is {\em $(n,\dD,\eps)$-good}
if the corresponding $\mc{V}_n$-intersection problem
exhibits `full supersaturation' analogous
to that in Conjecture \ref{kalai}, 
i.e.\ any $\mc{A} \subset \mc{X}_n$ with 
$|\mc{A}| \geq (1-\delta )^n |\mc{X}_n|$ satisfies
$\big| ({\cal A}\times {\cal A})^{(\mc{V}_n)_{\cap}}_{{\bf w}_n} \big| 
\geq (1-\eps )^n \big|
(\mc{X}_n \times \mc{X}_n)^{(\mc{V}_n)_{\cap}}_{{\bf w}_n} \big|$.
We will outline the proof of the following analogue 
of Theorem \ref{thm: Kalai}, which shows that 
if we exclude the case of `uniformly random sets'
(i.e.\ $\blA \neq \big (1/2 \big )_{d\in [D]}$)
then `full supersaturation' only occurs
for one specific value of $\blB$.

\begin{theo}\label{thm: Kalai+}
In the above setting, if $\blA \neq \big (1/2 \big )_{d\in [D]}$
then there is $\blB^* = \blB^{*}(\blA ) \in (0,1)^D$
such that for $n^{-1} \ll \dD \ll \eps \ll \eps' \ll \blA$,
\begin{enumerate}
\item if $\|\blB - \blB^*\|_1 \le \dD$ 
then $(\blA,\blB)$ is $(n,\dD,\eps)$-good, and
\item if $\|\blB - \blB^*\|_1 \ge \eps'$ 
then $(\blA,\blB)$ is not $(n,\dD,\eps)$-good. 
\end{enumerate}
\end{theo}
	
Similarly to the previous subsection,
we wish to determine when
$\mu_{\wt{\bf q}} =\mu ^{\wt{\cal V}}_{\wt{\bf x}}$
(with $\wt{\cal V}$ and $\wt{\bf x}$ 
as in Definition \ref{tilde})
has marginals close to
$\mu_{\vc{p}} = \mu^{{\cal V}}_{{\bf z}}$
(here we are omitting the subscript $n$ from our notation).
If these measures are $\kK$-bounded,
Lemma \ref{boltzmann} 
gives $\blL \in {\mathbb R}^D$ such that 
\[ p^i_1 = (p^{(n)}_{\blL})^i_1 := f_{\blL}({\bf v}^{(n)}_i/n),
\text{ where }
f_{\blL }({\bf x}) = e^{\blL \cdot {\bf x}} (1 + e^{\blL \cdot {\bf x}})^{-1},\]
and $\blP_1, \blP_2 \in {\mathbb R}^D$ such that 
$\wt{q}^i_{j,j'} = ({\bf q}^{(n)}_{\blP_1,\blP_2})^i_{j,j'}
:= g^{\blP_1,\blP_2}_{j,j'}({\bf v}^{(n)}_i/n)$, where
\begin{gather*}
g^{\blP_1,\blP_2}_{0,0}({\bf x}) = Z_{\blP_1,\blP_2}({\bf x})^{-1}, \qquad
g^{\blP_1,\blP_2}_{0,1}({\bf x}) = g^{\blP_1,\blP_2}_{1,0}({\bf x}) 
= e^{\blP _1 \cdot {\bf x}} Z_{\blP_1,\blP_2}({\bf x})^{-1}, \\
g^{\blP_1,\blP_2}_{1,1}({\bf x}) 
= e^{\blP _2 \cdot {\bf x}} Z_{\blP_1,\blP_2}({\bf x})^{-1}, \qquad \text{ with }
Z _{\blP_1,\blP_2}({\bf x}) = 1 + 2e^{\blP _1 \cdot {\bf x}} + e^{\blP _2 \cdot {\bf x}}.
\end{gather*}
Again we study the marginal problem
for $\mu_{\wt{\bf q}}$ and $\mu_{\vc{p}}$
via the limit marginal problem 
of characterising $\blL$ and $\blP$ such that 
$f_{\blL}({\bf x}) = g^{\blP_1,\blP_2}_{0,1}({\bf x}) + g^{\blP_1,\blP_2}_{1,1}({\bf x})$.
The constraints are
${\bf z}_n = \sum_{i \in [n]} p^i_1 \bf{v}^{(n)}_i$ and
${\bf w}_n = \sum_{i \in [n]} \wt{q}^i_{1,1} \bf{v}^{(n)}_i$.
The limit versions
are $h(\blL) = \blA$ and $h^*(\blP_1,\blP_2) = \blB$, where
\[ h(\blL) = \int _{[0,1]^D} 
{\bf x} f_{\blL}({\bf x}) p({\bf x}) d{\bf x}
\qquad \text{ and } \qquad
h^*(\blP_1,\blP_2) = \int _{[0,1]^D} 
{\bf x} g^{\blP_1,\blP_2}_{1,1}({\bf x}) p({\bf x}) d{\bf x}.\]
Our next lemma is analogous to Lemma \ref{reparam}.

\begin{lemma} \label{reparam2} $ $
\begin{enumerate}
\item $h$ is a homeomorphism between $\mb{R}^D$ and $h(\mb{R}^D)$.
\item For large $n$ we have
$\mu^{{\cal V}_n}_{{\bf z}_n} = \mu_{\bf p}$, 
for some ${\bf p} = {\bf p}^{(n)}_{{\blL}^{(n)}}$
where $\blL^{(n)} \to \blL = h^{-1}(\blA)$. 
\end{enumerate}
\end{lemma}

We omit the proof of Lemma \ref{reparam2}, 
as it is the same as that of Lemma \ref{reparam},
except in one detail which we will now check,
namely that the principal minors of the Jacobian of $h$ are positive.  
To see this, note that the Jacobian $J$ has entries
\begin{equation*}
J_{i,j} = \int _{{\bf x} \in [0,1]^D} 
x_i x_j \Big ( \frac {f_{\blL }({\bf x})}{1 + e^{\blL .{\bf x}}} \Big )
p({\bf x}) d{\bf x}.  
\end{equation*}
For any ${\bf y} \in {\mathbb R}^D$ we have 
${\bf y}^T J {\bf y} = \int _{{\bf x} \in [0,1]^D} 
|\langle {\bf x},{\bf y}\rangle |^2 
f_{\blL }({\bf x})(1 + e^{\blL .{\bf x}})^{-1} p({\bf x}) d{\bf x}$. 
As $f_{\blL }$ and $p$ are positive,
we have  ${\bf y}^T J {\bf y} > 0$ 
whenever ${\bf y} \ne 0$, as required.
We also have the following analogue of Lemma \ref{limitpi};
again, we omit the similar proof.

\begin{lemma} \label{limitpi2}
Write ${\mu }_{\wt{\bf q}^{(n)}} = \mu ^{\wt{\cal V}_n}_{\wt{\bf x}_n}$.
Then either 
\begin{enumerate}
\item there is $(j,j') \in \{0,1\}^2$ such that
 $\min_{i \in [n]} (\wt{q}^{(n)})^i_{j,j'} \to 0$, or 
\item for large $n$ we have
$\wt{\bf q}^{(n)} = {\bf q}^{(n)}_{\blP_1^{(n)},\blP_2^{(n)}}$,
where $(\blP_1^{(n)},\blP_2^{(n)})$ 
converges to some $(\blP_1,\blP_2) \in \mb{R}^{2D}$.
\end{enumerate}
Furthermore, the following are equivalent to case $ii$:
\begin{enumerate}
\item there is $\kK>0$ such that
${\mu }_{\wt{\bf q}^{(n)}}$ is $\kK$-bounded for large $n$,
\item there is $\lL>0$ such that
${\mbox {dim}}_{VC}(
(\mc{X}_n \times \mc{X}_n)^{(\mc{V}_n)_{\cap}}_{{\bf w}_n}
) > \lambda n$ for large $n$.
\end{enumerate}
\end{lemma}

The uniqueness in Theorem \ref{thm: Kalai+} is explained by the following lemma
which solves the limit marginal problem.

\begin{lemma} 
\label{lem: value of beta}
Suppose $h(\blL)=\blA \neq (1/2)_{d \in [D]}$.
Then there is unique $\blB^* \in [0,1]^D$ 
such that there is 
$(\blP _1, \blP _2) \in {\mathbb R}^{D} \times {\mathbb R}^D$ 
with $h^*(\blP_1,\blP_2) = \blB^*$ and
$f_{\blL}({\bf x}) = g^{\blP_1,\blP_2}_{0,1}({\bf x}) + g^{\blP_1,\blP_2}_{1,1}({\bf x})$.
Furthermore, $(\blP _1, \blP _2)$ is unique, and 
$\log_2 \bsize{ (\mc{X}_n \times \mc{X}_n)^{(\mc{V}_n)_{\cap}}_{{\bf w}_n} }
= H(\mu_{{\bf q}'}) + o(n)$, where 
${\bf q}' = {\bf q}^{(n)}_{\blP_1,\blP_2}$.
\end{lemma}

\begin{proof}
As $h$ is injective and 
$h(0) = (1/2)_{d\in [D]} \neq \blA $, we have $\blL \neq 0$.
Rearranging $1-f_{\blL}({\bf x}) 
= g^{\blP_1,\blP_2}_{0,0}({\bf x}) + g^{\blP_1,\blP_2}_{0,1}({\bf x})$
gives $e^{\blP _1 \cdot {\bf x}} + e^{\blP _2 \cdot {\bf x}} = 
e^{\blL \cdot {\bf x}} + e^{(\blP _1 + \blL ) \cdot {\bf x}}$,
so (a) $\blP _1= \blL $ and $\blP _2 = \blP _1+ \blL $,
or (b) $\blP _1= \blP _1+ \blL $ and $\blL  = \blP _2$.
However, (b) cannot hold, as ${\blL} \neq 0$.
Thus $\blP _1= \blL $ and $\blP _2 = 2\blP _1$, 
so $g^{\blP_1,\blP_2}_{1,1}({\bf x}) = f_{\blL}({\bf x})^2$
and $\blB ^*= \int _{{\bf x} \in [0,1]^D} {\bf x} 
f_{\blL }({\bf x})^2 p({\bf x}) d{\bf x}$.
Uniqueness of $(\blP _1, \blP _2)$ is clear,
and the final estimate follows in the same way
as the case ${\bf g} \in \GG_1$ of Lemma
\ref{limiting distributions for (t,w)-intersections in the (k,s)-sets}.
\end{proof}

Given the above lemmas, the proof of Theorem \ref{thm: Kalai+} 
is very similar to that of Theorem \ref{thm: Kalai},
so we omit the details.

\section{Optimal supersaturation} \label{sec:tightsupersat}

In this section we characterise the optimal level of supersaturation
for ${\cal V}$-intersections in terms of a certain optimisation problem;
as outlined in subsection \ref{subsec: supersaturation},
this corresponds to the optimal choice of measure
satisfying the hypotheses of Theorem \ref{binary},
i.e.\ determining $H_{max}$ in the following setting,
which we adopt throughout this section.
\begin{itemize}
\item Let $0 < n^{-1}, \dD \ll
\kappa \ll  \gG _1, \gG _1' \ll \gG _2, \gG _2' 
\ll \eps \ll \alpha, D^{-1}, C^{-1}, k^{-1}$.
\item Suppose ${\cal V} = ({\bf v}_i : i\in [n])$ and
each ${\bf v}_i \in {\mathbb Z}^D$ is ${\bf R}$-bounded,
where $\vc{R} \in \mb{R}^D$ with $\max_d R_d < n^C$.
\item Suppose ${\cal V}$ is $\gG _i'$-robustly $(\gG _i, {\bf R})$-generic 
and $\gG _i$-robustly $({\bf R},k)$-generating for $i = 1,2$.
\item Let ${\bf z} \in {\mathbb Z}^D$ and 
$\mc{X} = (\{0,1\}^n)^{\mc{V}}_{\vc{z}}$ 
with $|{\cal X}| \geq (1+ \alpha )^n$.
Write $\mu_{\bf p} := {\mu }^{\cal V}_{\bf z}$.
\item Let ${\bf w} \in {\mathbb Z}^D$ and let ${\mathbf Q}$ 
denote the set of ${\bf q}$ such that $\mu _{\bf q}$
is a $\kK$-bounded product measure on $(\{0,1\}\times \{0,1\})^n$
with both marginals $\mu_{\bf p}$ 
and ${\cal V}_{\cap }(\mu_{\bf q}) = {\bf w}$. 
\item Let $H_{max} = \max_{{\bf q} \in {\mathbf Q}} H(\mu_{\bf q})$
if ${\mathbf Q} \ne \es$ or $H_{max} =  0$ if ${\mathbf Q} = \es$.
\end{itemize}
The main result of this section is as follows.

\begin{theo} \label{optsupersat}
In the above setting, 
\begin{enumerate}
\item if $\mc{A} \sub \mc{X}$ with $|\mc{A}| > (1-\dD)^n |\mc{X}|$ then
$|(\mc{A} \times \mc{A})^{\mc{V}_\cap}_{\vc{w}}| \geq \bfl{ (1-\eps)^n 2^{H_{max}} }$.
\item there is $\mc{A} \sub \mc{X}$ with $|\mc{A}| > (1-\eps )^n |\mc{X}|$ and
$|(\mc{A} \times \mc{A})^{\mc{V}_\cap}_{\vc{w}}| \leq (1+\eps)^n 2^{H_{max}}$.
\end{enumerate}
\end{theo}

We start by giving the short deduction 
of statement $i$ from Theorem \ref{binary}.
The hypotheses of the latter hold by Lemma \ref{binary bounded},
and Theorem \ref{ldp} gives $\mu_{\bf p}(\mc{A}) > (1-\dD')^n$,
where $\dD \ll \dD' \ll \kK$. We can assume ${\mathbf Q} \ne \es$,
and Theorem \ref{binary} applied to ${\bf q} \in {\mathbf Q}$ gives 
$\mu_{\vc{q}}((\mc{A} \times \mc{A})^{\mc{V}_\cap}_{\vc{w}}) > (1-\eps/2)^n$.
Also, by Lemma \ref{uctg1}, 
$\mc{B} := \{ \vc{x} \in (\{0,1\} \times \{0,1\})^n: 
\log_2 \mu_{\vc{q}}(\vc{x}) \notin - H_{max} \pm \dD n \}$
has $\mu_{\vc{q}}(\mc{B}) \le (1-\dD^3)^n$,
so $|(\mc{A} \times \mc{A})^{\mc{V}_\cap}_{\vc{w}}|
\ge 2^{H_{max}-\dD n} \mu_{\vc{q}}(
(\mc{A} \times \mc{A})^{\mc{V}_\cap}_{\vc{w}} \sm \mc{B} )
\ge (1-\eps)^n 2^{H_{max}}$, as required.

The remainder of the section will be occupied
with the proof of statement $ii$.
A key idea is the use of `empirical measures',
which we will now introduce. 
First we note by Lemma \ref{binary bounded}
that ${\bf p}$ is $\kK'$-bounded, where 
\[\gG _1, \gG _1' \ll \kK' \ll \gG _2, \gG _2'.\]
We fix a partition of $[n]$ 
into sets $S_1,\ldots ,S_M$ so that 
\[ |p_i - p_j| \leq \kappa \text{ and }
\|{\bf v}_i - {\bf v}_j \|_{ \bf R} \leq \kappa 
\text{ if } i,j \in S_m \text{ for some } m \in [M]. \]
This can be achieved with $M \leq (2\kappa ^{-1} +2)^{D+1}$. 
We define the {\em type} of $A \sub [n]$ 
as ${\bf k}(A)=(|A \cap S_1|,\dots,|A \cap S_M|)$.
We let ${\bf k}=(k_1,\dots,k_M)$ 
be the most common type of sets in $\mc{X}$,
write \[ \mc{B} = \{ A \in \mc{X}: {\bf k}(A) = {\bf k} \},\]
and define the empirical measure $\mu_{\bf p '}$ by
\[ (p')^i_1 = k_m/|S_m| \text{ for all } m \in [M], i \in S_m. \]
Note that $|\mc{B}| \ge |\mc{X}|/n^M$.
The following lemma shows that $\mu_{\bf p '}$
is a good approximation to the maximum entropy measure $\mu_{\bf p}$.

\begin{lemma} \label{empiricalp}
$\|{\bf p} - {\bf p}'\|_1 \leq \kK n$.
\end{lemma}

\begin{proof}
Let ${\cal E}$ be the set of $A \subset [n]$ such that some
$| |A\cap S_m| - \sum _{i\in S_m} p_i | > \tfrac{\kappa n}{2M} $.
Let $\kK \ll \kK_1 \ll \kK_2 \ll \kK'$.
Then $\mu _{\bf p}({\cal E}) \leq (1 -\kK_1)^n$ by Chernoff's inequality,
so $|{\cal E} \cap {\cal X}| < (1 - \kK_2)^n |{\cal X}| < |{\cal B}|$ 
by Theorem \ref{ldp}. We deduce 
$|k_m - \sum _{i\in S_m} p_i | \leq \kK n/2M$ for all $m\in [M]$,
so $\sum _{m \in [M]} |k_m - \sum _{i\in S_m} p_i| \leq \kK n/2$;
the lemma follows.
\end{proof}

We use a similar construction of an empirical measure 
that represents ${\bf w}$-intersections.
Let $G$ be the graph with $V(G)={\cal B}$
where $AB \in E(G)$ if $|A \cap B|_{\mc{V}} = {\bf w}$. 
We define the {\em type} of $AB \in E(G)$
as ${\bf t}_{AB} = (t_1,\ldots ,t_M)$,
where $t_m = |A \cap B \cap S_m|$ for $m\in [M]$.
A type ${\bf t}$ gives rise to a measure $\mu _{\bf q({\bf t})}$, 
where for $i\in S_m$ we define 
\[{q}({\bf t})^i_{1,1} = {t_m}/{|S_m|}, \
{q}({\bf t})^i_{1,0} = {q}({\bf t})^i_{0,1} = {(k_m - t_m)}/{|S_m|}
\text{ and } {q}({\bf t})^i_{0,0} =  {(|S_m| - 2k_m + t_m)}/{|S_m|}.\] 
Note that each $\mu _{{\bf q}({\bf t})}$ 
has both marginals $\mu _{{\bf p}'}$ and 
$\| \mc{V}_{\cap}(\mu _{{\bf q}({\bf t})}) - {\bf w} \|_{\bf R} 
\leq \|{\bf p} - {\bf p}'\|_{1} \leq \kappa n$.

We can assume
\begin{equation} \label{Hmax}
H(\mu_{\bf q}) < \log _2 |E(G)| - \eps n/2
\text{ for all } {\bf q} \in {\bf Q},
\end{equation}
otherwise the proof is complete.
We fix a type $\wt{\bf t}$ occurring at least 
$e(G) / n^{2M}$ times and set $\wt{\bf q} = {\bf q}(\wt{\bf t})$.
Then $H(\mu_{\wt{\bf q}}) \ge \log_2 (e(G) / n^{2M})$,
so $\wt{\bf q} \notin {\bf Q}$ by (\ref{Hmax}).
The following lemma will show that all empirical measures
associated to edges of $G$ are close to $\wt{\bf q}$;
we will then use this and $\wt{\bf q} \notin {\bf Q}$
in Lemma \ref{indepG} to find a large independent set
in $G$, which will complete the proof of Theorem \ref{optsupersat}.

We fix $\lL$ with $\gG_1,\gG'_1 \ll \lL \ll \kK'$.

\begin{lemma} \label{closetypes}
Suppose $CD \in E(G)$ has type ${\bf t}'$.
Then $\|{\wt{\bf q}} - {\bf q}({\bf t}') \|_1 < \lL n$.
\end{lemma}

For the proof
we require the following bound analogous to (\ref{Hmax})
for a wider class of measures.

\begin{lemma}
\label{lem: approximate kappa bounded distribution is enough}
Let $\mu _{{\bf q}'}$ be a $\lL$-dense product measure on 
$(\{0,1\}\times \{0,1\})^n$ with marginals $\mu _{\bf p'}$ 
and ${\cal V}_{\cap }(\mu _{{\bf q}'}) = {\bf w}'$
with $\|{\bf w} - {\bf w}'\|_{\bf R} \leq \kappa n$.
Then $H(\mu_{\bf q '}) < \log _2 |E(G)| - \eps n/3$.
\end{lemma}

\begin{proof}
We will obtain the required bound from (\ref{Hmax})
a measure in ${\bf Q}$ close to $\mu _{\bf q '}$.
Recall that ${\bf p '}$ is $\kK'$-bounded
and $\| {\bf p} -{\bf p '} \|_1 \le \kK n$
Consider ${\bf q}''$ that minimises $\| {\bf q}'' -{\bf q '} \|_1$
subject to $\mu _{\bf q}$ being $\kK$-bounded
and having marginals $\mu _{\bf p}$.
For each $i$ we can construct 
$\{(q'')^i_{j,j'}\}$ from $\{(q')^i_{j,j'}\}$
by moving probability mass $|p_i-p'_i|$
to create the correct marginals,
and moving a further mass of at most $2\kK$
while maintaining the same marginals
to ensure $\kK$-boundedness.
Therefore $\|{\bf q}'' - {\bf q}'\|_1 \le 6 \kK n$.

Now we will perturb 
${\bf q}''$ to obtain ${\bf q} \in {\bf Q}$, 
i.e.\ we maintain $\kK$-boundedness
and the same marginals $\mu _{\bf p}$,
and obtain ${\cal V}_{\cap }(\mu _{{\bf q}}) = {\bf w}$.

As $\mu _{{\bf q}'}$ is $\lL$-dense and $\kK \ll \lL$
there is $S \sub [n]$ with $|S| \geq \lambda n/2$ such that 
$(q'')^i_{j,j'} \geq \lL/2$ for all $i\in S$ and $j,j' \in \{0,1\}$. 
As ${\cal V}$ is $\gG_1'$-robustly $(\gG_1, {\bf R})$-generic, 
and $\lambda \gg \gamma _1'$ we can find 
$M \geq |S|/2D \geq \lambda n/4D$ disjoint sets 
$I_1,\ldots ,I_M \subset S$, with $|I_m| = D$ 
and $|\det ({\cal V}_{I_m})| \geq 
\gamma _1 R_1\cdots R_D$ for all $m\in [M]$.

Write ${\cal V}_{\cap }(\mu _{{\bf q}''}) = {\bf w}''$,
and note that $\| {\bf w}'' - {\bf w}\|_{\bf R} 
\le \|{\bf q}'' - {\bf q}'\|_1 \le 6 \kK n$.
Then ${\bf u} = ({\bf w} - {\bf w}'') / M$ has 
$\| {\bf u} \|_{\bf R} \le 24D \kK \lL^{-1} < \sqrt{\kK}$.
Applying Cramer's rule as in Lemma \ref{bounded}, 
for each $m\in [M]$ we find coefficients $b_i$ 
with $\sum _{i\in I_m} b_i {\bf v}_i = {\bf u}$
and $|b_i| \le \sqrt{\kK} D! \gG_1^{-1}$.

Now we obtain ${\bf q}$ from ${\bf q}''$
where for each $i\in \cup _{m\in [M]} I_m$
we let $q^i_{1,1} = (q'')^i_{1,1} + b_i$, 
$q^i_{0,1} = q^i_{1,0} = (q'')^i_{1,0} - b_i$ 
and $q^i_{1,1} = (q'')^i_{1,1} + b_i$,
and $q^{i}_{j,j'} = (q'')^{i}_{j,j'}$ otherwise.
By construction ${\bf q} \in {\bf Q}$
and $\|{\bf q}' - {\bf q}\|_1 
 \leq \|{\bf q}' - {\bf q}''\|_1 
+ \|{\bf q}'' - {\bf q}\|_1 < \kK^{1/3} n$.
The lemma now follows from (\ref{Hmax}). 
\end{proof}

\nib{Proof of Lemma \ref{closetypes}.}
Suppose for a contradiction that 
$\|{\wt{\bf q}} - {\bf q}({\bf t}') \|_1 \ge \lL n$.
Consider the interpolation
${\bf q}' = \lL {\bf q}({\bf t}') + (1-\lL) {\wt{\bf q}}$.
Recall that any $\mu _{{\bf q}({\bf t})}$ 
has marginals $\mu _{{\bf p}'}$ and satisfies
$\| \mc{V}_{\cap}(\mu _{{\bf q}({\bf t})}) - {\bf w} \|_{\bf R} 
\leq \kappa n$, so $\mu_{\bf q '}$ has the same properties.
Also, as $H(\mu_{\wt{\bf q}}) \ge \log_2 (e(G) / n^{2M})$
we have $H(\mu_{\bf q '}) > \log _2 |E(G)| - \eps n/3$.

As $\|{\wt{\bf q}} - {\bf q}({\bf t}') \|_1 \ge \lL n$
we can find $S \sub [n]$ with $|S| \ge \lL n/2$ such that
$\sum_{j,j' \in \{0,1\}} |\wt{q}^i_{j,j'}-q({\bf t}')^i_{j,j'}|
\ge \lL/2$ for all $i \in S$. As ${\wt{\bf q}}$ and
${\bf q}({\bf t}')$ have the same marginals $\mu _{{\bf p}'}$
we have $|\wt{q}^i_{j,j'}-q({\bf t}')^i_{j,j'}| \ge \lL/8$ 
for all $i \in S$ and $j,j' \in \{0,1\}$.
For each such $i,j,j'$ we deduce
$(q')^i_{j,j'} \ge \lL^2/8$. However, this contradicts Lemma
\ref{lem: approximate kappa bounded distribution is enough}
(with $\lambda ^2/8$ in place of $\lambda $). \qed

\medskip

The following lemma completes the proof of Theorem \ref{optsupersat}.

\begin{lemma} \label{indepG}
There is $\mc{A} \sub \mc{X}$
with $|\mc{A}| \ge (1-\eps )^n|{\cal X}|$ and 
$({\cal A } \times {\cal A})^{{\cal V}_{\cap }}_{\bf w} = \es$.
\end{lemma}

\nib{Proof.}
We can assume $e(G) \geq (1+\eps /2)^n|{\cal B}|$,
as otherwise by Tur\'an's theorem 
(\cite{T41}, see also \cite[IV.2]{MGT})
$G$ contains an independent set ${\cal A}$ 
of order $(1+\eps /2)^{-n}|{\cal B}|/2 \geq (1-\eps )^n|{\cal  X}|$.
As $\log_2 |{\cal  X}| \ge H(\mu_{\vc{p}}) - \kK n$
by Lemma \ref{uctg3} and $\|{\bf p} - {\bf p}'\|_1 \leq \kK n$
by Lemma \ref{empiricalp} we deduce
$H(\mu_{\wt{\bf q}}) \ge \log_2 (e(G) / n^{2M})
\ge  H(\mu _{\bf p '}) + \eps n/4$.
We have $H(\mu_{\wt{\bf q}}) = \sum_{i \in [n]} H({{\wt{\bf q}}^i}) $
and $H(\mu _{\bf p '}) = \sum_{i \in [n]} H({{\bf p}'^i})$,
where each $H({{\wt{\bf q}}^i}) \le \log_2 4 = 2$,
so there is $T \sub [n]$ with $|T| \geq \eps n/16$ 
such that $H({{\wt{\bf q}}^i}) \geq H({{\bf p}'^i}) + \eps /16$.
As $\mu_{\wt{\bf q}}$ has marginals $\mu _{\bf p '}$
we deduce ${\wt {\bf q}}^i_{0,1} = {\wt {\bf q}}^i_{1,0} 
> \eps^2$ for all $i\in T$. 
Let $T_1 = \{i\in T: {\wt q}^i_{1,1} < \lL \}$ 
and $T_0 = \{i\in T: {\wt q}^i_{0,0} < \lL \}$. By Lemma 
\ref{lem: approximate kappa bounded distribution is enough} 
we have $|T_1| \geq |T|/4$ or $|T_0| \geq |T|/4$.
 
	\vspace{1mm}	
	
	\noindent \textbf{Case 1:} $|T_1| \geq |T|/4$.
	
	\vspace{1mm}

Let ${\cal B}^* = \{ B \in {\cal B}: |B \cap T_1| \geq {\kK'|T_1|}/{2} \}$. 
As ${\bf p}$ is ${\kappa }'$-bounded and 
$|T_1| \geq \eps n/64$ we have 
$\mu _{\bf p}({\cal B} \sm {\cal B}^*) \leq (1-c_{\kappa '})^n$,
which by Theorem \ref{ldp} gives
$|{\cal B}^*| \geq |{\cal B}|/2$. 
Let $G^*=G[B^*]$ denote the induced subgraph
of $G$ with vertex set ${\cal B}^*$. 

We claim that for all $AB \in E(G^*)$ we have 
$|A\cap B \cap T_1| < 4\lambda ^{1/2}n$. 
Indeed, suppose for a contradiction that
 $|A\cap B \cap T_1| \geq  4\lambda ^{1/2}n$.
Let $J$ be the set of $m\in [M]$ 
with $|T_1 \cap S_m| \geq 2\lambda ^{1/2}|S_m|$.
Then $\sum _{m\in J} |S_m| \geq 2\lambda ^{1/2}n$.
For all $i\in \bigcup _{m\in J} S_m$ we have 
$q({\bf t}_{AB})^i_{1,1} - {\wt q}^i_{1,1}
\geq 2\lambda ^{1/2} - \lambda > \lambda ^{1/2}$,
by definition of $J$ and $T_1$.
But then $\|{\wt {\bf q}} - {\bf q}({\bf t}_{AB})\|_1 
\geq \lambda ^{1/2} \sum _{m\in J} |S_m| > \lambda n$.
This contradicts Lemma \ref{closetypes}, so the claim holds.

Therefore, for any $U \sub T_1$ of size 
$u = \lceil 4 \lambda ^{1/2} n\rceil $, the family  
${\cal A}_U := \{B \in {\cal B}^*: U \subset B\}$
forms an independent set in $G^*$.
Consider a uniformly random choice of such $U$.
For any	$B \in {\cal B}^*$,
as $|B \cap T_1| \geq {\kK'|T_1|}/{2}$ 
we have ${\mathbb P}(B \in {\cal A}_U ) 
\geq (\kappa '/4)^u \ge (1-\lL^{1/3})^n$, 
as $\lL \ll \kK'$. Therefore 
${\mathbb E}_{U} |{\cal A}_U| 
= \sum _{B \in {\cal B}^*} {\mathbb P}(B \in {\cal A}_U)					
\geq (1-\eps )^n|{\cal X}|$.
Thus for some $U$ we obtain an independent set ${\cal A}_U$
of at least this size, which completes the proof of Case 1.

	\vspace{1mm}
	\noindent \textbf{Case 2:} $|T_0| \geq |T|/4$.
	\vspace{1mm}

The proof of this case is similar to that of Case 1,
so we just outline the differences. Now we let $G^*=G[B^*]$, where
${\cal B}^* = \{ B \in {\cal B}: |T_0 \sm B| \geq \kK'|T_0|/2 \}$.
Similarly to Case 1, we have
$|{\cal B}^*| \geq |{\cal B}|/2$,
and there is no edge $AB \in E(G^*)$ 
with $|T_0 \sm (A\cup B)| \geq 4\lambda^{1/2}n$.
Thus for any $U \sub T_0$ with $|U|=u$, the family 
${\cal A}_U := 	\{B \in {\cal B}^*: U \cap B = \es \}$
is an independent set in $G^*$.
Consider a uniformly random choice of such $U$.
For any	$B \in {\cal B}^*$,
as $|T_0 \sm B| \geq \kK'|T_0|/2$ 
we have ${\mathbb P}(B \in {\cal A}_U ) 
\geq (\kK'/4)^u \ge (1-\lL^{1/3})^n$, 
as $\lL \ll \kK'$. Therefore 
for some $U$ we obtain an independent set ${\cal A}_U$
with size at least the expectation,
which is at least $(1-\eps )^n|{\cal X}|$. \qed

\section{Exponential continuity} \label{sec:cts}

In this section we recast our results 
using the following notion of continuity
that arises naturally when comparing distributions
according to exponential contiguity.

\begin{dfn} \label{expcts}
Let $\OO = (\OO_n)_{n \in \mb{N}}$ and $\mu=(\mu_n)_{n \in \mb{N}}$,
where each $\mu_n$ is a probability measure on $\OO_n$.
Let $\mc{F} = (\mc{F}_n)_{n \in \mb{N}}$ where each 
$\mc{F}_n$ is a set of measurable subsets of $\OO_n$.
We say that $\mc{B}=(\OO,\mc{F})$ 
is an exponential probability space
and write $\mc{M}(\mc{B})=\mc{M}(\OO)$
for the set of such $\mu$.
We write $\nu \approx \mu$
when $\nu \approx_{\mc{F}} \mu$.
Given exponential probability spaces
$\mc{B}=(\OO,\mc{F})$, $\mc{B}'=(\OO',\mc{F}')$
we say that $f: \mc{M}(\OO) \to \mc{M}(\OO')$
is exponentially continuous at $\mu \in \mc{M}(\OO)$
if $\mu' \approx \mu \Ra f(\mu') \approx f(\mu)$.
\end{dfn}


\begin{theo}  \label{general+ctg}
Let $0 < n^{-1} \ll \zZ \ll \kK, \gG \ll D^{-1}, M^{-1}, C^{-1}, k^{-1}$.
Suppose
\begin{enumerate}
\item $\mc{B}_s=(\OO_s,\mc{F}_s)$ are exponential probability spaces 
with $\OO_{s,n} = J_s^n$ for $s \in S$,
\item $\mu_{\vc{q}}$ is $\kK$-bounded product measure on $\OO_n$,
\item $\mc{V} = (\vc{v}^i_{j_1,\dots,j_S})$ is an $(n,\prod_{s \in S} J_s)$-array in $\mb{Z}^D$,
\item all $\|\vc{v}^i_{j_1,\dots,j_S}\|_{\vc{R}} \le 1$, where $\vc{R}=(R_1,\dots,R_D)$ with $\max_d R_d < n^C$,
\item $\mc{U}=\{\vc{u}_1,\dots,\vc{u}_M\} \sub \mb{Z}^D$ is $\vc{R}$-bounded and $(k,k\zZ n,\vc{R})$-generating, 
\item $\mc{V}$ has $\gG$-robust transfers for $\mc{U}$,
\item $\vc{w} \in \mb{Z}^D$ with $\|\vc{w}-\mc{V}(\mu_{\bf q})\|_{\vc{R}} < \zZ n$.
\end{enumerate}
Let $\mc{B} = (\OO,\mc{F}) = \prod_{s \in S} \mc{B}_s$ and $\mc{B}' = (\OO',\mc{F}')$,
where $\OO'_n = (\OO_n)^{\mc{V}}_{\vc{w}}$ and $\mc{F}'_n = \{ \mc{A} \cap \OO'_n: \mc{A} \in \mc{F} \}$.
Let $f$ be restriction of measure from $\mc{M}(\OO)$ to $\mc{M}(\OO')$.
Then $f$ is exponentially continuous at $\mu_{\vc{q}}$.
\end{theo}

\nib{Proof.}
Let $\mu_{\vc{q}}$ have marginals $(\mu_{\vc{p}_s}: s \in S)$
and suppose $\mu_{\vc{q}'} \approx \mu_{\vc{q}}$ with marginals $(\mu_{\vc{p}'_s}: s \in S)$.
Suppose $n^{-1} \ll \dD \ll \zZ \ll \eps \ll \kK, \gG$.
We want to show for  $\mc{A} = \prod_{s \in S} \mc{A}_s \in \mc{F}'_n$ 
that $f(\mu_{\vc{q}})(\mc{A}) > (1-\dD)^n \Ra f(\mu_{\vc{q}'})(\mc{A}) > (1-\eps)^n$
and $f(\mu_{\vc{q}'})(\mc{A}) > (1-\dD)^n \Ra f(\mu_{\vc{q}})(\mc{A}) > (1-\eps)^n$.
As $f(\mu_{\vc{q}})(\mc{A}) = \mu_{\vc{q}}(\mc{A})/\mu_{\vc{q}}(\OO'_n)$
and $f(\mu_{\vc{q}'})(\mc{A}) = \mu_{\vc{q}'}(\mc{A})/\mu_{\vc{q}'}(\OO'_n)$,
it suffices to show that $\mu_{\vc{q}}(\OO'_n), \mu_{\vc{q}'}(\OO'_n) > (1-\eps')^n$
with $\eps' \ll \eps$. This holds for $\mu_{\vc{q}}(\OO'_n)$ by Theorem \ref{general},
and so for $\mu_{\vc{q}'}$ by exponential contiguity. \qed

\medskip

\nib{Remark.} In the setting of the above theorem,
if $\mu_{\vc{q}}$ has marginals $(\mu_{\vc{p}_s}: s \in S)$
$\mu_{\vc{q}'}$ has marginals $(\mu_{\vc{p}'_s}: s \in S)$,
and each $\mc{F}_{s,n}$ is the set of subsets 
of some $\DD_{s,n} \sub \OO_{s,n}$,
then we have $\mu_{\vc{q}} \approx \mu_{\vc{q}'}$
precisely when each $\mu_{\vc{p}_s} \approx \mu_{\vc{p}'_s}$:
this holds by Theorem \ref{corctg} and the following lemma.

\begin{lemma} \label{ctgmarg}
Suppose $\mu=(\mu_n)_{n \in \mb{N}}$ and $\nu=(\nu_n)_{n \in \mb{N}}$
where each $\mu_n$ and $\nu_n$ is a probability measure on $\OO_n$.
Suppose also $\mu'=(\mu'_n)_{n \in \mb{N}}$ and $\nu'=(\nu'_n)_{n \in \mb{N}}$
where each $\mu'_n$ and $\nu'_n$ is a probability measure on $\OO'_n$.
Let $\DD = (\DD_n)_{n \in \mb{N}}$ with each $\DD_n \sub \OO_n$
and $\DD' = (\DD'_n)_{n \in \mb{N}}$ with each $\DD'_n \sub \OO'_n$.
Then $\mu \times \mu' \approx_{\DD \times \DD'} \nu \times \nu'$
if and only if $\mu \approx_\DD \nu$ and $\mu' \approx_{\DD'} \nu'$.
\end{lemma}

\nib{Proof.}
Let $n^{-1} \ll \dD \ll \eps$. Suppose first that
$\mu \times \mu' \approx_{\DD \times \DD'} \nu \times \nu'$.
Consider $A^1_n \sub \DD_n$ with $\mu_n(A^1_n) > (1-\dD)^n$.
Let $A_n = A^1_n \times \DD'_n$. 
Then $(\mu_n \times \mu'_n)(A_n) = \mu_n(A^1_n) > (1-\dD)^n$,
so $\nu_n(A^1_n) = (\nu_n \times \nu'_n)(A_n) > (1-\eps)^n$ by
assumption, i.e.\ $\mu \lesssim_\DD \nu$.
Similarly $\nu \lesssim_\DD \mu$, so $\mu \approx_\DD \nu$,
and similarly $\mu' \approx_{\DD'} \nu'$.
Now suppose $\mu \approx_\DD \nu$ and $\mu' \approx_{\DD'} \nu'$.
Let $B_n = \{ (\vc{x},\vc{y}) \in \DD_n \times \DD'_n:
 (\nu_n \times \nu'_n)(\vc{x},\vc{y}) 
 < (1-\eps)^n  (\mu_n \times \mu'_n)(\vc{x},\vc{y})\}$.
We have $B_n \sub (B^1_n \times \DD'_n) \cup (\DD_n \times B^2_n)$,
where $B^1_n = \{ \vc{x} \in \DD_n: \nu_n(\vc{x}) 
 < (1-\eps)^{n/2} \mu_n(\vc{x}) \}$
and $B^2_n = \{ \vc{y} \in \DD'_n: \nu'_n(\vc{y}) 
 < (1-\eps)^{n/2} \mu'_n(\vc{y}) \}$.
By assumption, $\mu_n(B^1_n) \le (1-2\dD)^n$
and $\mu'_n(B^2_n) \le (1-2\dD)^n$.
Therefore $(\mu_n \times \mu'_n)(B_n) \le 2(1-2\dD)^n < (1-\dD)^n$,
i.e.\ $\mu \times \mu' \lesssim_{\DD \times \DD'} \nu \times \nu'$.
Similarly, $\nu \times \nu' \lesssim_{\DD \times \DD'} \mu \times \mu'$,
so $\mu \times \mu' \approx_{\DD \times \DD'} \nu \times \nu'$. \qed

\section{Concluding remarks}

There are several natural directions in which
to explore potential generalisations of our results:
instead of associating vectors in $\mb{Z}^D$ to each coordinate
we may consider values in another (abelian) group $G$,
and we may consider more general functions of the coordinate values,
e.g.\ a (low degree) polynomial (e.g.\ a quadratic for application
to the Borsuk conjecture) rather than a linear function
(is there a `local' version of Kim-Vu \cite{kv} polynomial concentration?).
Even for linear functions in one dimension, our setting seems somewhat 
related to some open problems in Additive Combinatorics, 
such as the independence number of Paley graphs,
but here our assumptions seem too restrictive (one cannot use transfers).
We may also ask when better bounds hold, e.g.\
for $G=\mb{Z}/6\mb{Z}$ we recall an open problem of Grolmusz \cite{G}:
is there a subexponential bound for set systems 
where the size of each set is divisible by $6$ 
but each pairwise intersection is not divisible by $6$?

Our results may interpreted as giving robust statistics in the theory of social choice.
Suppose that we represent a voter by an opinion vector $\vc{x} \in J^n$,
where each $x_i$ represents an opinion on the $i$th issue, for example, 
when $|J|=2$ each issue could be a question with a yes/no answer.
Then we can represent a population of voters by a probability measure $\mu$ on $J^n$,
where $\mu(\vc{x})$ is the proportion of a voters with opinion $\vc{x}$.
Now suppose that we want to compare two (or more) voters.
One natural measure of comparison is to assign a score to each opinion
and calculate the total score on opinions where they agree.
If this is too simplistic, then we could assign score vectors in some $\mb{R}^D$,
where $D$ is small enough to give a genuine compression of the data,
but large enough to capture the varied nature of the issues:
we compare $\vc{x}$ and $\vc{x}'$ according to $\mc{V}_\cap(\vc{x},\vc{x}')$.
Taking the perspective of robust statistics (see \cite{HR}),
it is natural to ask whether this statistic
is sensitive to our uncertainty in the probability measure
that represents the population as a whole: 
Theorem \ref{general+ctg} (with the remark following it)
gives one possible answer.


\begin{thebibliography}{99}

\bibitem{aands} N. Alon and J. Spencer,
\textit{The Probabilistic Method}, Wiley, 2008.

\bibitem{Babai-Frankl} L. Babai and P. Frankl,  \textit{Linear algebra methods in combinatorics}, 
Department of Computer Science, University of Chicago, preliminary version, September 1992.


\bibitem{BH} A. Barvinok and J. Hartigan, Maximum entropy Gaussian approximations for the number of integer points and volumes of polytopes, \textit{Adv. Appl. Math.} \textbf{45} (2010), 252-289.

\bibitem{com} B. Bollob\'as, \textit{Combinatorics}, Cambridge University Press, 1986.

\bibitem{MGT} B. Bollob\'as, \textit{Modern Graph Theory}, Springer, 1998.

\bibitem{BCW} H. Buhrman, R. Cleve and A. Wigderson, Quantum vs. classical communication and computation, 
\textit{Proceedings of 30th STOC} (1998), 63--68.


\bibitem{covthom} T.M. Cover and J.A. Thomas, \textit{Elements of Information Theory}, 
Wiley Series in Telecommunications and Signal Processing, 2006.

\bibitem{DZ} A. Dembo and O. Zeitouni, \textit{Large deviations techniques and applications}, Springer, 2009.


\bibitem{Erdos-Ko-Rado} P. Erd\H{o}s, C. Ko and R. Rado, Intersection theorems for systems of finite sets,
\textit{Quart. J. Math. Oxford Ser. (2)} \textbf{12} (1961), 313--320.

\bibitem{FS} J. Fox and B. Sudakov, Dependent random choice, 
\textit{Random Structures Algorithms} \textbf{38} (2011), 68--99.


\bibitem{FrRo} P. Frankl and V. R\"odl, Forbidden intersections, 
\textit{Trans. Amer. Math. Soc.} \textbf{300} (1987), 259--286.

\bibitem{FRGeomRams} P. Frankl and V. R\"odl, A partition property of simplices in Euclidean space, 
\textit{J. Amer. Math. Soc.} \textbf{3} (1990), 1--7.

\bibitem{FrWil} P. Frankl and R.M. Wilson, Intersection theorems with geometric consequences, 
\textit{Combinatorica} \textbf{1} (1981), 357--368.

\bibitem{GN} D. Gale and H. Nikaido, The Jacobian matrix and global univalence of mappings, {\em Math. Ann.} \textbf{159} (1965), 81-93.

\bibitem{G} V. Grolmusz, 
Superpolynomial size set-systems with restricted 
intersections mod $6$ and explicit Ramsey graphs,
{\em Combinatorica} {\bf 20} (2000), 71--85. 

\bibitem{HR} P.J. Huber and E.M. Ronchetti,
\textit {Robust statistics}, Wiley, 2009.

\bibitem{JLR} S.~Janson, T.~\L uczak and A.~Ruci\'nski, {\em Random graphs},
Wiley-Interscience, 2000.

\bibitem{jaynes} E.T. Jaynes, Information Theory and Statistical Mechanics,
\textit{Physical Review Series II} \textbf{106} (1957), 620--630. 

\bibitem{jukna} S. Jukna, \textit{Extremal combinatorics}, Springer, 2011.

\bibitem{KahnKalai} J. Kahn and G. Kalai, A counterexample to Borsuk's conjecture, 
\emph{Bull. Amer. Math. Soc.} {\bf 29} (1993), 60--62.

\bibitem{Kalai} G. Kalai, Some old and new problems in combinatorial geometry I: Around Borsuk's problem, 
in \emph{Surveys in Combinatorics} 2005, 147--174, London Math. Soc. Lecture Note Ser., 424, Cambridge Univ. Press, 2015.

\bibitem{Katona} G.O.H. Katona, Intersection theorems for systems of finite sets, \emph{Acta Math. Acad. Sci. Hung.} \textbf{15} (1964), 329--337.

\bibitem{kl} P. Keevash and E. Long, Frankl--R\"odl type theorems for codes and permutations, \textit{Trans. Amer. Math. Soc.} {\bf 369} (2017), 1147--1162.

\bibitem{kv} J.H. Kim and V. Vu, Concentration of multivariate polynomials and applications, 
\textit{Combinatorica} {\bf 20} (2000), 417--434.

\bibitem{McD} C. McDiarmid, Concentration, in: Probabilistic Methods for Algorithmic Discrete Mathematics,
{\em Alg. Combin.} 16:195--248 (1998).


\bibitem{DHJ} D.H.J. Polymath, A new proof of the density Hales-Jewett theorem, 
\textit{Ann. of Math.} \textbf{175} (2012), 1283--1327.

\bibitem{Sa} N. Sauer,
On the density of families of sets,
{\em J. Combin. Theory Ser. A} {\bf 13} (1972), 145--147.

\bibitem{Sgall} J. Sgall, Bounds on pairs of families with restricted intersections, \textit{Combinatorica} \textbf{19}(4), (1999), 555-566.

\bibitem{Sh} S. Shelah,
A combinatorial problem; stability and order for models and
theories in infinitary languages,
{\em Pacific J. Math.} {\bf 41} (1972), 247--261.

\bibitem{T41} P. Tur\'an,
On an extremal problem in graph theory (in Hungarian),
{\em Mat. Fiz. Lapok} {\bf 48} (1941), 436--452.

\bibitem{VC} V.N. Vapnik and A.Ya. Chervonenkis,
On the uniform convergence of relative frequencies of
events to their probabilities, {\em Theory Probab. Appl.}
{\bf 16} (1971), 264--280.

\end{thebibliography}
\end{document}